\documentclass{amsart}

\usepackage[bookmarksnumbered,colorlinks, linkcolor=blue, citecolor=red, pagebackref, bookmarks, breaklinks]{hyperref}

\usepackage[hyperpageref]{backref}
\usepackage{amsfonts,color}
\usepackage{latexsym,amssymb}
\usepackage{amsmath}
\usepackage[utf8]{inputenc}

\usepackage{verbatim} 

\newtheorem{theorem}{Theorem}[section]

\newtheorem{lemma}{Lemma}[section]
\newtheorem{corollary}{Corollary}[section]
\newtheorem{remark}{Remark}[section]
\newtheorem{proposition}{Proposition}[section]

\newtheorem{Step}{Step}[lemma]
\newtheorem{case}{Case}[lemma]

\usepackage{verbatim} 

\newcommand*{\avint}{\mathop{\ooalign{$\int_{\mathbb{S}^n}$\cr$-$}}}

\input cyracc.def

\begin{document}
\sloppy 

\numberwithin{equation}{section}

\title[UNIQUENESS FOR THE BREZIS-NIRENBERG TYPE PROBLEMS]{UNIQUENESS FOR THE BREZIS-NIRENBERG TYPE PROBLEMS ON SPHERES AND HEMISPHERES}
\author{Emerson Abreu}
\address{Universidade Federal de Minas Gerais (UFMG), Departamento de Matem\'{a}tica, Caixa Postal 702, 30123-970, Belo Horizonte, MG, Brazil}
\email{eabreu@ufmg.br}
\author{Ezequiel Barbosa}
\address{Universidade Federal de Minas Gerais (UFMG), Departamento de Matem\'{a}tica, Caixa Postal 702, 30123-970, Belo Horizonte, MG, Brazil}
\email{ezequiel@mat.ufmg.br}
\author{Joel Cruz Ramirez}
\address{Universidade Federal de Minas Gerais (UFMG), Departamento de Matem\'{a}tica, Caixa Postal 702, 30123-970, Belo Horizonte, MG, Brazil}
\email{jols.math@gmail.com}
\thanks{This study was financed by the Brazilian agencies: Conselho Nacional de Desenvolvimento Cient\'{i}fico e Tecnol\'{o}gico (CNPq), Coordena\c c\~ao de
Aperfei\c coamento de Pessoal de N\'{i}vel Superior (CAPES), and Funda\c c\~ao de Amparo \`a Pesquisa do Estado de Minas Gerais (FAPEMIG)}
\date{\today}

\maketitle

\allowdisplaybreaks

\begin{abstract}
In this work, we develop a study involving some nonlinear partial differential equations on spheres and hemispheres, with the zero Neumann boundary condition, which are so-called Brezis-Nirenberg type problems, and we give conditions on which such equations have only constant solutions. We also extend these results for some nonlinear partial differential systems. 
\end{abstract}

\section{Introduction and main results}

Let $(M^n,g)$, $n\geq3$, be a compact Riemannian manifold (possibly with non-empty boundary). We consider the following problem 
\begin{equation*}\label{P}
\left\{ \begin{aligned}
     -L_g u=&\ f(u) &\mbox{on}\ \ \ M, \\
     u>&\ 0 &\mbox{on}\ \ \ M, \\
     \frac{\partial u}{\partial \nu}=&\ 0 & \mbox{on}\ \ \ \partial M,
    \end{aligned} \right.   \tag{P} 
    \end{equation*}
where $L_g$ is a second order partial differential operator on $M^n$ with respect to the metric $g$, and $\frac{\partial u}{\partial \nu}$ is the normal derivative of $u$ with respect to the unit exterior normal vector field $\nu$ of the boundary $\partial M$, and $f:(0,\infty) \rightarrow\mathbb{R}$ is a smooth function. The problem will be assumed without boundary conditions if the boundary of $M$ is empty.  Our main interest here is to find conditions on $f$, and on the geometry, or on the topology of $M$ that imply the existence of only constant solutions to the problem \eqref{P}. A particular case of this problem is the following:
\begin{equation*}\label{Q} 
\left\{ \begin{aligned}
     -\Delta_g u+\lambda u=&\ F(u)u^{\frac{n+2}{n-2}} &\mbox{on}\ \ \ M, \\
     u>&\ 0 &\mbox{on}\ \ \ M, \\
     \frac{\partial u}{\partial \nu}=&\ 0 & \mbox{on}\ \ \ \partial M,
    \end{aligned} \right.   \tag{Q}
    \end{equation*}
where $\Delta_g$ is the Laplace-Beltrami operator on $M^n$ with respect to the metric $g$, $\lambda$ is a real smooth function on $M$ and $F:(0,\infty) \rightarrow\mathbb{R}$ is a real smooth function. Note that when $\lambda>0$ is a constant and $F(u)u^{\frac{n+2}{n-2}}=u^p$, $p>1$, then $u=\lambda^{\frac{1}{p-1}}$ is a solution of the problem \eqref{Q}.
In the case where $F$  is a constant and $\lambda=\frac{(n-2)}{4(n-1)}R_g$, where $R_g$ denotes the scalar curvature of the Riemannian manifold $(M,g)$, which is just the Yamabe problem in the conformal geometry for the closed case. Otherwise, when $\partial M$ is not empty, we are in the case of minimal boundary. See \cite{JE} for more information on this subject. When $M^n=\mathbb{S}^n$ is the standard unit $n$-sphere and $g$ is the standard metric, then there are infinitely many solutions to the Yamabe problem with respect to the metric $g$ since the conformal group of the standard unit $n$-sphere is also infinite, see \cite{Ob}. If $(M^n,g)$ is an Einstein manifold which is conformally distinct from the standard $n$-sphere, it was shown by Obata that the Yamabe problem has a unique solution. However, R. Schoen has proved that there are at least three solutions for the Yamabe problem on $\mathbb{S}^1\times \mathbb{S}^{n-1}$ with respect to the standard product metric. In this direction and a good bibliography of literature, see the survey \cite{BMa}, to this and other related questions.

In a more specific situation, the problem \eqref{Q} were studied  by  Lin,  Ni and Takagi for the case when $f(u)u^{\frac{n+2}{n-2}}=u^p$, $p>1$, $\lambda>0$ is a constant function and $M$ is a bounded convex domain with smooth boundary in the Euclidean space $\mathbb{R}^n$ (see  \cite{17}, \cite{19} and references therein). When $p$ is a subcritical exponent, that is, $p<\frac{n+2}{n-2}$, it was shown in \cite{17}  that exists a unique  solution  if $\lambda$ is  sufficiently small. This kind of uniqueness of radially symmetric solution of \eqref{P} were also obtained by Lin and Ni in \cite{16}, when $\Omega$ is an annulus and $p>1$, or when $\Omega$ is a ball,  and $p>\frac{n+2}{n-2}$. They conjectured that, for any $p>1$ there exists a $\bar{\lambda}>0$ such that the problem \eqref{P} has only the constant solution, for $0<\lambda<\bar{\lambda}$. However, this conjecture is not true in general, especially when $p$ is  the critical Sobolev exponent, that  is, $p=\frac{n+2}{n-2}$. In this case,  when $\Omega$ is a unit ball and $n=4, 5, 6$, it was shown by Adimurthi and Yadava  \cite{2} that  the problem \eqref{P} has at least two radial solutions if $\lambda>0$ is close enough to 0 (see also \cite{5}). On the other hand, if $\Omega$ is not a ball or $n$ is different from 4, 5 or 6, that conjecture is open. 

In \cite{BL}, Brezis and Li studied the problem \eqref{P} for the case of the standard unit sphere $(\mathbb{S}^n,g)$, where $g$ is the standard metric, and so, by using results of symmetry due to Gidas, Ni and Nirenberg \cite{GNN}, they showed that this problem has only constant solutions provided that $L_g=\Delta_g$, and $f$ is a function with the property that $h(t)=t^{-\frac{n+2}{n-2}}(f(t)+\frac{n(n-2)}{4}t)$ is a decreasing function on $(0,+\infty)$. As a consequence of this result, they proved that the problem \eqref{Q} has only constant solutions if $0<\lambda<\frac{n(n-2)}{4}$, and $F$ is a decreasing function on $(0,+\infty)$.

Motivated by these results, we will carry out a study on the nonexistence of positive solutions for a wide class of problems, which include equations, systems, and some types of operators, on spheres and hemispheres. Our first attempt is the following nonlinear elliptic equations and systems:
\begin{equation}\label{p1.1}
   \left\{ \begin{aligned}
    -L^1_{g} u & =  f(u)  & \text{ on } \mathbb{S}^n_+, \\
    \frac{\partial u}{\partial \nu}  &=   b(u) & \text{ on } \partial \mathbb{S}^n_+,
    \end{aligned}
    \right.
    \end{equation}
    and
    \begin{equation} \label{p1.2}
   \left\{ \begin{aligned}
  &  -L^1_{g} u_1 =  f_1(u_1,u_2) & \text{ on } \mathbb{S}^n_+, \\
   & -L^1_{g} u_2 =  f_2(u_1,u_2)& \text{ on } \mathbb{S}^n_+,\\
  & \frac{\partial u_1}{\partial \nu} =b_1(u_1,u_2),~  \frac{\partial u_2}{\partial \nu}  =  b_2(u_1,u_2)  & \text{ on } \partial \mathbb{S}^n_+,
    \end{aligned}
    \right.
    \end{equation}
    where $L^1_g = \Delta_g$, and $g$ is a metric in the conformal class of the standard metric $g_{\mathbb{S}^n_+}$, $n>2$,  and $f:[0,+\infty)\rightarrow \mathbb{R}$, $f_1$, $f_2:[0,+\infty)\times [0,+\infty)\rightarrow \mathbb{R}$ are continuous functions. The notation $L^1_g$ will make sense in Section \ref{sec4}.

Thus we have the following result:

\begin{theorem}
    \label{th1.1} 
 Let $g$ be the standard metric on $\mathbb{S}^n_+$. Suppose that
    \begin{equation}
    \label{e1.1}
   \left\{ \begin{aligned}
   & h_1(t):=t^{-\frac{n+2}{n-2}}\left(f(t)+\frac{n(n-2)}{4}t\right)~ \text{ is decreasing on } ~ (0,+\infty), \text{ and }\\
    & k_1(t):=t^{-\frac{n}{n-2}}b(t) \text{ is nonincreasing nonnegative on } ~ (0,+\infty).
    \end{aligned}\right.
\end{equation}      
Then
\begin{enumerate}
\item[$(i)$] any continuous positive solution $u$ on $\overline{\mathbb{S}^n_+}$ of \eqref{p1.1} is constant in each circle parallel to the equator. Moreover, if $b(u(\zeta_0))=0$ for some $\zeta_0\in \partial\mathbb{S}^n_+$, then $u$ is constant in $\mathbb{S}^n_+$;
\item[$(ii)$] if $h_1(t)t^{\frac{n+2}{n-2}}$ is nonnegative in $t\geq 0$, then any nonnegative continuous solution of \eqref{p1.1} is zero or positive on $\overline{\mathbb{S}^n_+}$, and so, it follows the same conclusion of (i). Moreover, if $b(t)\geq \frac{n-2}{2}t$ for $t>0$, then there are no positive solutions of \eqref{p1.1}.
\end{enumerate}
    \end{theorem}
    
    We can extend Theorem \ref{th1.1} for systems. Define for $j=1,2$:
    \begin{equation*}
    \label{e1.3}
    \begin{aligned}
    h_{j1}(t_1,t_2) & :=t_j^{-\frac{n+2}{n-2}}\left(f_j(t_1,t_2)+\frac{n(n-2)}{4}t_j\right),\\
    k_{j1}(t_1,t_2) & := t_j^{-\frac{n}{n-2}}b_j(t_1,t_2), ~t_j>0.\\
    \end{aligned}
    \end{equation*}
\begin{theorem}
	\label{th1.2} 
	 Let $g$ be the standard metric on $\mathbb{S}^n_+$. Suppose that for $j,l=1,2$:
	\begin{equation}
	\label{e1.3.1}
	\left\{\begin{aligned}
	 & h_{j1}(t_1,t_2), ~k_{j1}(t_1,t_2) \text { are nondecreasing in }  t_l>0, \text{ with } j\neq l,& \\
	 & h_{j1}(t_1,t_2)t_j^{\frac{n+2}{n-2}}, ~b_{j1}(t_1,t_2) \text{ are nondecreasing in } t_j>0,\text{ and }\\
	  & h_{j1}(a_1t,a_2t), ~k_{j1}(a_1t,a_2t) \text{ are decreasing in } t>0 \text{ for any } a_j>0.& 
	\end{aligned}\right.
	\end{equation}
	Then
\begin{enumerate}
	\item[$(i)$] for any pair $(u_1,u_2)$ of positive continuous solutions on $\overline{\mathbb{S}^n_+}$ of \eqref{p1.2},  $u_1$ and $u_2$ are constant in each circle parallel to the equator. Moreover, if $b_1(u_1(\zeta_1),u_2(\zeta_2))= b_2(u_1(\zeta_3),u_2(\zeta_4)) =0$ for some $\zeta_1,\zeta_2,\zeta_3,\zeta_4 \in \partial \mathbb{S}^n_+$, then $u_1$ and $u_2$ are constant;
	\item[$(ii)$] for any pair $(u_1,u_2)$ of nonnegative continuous solutions of \eqref{p1.2}, $u_j$ is zero or positive, and so, it follows the same conclusion of (i). Moreover, if $b_j(t_1,t_2)\geq \frac{n-2}{2}t_j$ for some $j=1,2$, then there are no positive solutions of \eqref{p1.2}.
	\end{enumerate}
	\end{theorem}

To get a better understanding of the above results, let us consider the following problems:
\begin{equation}
\label{pe1.5}
	\left\{\begin{aligned}
	-\Delta_{g} u &= u^{p} - \lambda u && \text{on } {\mathbb{S}^n_+},\\
	\frac{\partial u}{\partial\nu} &=0				&& \text{on } {\partial\mathbb{S}^n_+},
	\end{aligned}\right.
	\end{equation}
	 where $p>1$, $\lambda>0$, and
	\begin{equation}
	\label{pe1.6}
	\left\{\begin{aligned}
	-\Delta_{g} u_1 &=u_1^{p_1}-\lambda_1 u_1 -\lambda_2 u_2, && \text{on } {\mathbb{S}^n_+},\\
	-\Delta_{g} u_2 &= u_2^{p_2}-\lambda_3 u_1 - \lambda_4 u_2 && \text{on } {\mathbb{S}^n_+},\\
	\frac{\partial u_1}{\partial\nu} &= \frac{\partial u_2}{\partial\nu}  = 0				&& \text{on } {\partial\mathbb{S}^n_+},
	\end{aligned}\right.
	\end{equation}
	 for $\lambda_1,-\lambda_2,-\lambda_3,\lambda_4 $ being nonnegative values, and $p_i>1$ for $i=1,2$. The above problems are related to Brezis-Nirenberg type problems. Then, in the spirit of Lin-Ni's conjecture, we obtain the following uniqueness results:
	\begin{corollary}
	\label{c1.5}
	 Let $g$ be the standard metric on $\mathbb{S}^n_+$. Assume that $p\leq \frac{n+2}{n-2}$ and $\lambda\leq \frac{n(n-2)}{4}$, and at least one of these inequalities is strict. Then the only positive solution of \eqref{pe1.5} is the constant $u\equiv\lambda^{\frac{1}{p-1}}$. 
	\end{corollary}

A similar result on a smooth domain in the Euclidean space with zero
Neumann boundary data was established in \cite{Zhu}.
	
	\begin{corollary}
 	\label{c1.6}
 	 Let $g$ be the standard metric on $\mathbb{S}^n_+$. Assume that $\lambda_1,-\lambda_2,-\lambda_3,\lambda_4 $ are nonnegative values,
 	\begin{enumerate}
 	\item[$(i)$] $p_1\leq \frac{n+2}{n-2}$, $\lambda_1\leq \frac{n(n-2)}{4}$,
 	\item[$(ii)$] $p_2\leq \frac{n+2}{n-2}$, $ \lambda_4\leq \frac{n(n-2)}{4}$,
 	\end{enumerate}
  and in each above item at least one of these inequalities is strict. Then for any pair $(u_1,u_2)$ of positive solutions  of \eqref{pe1.6},  $u_1$ and $u_2$ are constant.
 	\end{corollary}

	In order to generalize the results of Brezis and Li, let us recall a little bit on conformal operators. In \cite{GJM}, Graham, Jenne, Mason and Sparling constructed a sequence of conformally covariant elliptic operators $P_k^g$, on Riemannian manifolds $(M,g)$ for all positive integers $k$ if $n$ is odd, and for $k \in \{1,.., n/2\}$ if $n$ is even. Moreover, $P_1^g$ is the well known conformal Laplacian $L^1_g := - \Delta_g + d_{n,1}R_g$, where $d_{n,1} =  (n-2)/[4(n-1)]$, $n \geq 3$,  and $P_2^g$ is the Paneitz operator.
	
	After that, by using a generalized Dirichlet to Neumann map, Graham and Zworski \cite{GZ} introduced a meromorphic family of conformally invariant operators on the conformal infinity of asymptotically hyperbolic manifolds, see Mazzeo and Melrose \cite{MM}. Chang and Gonz\'alez \cite{CG} established a connection with the method of Graham and Zworski, and the localization method of Caffarelli and Silvestre \cite{CS}, for fractional Laplacian $(-\Delta)^s$ on the Euclidean space  $\mathbb{R}^n$, and used it to define conformally invariant operators $P_s^g$ of non-integer order $s\in(0,n/2)$. These lead naturally to a fractional order curvature $R_s^g=P_s^g(1)$, which is called $s$-curvature. There are several works on these conformally invariant equations of fractional order and prescribing $s$-curvature problems, which are related to fractional Yamabe and Nirenberg problems, see, e.g., \cite{ GQ, JLX1, JLX2, JX} and references therein.
	
	Let $({\mathbb{S}^n},g_{\mathbb{S}^n})$ be the $n$-sphere, $n\geq 3$ and $g_{\mathbb{S}^n}$ the standard metric. The operator $P_s:=P_s^{g_{\mathbb{S}^n}}$ has the formula (see \cite{Bra})
	\begin{equation*}
	P_s=\frac{\Gamma\left(B + \frac{1}{2}+s\right)}{\Gamma\left(B + \frac{1}{2}-s\right)},~ B=\sqrt{-\Delta_{g_{\mathbb{S}^n}} +\left(\frac{n-1}{2}\right)^2},\mbox{ for } s\notin \frac{n}{2}+\mathbb{N},
	\end{equation*}
	where $\Gamma$ is the Gamma function.  Moreover, from \cite{CG, GQ},
	\begin{equation}
	\label{r1.1}
	 P^{\varphi^{\frac{4}{n-2s}}g_{\mathbb{S}^n}}_s(u)=\varphi^{-\frac{n+2s}{n-2s}}P_s(\varphi u) \text{ for all } \varphi, u\in C^{\infty}(\mathbb{S}^n), ~\varphi>0.
	\end{equation}
The Green function of $P_s^{-1}$ is the spherical Riesz potential,
\begin{equation}
(P_s^{-1} \phi)(\zeta)=c_{n,s}\int_{\mathbb{S}^n}\frac{\phi(\omega)}{|\zeta-\omega|^{n-2s}}d\upsilon_{g_{\mathbb{S}^n}}^{(\omega)} \text{ for } \phi\in L^p(\mathbb{S}^n),
\end{equation}
where $c_{n,s}=\frac{\Gamma(\frac{n}{2}-s)}{\pi^{\frac{n}{2}}2^{2s}\Gamma(s)}$, $p>1$ and $|\cdot|$ is the Euclidean distance in $\mathbb{R}^{n+1}$.

Denote by
	\begin{equation}
	\label{e1.16}
	-L^s_{g_{\mathbb{S}^n}} u =P_s u - d_{n,s}u~\text{ and }~ d_{n,s}=:\frac{\Gamma(\frac{n}{2}+s)}{\Gamma(\frac{n}{2}-s)}=P_s(1).
	\end{equation}
	In Section \ref{sec3}, we will establish uniqueness results for the following problems:
	\begin{equation}
	\label{p1.5}
	\left\{\begin{aligned}
	-L^s_{g_{\mathbb{S}^n}} u &= f(u) && \text{on } {\mathbb{S}^n},\\
	u &>0				&& \text{on } {\mathbb{S}^n},
	\end{aligned}\right.
	\end{equation}
	and
	\begin{equation}
	\label{p1.6}
	\left\{\begin{aligned}
	-L^s_{g_{\mathbb{S}^n}} u_1 &= f_1(u_1,u_2) && \text{on } {\mathbb{S}^n},\\
	-L^s_{g_{\mathbb{S}^n}} u_2 &= f_2(u_1,u_2) && \text{on } {\mathbb{S}^n},\\
	u_1, u_2 & > 0				&& \text{on } {\mathbb{S}^n},
	\end{aligned}\right.
	\end{equation}
	where $n>2s$, and $f:(0,+\infty)\rightarrow \mathbb{R}$, $f_1$, $f_2:(0,+\infty)\times (0,+\infty)\rightarrow \mathbb{R}$ are continuous functions.	Then, we have the following.

	\begin{theorem}
    \label{th1.5} 
     Assume that
    \begin{equation}
    \label{e1.19}
    h_s(t):=t^{-\frac{n+2s}{n-2s}}\left(f(t)+d_{n,s}t\right)~ \text{ is decreasing on } ~ (0,+\infty) \text{ for } s\in (0,n/2)
\end{equation} 
and 
\begin{equation}
	\label{e1.191}
    h_s(t)t^{\frac{n+2s}{n-2s}}~ \text{ is nondecreasing on } ~ (0,+\infty) \text{ if } s >1.
\end{equation} 
Any continuous solution of \eqref{p1.5} is constant. In particular,  if$f(t)=t^p-\lambda t$ with $p\leq \frac{n+2s}{n-2s}$ and $\lambda\leq d_{n,s}$, and at least one of these inequalities is strict, then the only positive solution of \eqref{p1.5} is the constant $u\equiv\lambda^{\frac{1}{p-1}}$ provided that $p>1$ and $\lambda>0$. 
    \end{theorem}

    For systems, following our approach, we define
    \begin{equation*}
    h_{is}(t_1,t_2):=t_i^{-\frac{n+2s}{n-2s}}\left(f_i(t_1,t_2)+d_{n,s} t_i\right),~ t_1>0,~t_2>0, \text{ for } i=1,2,
    \end{equation*}
	the result is:
\begin{theorem}
	\label{th1.6} 
	Let $s\in (0,n/2)$. Assume that for $i,j=1,2$: 
	\begin{equation}
	\label{e1.20}
	\left\{\begin{aligned}
	 & h_{is}(t_1,t_2) \text { is nondecreasing in }  t_j>0, \text{ with } i\neq j,& \\
	 & h_{is}(t_1,t_2)t_i^{\frac{n+2s}{n-2s}} \text{ is nondecreasing in } t_i>0,\\
	  & h_{is}(a_1t,a_2t) \text{ is decreasing in } t>0 \text{ for any } a_i>0.& 
	  \end{aligned}\right.
	\end{equation}
	Any pair $(u_1,u_2)$ of continuous solutions of \eqref{p1.6} is given by constant solutions.
	\end{theorem}

Now we consider the above problems on the conformal hemisphere $(\mathbb{S}^n_+,g)$ and sphere $(\mathbb{S}^n,g)$, where $g$ is a conformal metric in the conformal class $[g_{\mathbb{S}^n_+}]$ or $[g_{\mathbb{S}^n}]$ . We consider the following problems:
\begin{equation}
\label{p1.8}
\begin{aligned}
	-L^s_g u & = \tilde{f}(u) - \lambda u &\text{ in } M,\\
	\frac{\partial u}{\partial \nu}& =0 &\text{ on }\partial M,
	\end{aligned}
	\end{equation}	
	and
	\begin{equation}
	\label{p1.9}
	\left\{\begin{aligned}
	-L^s_g u_1 &=\tilde{f}_1(u_1,u_2)-\lambda_1 u_1 -\lambda_2 u_2 && \text{in } {M},\\
	-L^s_g u_2 &= \tilde{f}_2(u_1,u_2)-\lambda_3 u_1 -\lambda_4 u_2 && \text{in } {M},\\
	\frac{\partial u_1}{\partial \nu} & =\frac{\partial u_1}{\partial \nu} =0 				&& \text{on } \partial{M}.
	\end{aligned}\right.
	\end{equation}
where:
\begin{enumerate}
\item[-] for $M=\mathbb{S}^n_+$, we have that $n\geq 3$, $s=1$, $g$ belongs to the conformal class $[g_{\mathbb{S}^n_+}]$, $L^1_g=\Delta_g$, $R^1_g$ is the scalar curvature on $(\mathbb{S}^n_+,g)$;
\item[-] for $M=\mathbb{S}^n$, we have that $\partial M=\emptyset$, $n\geq 3$, $0<2s<n$, $g$ belongs to the conformal class $[g_{\mathbb{S}^n}]$, $L^s_g=P_s^g-R^g_s$, $P^s_g$ is a conformally invariant operator, $R^s_g$ is the $s$-scalar curvature on $(\mathbb{S}^n,g)$;
\end{enumerate}	
 $\tilde{f}:[0,+\infty)\rightarrow [0,+\infty)$ is a continuous function verifying:
 \begin{enumerate}
\item[$(H1)$] $\tilde{f}(t)t^{-\frac{n+2s}{n-2s}}$ is nonincreasing positive function in $(0,+\infty)$;
\item[$(H2)$] there exists $1<p<\frac{n+2s}{n-2s}$ such that $\lim_{t\rightarrow +\infty}\tilde{f}(t)t^{-p}=c_1>0$;
\item[$(H3)$] there exist $0<\mu\leq p-1$ and $C>0$ such that $\left| \frac{\tilde{f}(t)-\tilde{f}(r)}{t-r}\right|\leq C(t^{\mu}+t^{p-1})$ for all $0\leq r<t$;
\end{enumerate} 
$\tilde{f}_1,\tilde{f}_2:[0,+\infty)\times[0,+\infty)\rightarrow [0,+\infty)$ are continuous functions verifying for $j=1,2$:
\begin{enumerate}
\item[$(F1)$] $\tilde{f}_j(r,t)$ is nondecreasing in $r$ and $t$; 
\item[$(F2)$] ${\tilde{f}_j(r,t)}{(r+t)^{-\frac{n+2}{n-2}}}$ is nonincreasing in one variable fixing the other variable;
\item[$(F3)$] there exist $0<\mu<p-1$, $1<p<\frac{n+2}{n-2}$ such that
\begin{equation*}
\left| \frac{\tilde{f}_j(r_1,t_1)-\tilde{f}_j(r_2,t_2)}{(r_1,t_1)-(r_2,t_2)}\right|\leq C \left[ (r_1+t_1)^{\mu}+(r_1+t_1)^{p-1}\right]
\end{equation*}
for all $ r_1>r_2\geq 0$, $ t_1>t_2\geq 0$ and some $C>0$;
\item[$(F4)$] there exist continuous functions $\psi_1,\psi_2:[0,+\infty)\times[0,+\infty)\rightarrow [0,+\infty)$ such that
\begin{equation*}
\lim_{t\rightarrow +\infty}\frac{\tilde{f}_j(a_1t,a_2t)}{t^p}=\psi_j(a_1,a_2) \text{ for all } a_j\geq 0 \text{ and }
\end{equation*}
$\psi_1(a_1,a_2)+\psi_2(a_1,a_2)>0$ provided that $a_1+a_2> 0$.
\end{enumerate}

	 Recall that Brezis and Li \cite{BL}, in the spirit of Lin-Ni's conjecture, showed that there exists a $\overline{\lambda}>0$ such that (\ref{p1.8}) has only a constant solution if $(M,g)=(\mathbb{S}^n,g)$, $s=1$, $0<\lambda<\overline{\lambda}$, $\tilde{f}(t)=t^p$ and $p\leq \frac{n+2}{n-2}$. Later, Hebey \cite{EH2} showed the same result for $n\geq 4$ and  $p=\frac{n+2}{n-2}$ with positive scalar curvature. Such kind of uniqueness results of (\ref{p1.8}) on a compact manifold $M$ without boundary, with $\tilde{f}(t)=t^p$, $s=1$ and $p< \frac{n+2}{n-2}$, it was also shown for by Licois and V\'{e}ron \cite{LV}, and by Gidas and Spruck \cite{GS}.

Denoting 
\begin{equation*}
	 A=\left(\begin{matrix}
	 \lambda_1 & \lambda_2 \\
	 \lambda_3  & \lambda_4 
	 \end{matrix}\right),\ \ \ \mbox{ and }\ \ \  \|A\|\ \ \ \mbox{ its norm.}
	 \end{equation*}
Since all matrix norms are equivalent we will use the above notation indiscriminately. Our results about uniqueness in this case are:
\begin{theorem}
\label{th1.3}
Let $(M,g)=(\mathbb{S}^n_+,g)$ and $g\in[g_{\mathbb{S}^n_+}]$. Assume that $\tilde{f}_1$ satisfies (H1)-(H3) for $s=1$. Then there exists some $\lambda^* =\lambda^*(n,g,\mathbb{S}^n_+)>0$ such that:
\begin{enumerate}
\item[$(i)$] for all $0<\lambda<\lambda^*$, any positive continuous solution of \eqref{p1.8} is constant. In particular, $u \equiv\lambda^{\frac{1}{p-1}}$ is the only positive solution of \eqref{pe1.5}  for $\lambda$ small and $1<p<\frac{n+2}{n-2}$;
\item[$(ii)$] for all $0<\|A\|<\lambda^*$, any pair $(u_1,u_2)$ of positive continuous solutions of \eqref{p1.9} is given constant solutions. 
\end{enumerate}
 
\end{theorem}

Note that, compared to Corollary \ref{c1.6}, the values $\lambda_1$, $\lambda_2$, $\lambda_3$ and $\lambda_4$ given by Theorem \ref{th1.3} (ii) does not necessarily depend on the sign in the case where $\tilde{f}_1(t,\cdot)=t^{p_1}$, $\tilde{f}_2(\cdot,t)=t^{p_2}$ and $p_1$, $p_2$ are subcritical values.

	 To establish some uniqueness results of solutions for problems involving conformal operators on conformal spheres, we are going to assume throughout this paper that:
 \begin{equation}
 \label{g1.40}
 \lambda_{1,s,g}=\min_{u\in H^s(\mathbb{S}^n,g)}\frac{\int_{\mathbb{S}^n}(u-\overline{u})(P^g_s(u-\overline{u})-R^g_s(u-\overline{u}))d\upsilon_g}{\int_{\mathbb{S}^n}(u-\overline{u})^2d\upsilon_g}>0,
 \end{equation}
where $s\in(1,n/2)$, $H^s(\mathbb{S}^n, g)$ is the completion of the space of smooth functions $C^{\infty}(\mathbb{S}^n)$ under the norm $\|\cdot\|_{s,g}$ defined by
\begin{equation}
\label{g1.10}
\|u\|^2_{s,g}:=\int_{\mathbb{S}^n}uP^g_s u~d\upsilon_g, ~u\in C^{\infty}(\mathbb{S}^n),
\end{equation}
and $\overline{u}:=\avint u~d\upsilon_g$. When $s=1$, the value $\lambda_{1,1,g}$ is the first positive eigenvalue of the operator $P^g_1-R^g_1=-\Delta_g$ on $(\mathbb{S}^n,g)$. If $s\in(0,1)$, it is not difficult to check (Lemma \ref{gl6.1}) that (\ref{g1.40}) holds. In the case $g=g_{\mathbb{S}^n}$ and $0<2s<n$, the value of $\lambda_{1,s,g}$ is determined by spherical harmonics (see \cite{CM}). Motivated by the results on the uniqueness of solutions mentioned above, we have: 
	 \begin{theorem}
	 \label{th1.7}
	 Let $(M,g)=(\mathbb{S}^n,g)$, $g\in[g_{\mathbb{S}^n}]$ and $0<2s<n$. Assume that $\tilde{f}_1$ satisfies (H1)-(H3) and $\tilde{f}_1,\tilde{f}_2$ satisfy (F1)-(F4). If $R_s^g$ is positive for $s>1$, then there exists some $\lambda^* =\lambda^*(n,s,g,\mathbb{S}^n)>0$ such that:
\begin{enumerate}
\item[$(i)$]  for all $0<\lambda<\lambda^*$, any positive continuous solution of \eqref{p1.8} is constant.
\item[$(ii)$]  for all $0<\|A\|<\lambda^*$, and nonnegative values $\lambda_2$ and $\lambda_3$, any pair $(u_1,u_2)$ of positive continuous solutions of \eqref{p1.9} is given by constant solutions.
\end{enumerate}	 
	 \end{theorem}
	 
Some examples related to the Ni-Li's conjecture and Brezis-Nirenberg type problems for the fractional case are
\begin{equation}
\label{p1.13}
-L^s_g u=u^p-\lambda u \text{ in } \mathbb{S}^n,
\end{equation}
	 and
\begin{equation}
\label{p1.14}
\left\{\begin{aligned}
-L^s_g u_1=u_2^{p_1}-\lambda_1 u_1 - \lambda_2 u_2 &\text{ in } \mathbb{S}^n,\\
-L^s_g u_2=u_1^{p_2}-\lambda_3 u_1 - \lambda_4 u_2 &\text{ in } \mathbb{S}^n,
\end{aligned}\right.
\end{equation}	 
	with similar assumptions as  in the Corollaries \ref{c1.5} and \ref{c1.6} we have:
	 \begin{corollary}
	 \label{c1.7}
	 Let $g\in[g_{\mathbb{S}^n}]$ be a conformal metric on $\mathbb{S}^n$. Assume that $R_s^g$ is positive for $s>1$ and $\lambda_1,-\lambda_2,-\lambda_3,\lambda_4$ are nonnegative values.
	 \begin{enumerate}
	 \item[$(i)$] If $p<\frac{n+2s}{n-2s}$, then the only positive solution of \eqref{p1.13} is the constant $u\equiv\lambda^{\frac{1}{p-1}}$ for $\lambda$ enough small.
	 \item[$(ii)$] If $p_j< \frac{n+2s}{n-2s}$ for $j=1,2$, then any pair $(u_,u_2)$ of positive solutions of \eqref{p1.14} is given by constant solutions.
	 \end{enumerate}
	 \end{corollary}
	 Recently in \cite{ABR2}, it was shown that Corollary \ref{c1.7} (i) holds in the critical case $p=\frac{n+2s}{n-2s}$ with $s$-curvature $R^g_s$ being positive.

In the proof of Theorem \ref{th1.7} for $1<s<n/2$, we used some kind of Schauder type estimates for the solutions of \eqref{p1.8} and \eqref{p1.9}. To determine these estimates, we will show in Section \ref{sec3} that the positive solutions of these problems are also positive solutions of the problems
\begin{equation}
\label{p1.10}
u(\zeta)=c_{n,s}\int_{\mathbb{S}^n}\frac{\tilde{f}(u)(\omega)+(R^g_s(\omega)-\lambda)u(\omega)}{\mathcal{K}(\zeta,\omega)}d\upsilon_g^{(\omega)}, ~ \zeta \in \mathbb{S}^n,
\end{equation}
and
\begin{equation}
\label{p1.11}
\left\{\begin{aligned}
u_1(\zeta) & =c_{n,s}\int_{\mathbb{S}^n}\frac{\tilde{f}_1(u_1,u_2)(\omega)+(R^g_s(\omega)-\lambda_1)u_1(\omega)+\lambda_2u_2(\omega)}{\mathcal{K}(\zeta,\omega)}d\upsilon_g^{(\omega)}, ~ \zeta \in \mathbb{S}^n,\\
u_2(\zeta) & =c_{n,s}\int_{\mathbb{S}^n}\frac{\tilde{f}_2(u_1,u_2)(\omega)-\lambda_1u_1(\omega)+(R^g_s(\omega)-\lambda_4)u_2(\omega)}{\mathcal{K}(\zeta,\omega)}d\upsilon_g^{(\omega)}, ~ \zeta \in \mathbb{S}^n,
\end{aligned}\right.
\end{equation}
 provided that $R_s^g$ is positive on $\mathbb{S}^n$, where $g=\varphi^{\frac{4}{n-2s}}g_{\mathbb{S}^n}$, $\varphi\in C^{\infty}(\mathbb{S}^n)$ is a positive function, $\mathcal{K}(\zeta,\omega)=\varphi(\zeta)\varphi(\omega)|\zeta-\omega|^{n-2s}$ and $|\cdot| $ is the Euclidean distance on $\mathbb{R}^{n+1}$. On the other hand, the converse holds regardless of the sign of $R^g_s$, i.e., positive solutions of \eqref{p1.10} and \eqref{p1.11}  are also positive solutions of \eqref{p1.8} and \eqref{p1.9}.  This leads to the following result:
\begin{theorem}
\label{th1.9}
 Let $(M,g)=(\mathbb{S}^n,g)$, $g\in[g_{\mathbb{S}^n}]$ and $0<2s<n$. Then there exists some $\lambda^* =\lambda^*(n,s,g,\mathbb{S}^n)>0$ such that:
 \begin{enumerate}
 \item[$(i)$]  for all $0<\lambda<\lambda^*$, any positive continuous solution of \eqref{p1.10} is constant;
 \item[$(ii)$] for all $0<\|A\|<\lambda^*$, any pair $(u_1,u_2)$ of positive continuous solutions of \eqref{p1.11} is given by constant solutions.
\end{enumerate}  
\end{theorem}

	In the next section we prove Theorems \ref{th1.1} and \ref{th1.2}, whereas, in Section \ref{sec3} we prove Theorems \ref{th1.5} and \ref{th1.6}. To prove these theorems we use the method of moving planes along with some integral inequalities, which were originally based on the ideas of S. Terracini \cite{TS}. This type of argument has been widely used to obtain some Liouville-type theorems in fractional differential or integral equations and systems, see, e.g. \cite{ADG, CLO1, GL1, GL, XY3, XY1} and references therein. However, to obtain some Liouville-type theorems it is necessary to use the Kelvin transformation, but in our case, the use of this method is not useful. So, the symmetry of $\mathbb{S}^n$ will be essential. Finally, in Section \ref{sec4} we prove Theorems \ref{th1.3}, \ref{th1.7} and \ref{th1.9}, which are based on the arguments developed by Licois and V\'{e}ron \cite{LV}.

\section{Proof of Theorems  \ref{th1.1} and \ref{th1.2}}
	
	Let $\zeta$ be an arbitrary point on $\partial \mathbb{S}^n_+$, which we will rename the north pole $N$. Let $\mathcal{F}^{-1}:\overline{ \mathbb{S}^n_+}\backslash \{N\}\rightarrow \overline{\mathbb{R}^n_+}$ be the stereographic projection
	\begin{equation*}
	\mathcal{F}(y)=\left( \frac{2y}{1+|y|^2}, \frac{|y|^2-1}{|y|^2+1}\right),~y\in\mathbb{R}^n, \ y:=(y^1,\ldots,y^n).
	\end{equation*}
For each $0<2s<n$, we set
    \begin{equation}
        \label{e2.1}
        \xi_s(y)=\left( \frac{2}{1+|y|^2}\right)^{\frac{n-2s}{2}}, ~~ y \in {\mathbb{R}}^n.
    \end{equation}
   Let $u$ be a solution of \eqref{p1.1}. Considering the new unknown $v$ defined on ${\mathbb{R}}^n_+$ by
    \begin{equation*}
        v(y)= \xi_1(y)u(\mathcal{F}(y)),
    \end{equation*}
   we have
	\begin{equation}
	\label{e2.2}
	v\in L^{\frac{2n}{n-2}}(\mathbb{R}^n_+)\cap L^{\infty}(\mathbb{R}^n_+) \text{ and } v\in L^{\frac{2(n-1)}{n-2}}(\partial \mathbb{R}^n_+).
\end{equation}	    
   From \eqref{p1.1} and standard computations we get
        \begin{equation}
    \label{p2.1}
   \left\{ \begin{aligned}
    -\Delta v = &~ h_1\left(\frac{v}{\xi_1}\right)v^{\frac{n+2}{n-2}} && \text{in } {\mathbb{R}}^{n}_+, \\
    \frac{\partial v}{\partial \nu} = &~  \xi_1^{\frac{2}{n-2}}k_1\left(\frac{v}{\xi_1}\right)v^{\frac{n}{n-2}} && \text{on } \partial {\mathbb{R}}^n_+,
    \end{aligned}
    \right.
    \end{equation}
    
    In order to show Theorem \ref{th1.1} we used the moving plane method to prove symmetry with respect to the axis $y^n$ of any positive solution of problem (\ref{p2.1}). The following results are based on \cite{ADG, XY1}. 
    \begin{lemma}
    \label{l2.1}
    Let $u$ be a positive solution on $\overline{\mathbb{S}^n_+}$ of \eqref{p1.1}. Under the assumptions of Theorem \ref{th1.1}, $v=\xi_1 (u\circ\mathcal{F})$ is symmetric with respect to the axis $y^n$.
    \end{lemma}
    
	    \begin{proof}
     Given $t\in \mathbb{R}$, we set
    \begin{equation*}
    Q_t=\{ y\in {\mathbb{R}}_+^n;~ y^1<t\}, ~\partial \tilde{Q}_t=\{ y\in \partial \mathbb{R}^n_+; y^1<t\}  \mbox{ and }U_t=\{y\in {\mathbb{R}}_+^n; ~y^1=t\},
    \end{equation*}
    where $y_t:= I_t(y):=(2t-y^1,y')$ is the image of point $y=(y^1,y')$ by the reflection through the hyperplane $U_t$. We define the reflected function by $v^t(y):=v(y_t)$. The proof is carried out in three steps. In the first step we show that
    \begin{equation*}
    \Lambda:=\inf\{t>0; v \geq v^\mu \text{ in } Q_\mu \cup \partial \tilde{Q}_\mu, \forall \mu\geq t\}
    \end{equation*}
    is well-defined, i.e. $\Lambda<+\infty$. The second step consists in proving that if $\Lambda>0$ then $v\equiv v^{\Lambda}$ in $Q_{\Lambda}$. In the third step we conclude that $v$ is a symmetric function in the first $(n-1)$-variables.
    
  \begin{Step}\label{step:1.1}
  $\Lambda<+\infty$.
   \end{Step} 
    Assume that there is $t>0$ such that $v^t(y)\geq v(y)$ for some $y\in Q_t \cup \partial \tilde{Q}_t$. Since $|y| < |y_t|$, we have $\frac{v(y)}{\xi_1(y)}<\frac{v^t(y)}{\xi_1^t(y)}$, and by (\ref{e1.1}),
    \begin{align}
     \notag
    -\Delta(v^t-v) & =  h_1\left(\frac{v^t}{\xi_1^t}\right)(v^t)^{\frac{n+2}{n-2}}-h_1\left(\frac{v}{\xi_1}\right)v^{\frac{n+2}{n-2}}\\ \notag
    & \leq  h_1\left(\frac{v^t}{\xi_1^t}\right)(v^t)^{\frac{n+2}{n-2}}-h_1\left(\frac{v^t}{\xi_1^t}\right)v^{\frac{n+2}{n-2}}\\ \notag
    & \leq  \frac{n+2}{n-2}\max\left\{h_1\left(\frac{v^t}{\xi_1^t}\right),0\right\}(v^t)^{\frac{4}{n-2}}(v^t-v)\\
    \label{e2.4}
    & \leq  C(v^t)^{\frac{4}{n-2}}(v^t-v),  
    \end{align}
    and
    \begin{align}
    \notag
    \frac{\partial (v^t-v)}{\partial \nu} & =(\xi_1^t)^{\frac{2}{n-2}} k_1\left(\frac{v^t}{\xi_1^t} \right)(v^t)^{\frac{n}{n-2}}-\xi_1^{\frac{2}{n-2}} k_1\left(\frac{v}{\xi_1}\right) v^{\frac{n}{n-2}}\\
    \notag
    & \leq (\xi_1^t)^{\frac{2}{n-2}} k_1\left(\frac{v^t}{\xi_1^t} \right) [(v^t)^{\frac{n}{n-2}}-v^{\frac{n}{n-2}}]\\
    \label{e2.4.1}
   & \leq C(v^t)^{\frac{2}{n-2}}(v^t-v),
    \end{align}
    where the last inequalities of (\ref{e2.4}) and (\ref{e2.4.1}) are consequence of $h_1(\frac{v}{\xi}), k_1(\frac{v}{\xi})\in L^{\infty}(\mathbb{R}^n_+)$. Since $v(y)\rightarrow 0$ as $|y|\rightarrow\infty$, for $\varepsilon>0$ fixed, we can take $(v^t-v-\varepsilon)^+$ as test function with compact support in ${Q_t}\cup \partial \tilde{Q}_t$. Then, from (\ref{e2.4}) and (\ref{e2.4.1}) we obtain
    \begin{equation}
    \label{e2.4.2}
    \begin{aligned}
    \int_{Q_t}|\nabla (v^t-v-\varepsilon)^+|^2 dy & \leq  C \int_{Q_t}(v^t)^{\frac{4}{n-2}}(v^t-v)(v^t-v-\varepsilon)^+dy\\
    &~ \text{ }~+C\int_{\partial Q_t}(v^t)^{\frac{2}{n-2}}(v^t-v) (v^t-v-\varepsilon)^+dy'.
    \end{aligned}
    \end{equation}
    Using (\ref{e2.2}), we obtain that the right hand side of the above inequality is limited by the integral of a function that independent of $\varepsilon$. In fact, if $(v^t(y)-v(y)-\varepsilon)^+>0$ for some $y\in Q_t$, then $v^t(y)>v(y)$,
    \begin{equation*}
    (v^t)^{\frac{4}{n-2}}(v^t-v)(v^t-v-\varepsilon)^+\leq  (v^t)^{\frac{2n}{n-2}}\in L^1(\mathbb{R}^n_+)
    \end{equation*}
    and
    \begin{equation*}
    (v^t)^{\frac{2}{n-2}}(v^t-v)(v^t-v-\varepsilon)^+\leq  (v^t)^{\frac{2(n-1)}{n-2}}\in L^1(\partial\mathbb{R}^n_+).
    \end{equation*}
    We denote
    \begin{equation*}
    I_t=:\int_{Q_t}[(v^t-v)^+]^{\frac{2n}{n-2}} dy ~\text{ and }~ II_t=:\int_{\partial \tilde{Q}_t}[(v^t-v)^+]^{\frac{2(n-1)}{n-2}} dy'.
    \end{equation*}
By Fatou's lemma, Sobolev inequality, \eqref{e2.4.2}, dominate convergence theorem and H\"{o}lder inequality, we obtain
    \begin{align}
	\notag
    I_t
    &   \leq \liminf_{\varepsilon \rightarrow 0} \int_{Q_t}[(v^t-v-\varepsilon)^+]^{\frac{2n}{n-2}} dy\\
    \notag
    &\leq \liminf_{\varepsilon \rightarrow 0} C \left(\int_{Q_t}|\nabla(v^t-v-\varepsilon)^+|^2 dy\right)^{\frac{n}{n-2}}\\
    \notag
    &\leq  C \liminf_{\varepsilon \rightarrow 0} \left(\int_{Q_t}(v^t)^{\frac{4}{n-2}}(v^t-v)(v^t-v-\varepsilon)^+ dy\right)^{\frac{n}{n-2}}\\
    \notag
    & ~\text{ } ~ +  C \liminf_{\varepsilon \rightarrow 0} \left(\int_{\partial \tilde{Q}_t}(v^t)^{\frac{2}{n-2}}(v^t-v)(v^t-v-\varepsilon)^+ dy'\right)^{\frac{n}{n-2}}\\
    \notag
    & = C\left(\int_{Q_t}(v^t)^{\frac{4}{n-2}}[(v^t-v)^+]^2 dy\right)^{\frac{n}{n-2}} +  C\left(\int_{\partial \tilde{Q}_t}(v^t)^{\frac{2}{n-2}}[(v^t-v)^+]^2 dy'\right)^{\frac{n}{n-2}}\\
\label{e2.5}
& \leq  C \phi_1(t) I_t +  C\phi_2(t) \left(II_t\right)^{\frac{n}{n-1}},
	\end{align}
    where 
    \begin{equation*}
    \phi_1(t)=\left(\int_{Q_t}(v^t)^{\frac{2n}{n-2}}dy\right)^{\frac{2}{n-2}} \text{ and } \phi_2(t)=\left(\int_{\partial \tilde{Q}_t}(v^t)^{\frac{2(n-1)}{n-2}}dy' \right)^{\frac{n}{(n-1)(n-2)}}.
    \end{equation*}
   Arguing as in the above computations and using Sobolev trace inequality, 
   \begin{align}
   \notag
     \left(II_t\right)^{\frac{n}{n-1}}
    & C\leq \liminf_{\varepsilon \rightarrow 0}  \left(\int_{Q_t}|\nabla(v^t-v-\varepsilon)^+|^2 dy\right)^{\frac{n}{n-2}}\\
    \label{e2.5.1}
& ~\text{ } \leq  C \phi_1(t) I_t +  C\phi_2(t) \left(II_t\right)^{\frac{n}{n-1}}.
   \end{align}
     Since $v^{\frac{2n}{n-2}}\in L^1(\mathbb{R}^n_+)$ and $v^{\frac{2(n-1)}{n-2}}\in L^1(\partial \mathbb{R}^n_+)$, we readily deduce that $\lim_{t\rightarrow +\infty}\phi_i(t)=0$ for $i=1,2$. We can choose $t_1>0$ large enough such that $C\phi_i(t_1)<1/4$, which from \eqref{e2.5} and \eqref{e2.5.1} implies 
    \begin{equation*}
    \int_{Q_t}[(v^t-v)^+]^{\frac{2n}{n-2}} dy=\int_{\partial \tilde{Q}_t}[(v^t-v)^+]^{\frac{2(n-1)}{n-2}} dy'=0,~~~\text{ for all } t>t_1.
    \end{equation*}
 Whence $(v^t-v)^+\equiv 0$ in $Q_t\cup \partial \tilde{Q}_t$ for $t>t_1$. Therefore $\Lambda$ is well defined, i.e. $\Lambda < +\infty$.\
     
        \begin{Step}\label{step:1.2}
    If $\Lambda>0$ then $v\equiv v_{\Lambda}$ in $Q_{\Lambda}$.
   \end{Step} 
     
     By definition of $\Lambda$ and the continuity of the solution, we get $v\geq v^{\Lambda}$ and $\xi_1>\xi_1^{\Lambda}$ in $Q_{\Lambda}$. 
     
     Suppose there is a point $y_0\in Q_{\Lambda}\cup \partial \tilde{Q}_{\Lambda}$ such that $v(y_0)=v^{\Lambda}(y_0)$. As a consequence there is $r>0$ sufficiently small so that    $ \frac{v}{\xi_1}<\frac{v^{\Lambda}}{{\xi_1^{\Lambda}}}$ in $B(y_0,r)\cap (Q_{\Lambda}\cup \partial \tilde{Q}_{\Lambda})$.
   Hence, for $y\in B(y_0,r)\cap Q_{\Lambda}$,
   \begin{align}
   \notag
   -\Delta (v(y)-v^{\Lambda}(y))& =  h_1\left(\frac{v}{\xi_1}\right)v^{\frac{n+2}{n-2}}(y)-h_1\left(\frac{v^\Lambda}{\xi_1^\Lambda}\right)(v^\Lambda)^{\frac{n+2}{n-2}}\\
   \notag
   			& \geq  h_1\left(\frac{v^{\Lambda}}{{\xi_1^{\Lambda}}}\right)(v^{\frac{n+2}{n-2}}(y)-(v^{\Lambda})^{\frac{n+2}{n-2}}(y))\\
   			\label{e2.5.2}
   			& \geq  -C(v(y)-v^{\Lambda}(y)),
\end{align}   
where $C$ is a non-negative constant. The last inequality is consequence of $v$, $v^{\Lambda}$, $h_1(\frac{v^{\Lambda}}{{\xi^{\Lambda}}}) \in L^{\infty}(\mathbb{R}^n)$. If $y_0\in Q_{\Lambda}$, then it follows from maximum principle (see \cite[Proposition 3.7]{GL}) that $v\equiv v^{\Lambda}$ in $B(y_0,r)$ for $r$ small, and since the set $\{y\in Q_{\Lambda};~ v(y)=v^{\Lambda}(y)\}$ is open and closed in $Q_{\Lambda}$, it is concluded that $v\equiv v^{\Lambda}$ in ${Q}_{\Lambda}$. On the other hand, if $y_0\in \partial \tilde{Q}_{\Lambda} $, then
\begin{align}
   \notag
   \frac{\partial (v-v^{\Lambda})}{\partial \nu}(y_0)& = \xi_1^{\frac{2}{n-2}} k_1\left(\frac{v}{\xi_1}\right)v(y_0)^{\frac{n}{n-2}}- (\xi_1^{\Lambda})^{\frac{2}{n-2}}k_1\left(\frac{v^\Lambda}{\xi_1^\Lambda}\right)(v^\Lambda)(y_0)^{\frac{n}{n-2}}\\
   \notag
   			& \geq  k_1\left(\frac{v^{\Lambda}}{{\xi_1^{\Lambda}}}\right)v^{\frac{n}{n-2}}(y_0)(\xi_1(y_0)^{\frac{2}{n-2}}-(\xi_1^{\Lambda})(y_0)^{\frac{2}{n-2}}) \\
   			\label{e2.5.3}
   			& \geq 0,
\end{align}  
and from (\ref{e2.5.2}) along with the proof of \cite[Proposition 3.7]{GL}, one can show that
\begin{equation*}
 \frac{\partial (v-v^{\Lambda})}{\partial \nu}(y_0)<0 \mbox{ provided that } v>v^{\Lambda} \mbox{ in }  B(y_0,r)\cap Q_{\Lambda},
\end{equation*}
which contradicts (\ref{e2.5.3}).
    
    Now, we assume that $v>v^{\Lambda}$ in $Q_{\Lambda} \cup \partial \tilde{Q}_{\Lambda}$. We can choose two compacts $K \subset Q_{\Lambda}$ and $W\subset \partial \tilde{Q}_{\Lambda}$, and a number $\delta > 0$ such that for all $  t \in (\Lambda-\delta, \Lambda )$, we have $K \subset Q_t$,  $W\subset \partial \tilde{Q}_t$ and
    \begin{equation}
    \label{e2.6}
    C\phi_{1}(t)^{\frac{n-2}{2}}=\int_{Q_{t}\backslash K}(v^t)^{\frac{2n}{n-2}}dy<\frac{1}{4},~C\phi_{2}(t)^{\frac{(n-1)(n-2)}{n}}=\int_{\partial \tilde{Q}_t \backslash W}(v^t)^{\frac{2(n-1)}{n-2}}dy'<\frac{1}{4}.
    \end{equation}
    Moreover, there exists $0 < \delta_1 < \delta$, so that 
    \begin{equation}
    \label{e2.7}
    v > v^t \text{ in } K\cup W~ \text{ for all } t \in (\Lambda - \delta_1 , \Lambda ).
    \end{equation}
Using (\ref{e2.6}) and following as in Step \ref{step:1.1}, by considering the integrals over $Q_t \backslash K \cup \partial \tilde{Q}_t\backslash W$, we see that
$(v^t - v)^+ \equiv 0$ in $Q_t \backslash (K \cup W)$. By (\ref{e2.7}) we get $v > v^t$ in ${Q}_t$ for all $t \in (\Lambda-\delta_1 , \Lambda )$, which contradicts the definition of $\Lambda$.
	\begin{Step}\label{step:1.3}
	 Symmetry.
	\end{Step}
	If $\Lambda>0$, then 
	\begin{equation*}
	h_1\left(\frac{v}{\xi_1}\right)=-\frac{\Delta v}{v^{\frac{n+2}{n-2}}}=-\frac{\Delta v^{\Lambda}}{(v^{\Lambda})^{\frac{n+2}{n-2}}}=h_1\left(\frac{v^\Lambda}{\xi_1^\Lambda}\right)=h_1\left(\frac{v}{\xi_1^\Lambda}\right)<h_1\left(\frac{v}{\xi_1}\right)~ \text{ in } Q_{\Lambda}.
	\end{equation*}
	This is a contradiction. Thus $\Lambda=0$. By continuity of $v$, we have $v^{0}(y) \leq v (y)$ for all $y\in Q_0$. We can also perform the moving plane argument to the left side to get a corresponding $\Lambda'$. An analogue to Step \ref{step:1.1} and Step \ref{step:1.2} give us $\Lambda' = 0$. Then we get $v^0(y) \geq v(y)$ for $y\in Q_0$. This fact and the opposite inequality implies that $v$ is symmetric with respect to $U_0$. Therefore, we obtain $\Lambda=\Lambda'=0$ for all directions that are vertical to the $y^n$ direction. Thus $v$ is symmetric with respect to the axis $y^n$, which completes the proof.    	    
	    \end{proof}

Before continuing with the proof of Theorem \ref{th1.1}, we want to comment on the proof of Lemma \ref{l2.1}. In \cite{XY1},  the functions $f$ and $b$ was required to be nondecreasing continuous to get the symmetry of the positive solutions on $\overline{\mathbb{R}^n_+} $ of \eqref{p2.1}. Thanks to maximum principle \cite[Proposition 3.7]{GL} and its proof, we don’t need this assumption here.

    \begin{lemma}
	\label{l2.2}
	Let $u$ be a positive solution on $\overline{\mathbb{S}^n_+}$ of \eqref{p1.1}. Assume as in Lemma \ref{l2.1}. Then for each $r\in [0,\pi/2]$,  $u$ is constant in 
	\begin{equation}
	\label{e2.8}	
	A_r=\{ \omega \in \overline{\mathbb{S}^n_+};~ r=\inf\{|\omega-\zeta|_{{\mathbb{S}^n}};~\zeta\in\partial \mathbb{S}^n_+\}\},
	\end{equation}
	where $|\cdot|_{{\mathbb{S}^n}}$ is the distance in ${\mathbb{S}^n}$.
	\end{lemma}    
    \begin{proof}
    Let $\zeta_1$, $\zeta_2\in A_r$, $r>0$. There is $\zeta\in \partial\mathbb{S}^n_+$ so that $|\zeta_1-\zeta|_{{\mathbb{S}^n}}=|\zeta_2-\zeta|_{{\mathbb{S}^n}}$. Let $\mathcal{F}^{-1}:\overline{ \mathbb{S}^n_+}\backslash \{ \zeta\}\rightarrow \overline{\mathbb{R}^n_+}$ be the stereographic projection. Then $\mathcal{F}(\zeta_1)$ and $\mathcal{F}(\zeta_2)$ are symmetrical points with respect to the axis $y^n$. We define $v=\xi_1(u\circ\mathcal{F})$ in $\mathbb{R}^n_+$. From Lemma \ref{l2.1} we obtain that $v$ is symmetric with respect to the axis $y^n$, and by the definition of $v$, and by the symmetry of $\xi_1$, we have $u(\zeta_1)=u(\zeta_2)$.
    
    Therefore $u$ is constant in $A_r$ for each $r\in(0,\pi/2]$. By continuity of $u$, we have $u$ is constant in $A_0$.    
    \end{proof}
    
   {\it Proof of Theorem \ref{th1.1}.} Case (i): let $u$ be a positive solution on $\overline{\mathbb{S}^n_+}$ of \eqref{p1.1}. For an arbitrary point $\zeta\in\partial \mathbb{S}^n_+$, let $\mathcal{F}^{-1}:\overline{\mathbb{S}^n_+}\backslash \{ \zeta\}\rightarrow \overline{\mathbb{R}^n_+}$ be the stereographic projection. We consider the equation (\ref{p2.1}), where $v=\xi_1(u\circ\mathcal{F})$ in $\mathbb{R}^n_+$. By Lemma \ref{l2.1} we have that $u$ is constant in each circle parallel to the equator. Assume now that there exists $\zeta_0 \in \partial \mathbb{S}^n_+$ such that $b(u(\zeta_0))=0$. By Lemma \ref{l2.1}, $\frac{\partial u}{\partial \nu}\equiv 0$ on $\partial \mathbb{S}^n_+$, and, therefore, $\frac{\partial v}{\partial \nu}\equiv 0$ on $\partial \mathbb{R}^n_+$. Define
  \begin{equation*}
  v^*(y)=\begin{cases}
  			v(y',y^n), &\text{ if } y^n\geq 0,\\
  			v(y',-y^n), &\text{ if } y^n<0,
  		\end{cases}
\end{equation*}   
where $y=(y',y^n)\in \mathbb{R}^n_+$. We easily can check that $v^*$ satisfies in the weak sense 
\begin{equation*}
    -\Delta v^*  =  h\left(\frac{v^*}{\xi_1}\right){v^*}^{\frac{n+2}{n-2}},~ v^*>0   \text{ in } {\mathbb{R}}^{n}.
\end{equation*}
Adapting the Lemma \ref{l2.1} for $v^*$ in the whole space $\mathbb{R}^n$, we conclude that $v^*$ is radially symmetric. Hence
   \begin{equation}
   \label{e2.9}
   v^*(y)=v(y)=C, ~\forall y\in \mathbb{R}^n_+, |y|=1,
   \end{equation}
    where $C$ is a constant.
    
    On the other hand, the set $\mathcal{F}(\{y\in\mathbb{R}^n_+;~|y|=1\})$ intersects perpendicularly to $A_r$ for any $r\in (0,\pi/2]$. From (\ref{e2.9}) we obtain that $u$ is constant in $\mathcal{F}(\{y\in\mathbb{R}^n_+; |y|=1\})$, which jointly with Lemma \ref{l2.2} implies $u$ is constant in $\mathbb{S}^n_+$.
    
    Case (ii): let $u$ be a nonnegative solution of \eqref{p1.1}. Then $v=\xi_1 (u\circ \mathcal{F})$, where $\mathcal{F}^{-1}$ is the stereographic projection with $\zeta \partial \mathbb{S}^n_+$ being the north pole, resolves \eqref{p2.1}. From maximum principle, $u$ is zero or positive in $\overline{\mathbb{S}^n_+}$. Now, consider $\tilde{v}=\xi_1 (u\circ \tilde{\mathcal{F}})$ where $\tilde{\mathcal{F}}^{-1}$ is the stereographic projection with $e^n=(0,...,1)\in \mathbb{R}^{n+1}$ being the south pole. By standard computations we have
    \begin{equation*}
    \begin{aligned}
    -\Delta \tilde{v} & =\xi_1^{\frac{n+2}{n-2}}\left(f(u\circ \mathcal{F})+\frac{n(n-2)}{4}(u\circ \mathcal{F}) \right) && \text{in } B^n_1(0),\\
    \frac{\partial \tilde{v}}{\partial \eta} & =-\frac{n-2}{2}\tilde{v}+b(\tilde{v}) &&\text{on } \partial B^n_1(0), 
    \end{aligned}
    \end{equation*}
    where $B^n_1(0)$ is the unit ball with center at $0$ in the Euclidean space $\mathbb{R}^n$. By Lemma \ref{l2.2}, $\tilde{v}$ is radially symmetrical. Therefore, the case (ii) follows from the assumptions on $f$ and $b$, and from the maximum principle to $\tilde{v}$.
    \hfill $\square$\\
    
    Now, we consider the functions $v_1$, $v_2$ defined on $\mathbb{R}^n_+$ by
	\begin{equation*}
	v_1(y)=\xi_1(y)u_1(\mathcal{F}(y)),~v_2(y)=\xi_1(y)u_2(\mathcal{F}(y)),
\end{equation*}	    
    where $\xi_1$ is defined in (\ref{e2.1}) and $(u_1,u_2)$ is a pair of solutions of (\ref{p1.2}). Hence
    \begin{equation}
    \label{e2.10}
    v_1,v_2\in L^{\frac{2n}{n-2}}(\mathbb{R}^n_+)\cap L^{\infty}(\mathbb{R}^n_+)\text{ and } v_1,v_2\in L^{\frac{2(n-1)}{n-2}}(\partial \mathbb{R}^n_+).
    \end{equation}  
   From (\ref{p1.2}) and standard computations gives
    \begin{equation}
    \label{p2.2}
   \left\{ \begin{aligned}
    -\Delta v_1 & =  h_{11}\left(\frac{v_1}{\xi_1},\frac{v_2}{\xi_1}\right)v_1^{\frac{n+2}{n-2}} && \text{ in } {\mathbb{R}}^{n}_+, \\
    -\Delta v_2 & =  h_{21}\left(\frac{v_1}{\xi_1},\frac{v_2}{\xi_1}\right)v_2^{\frac{n+2}{n-2}} && \text{ in } {\mathbb{R}}^{n}_+, \\
    \frac{\partial v_1}{\partial \nu} &= \xi_1^{\frac{2}{n-2}}k_{11}\left( \frac{v_1}{\xi_1},\frac{v_2}{\xi_1} \right) , \frac{\partial v_2}{\partial \nu}  = \xi_1^{\frac{2}{n-2}} k_{21}\left( \frac{v_1}{\xi_1},\frac{v_2}{\xi_1} \right)  && \text{ on } \partial{\mathbb{R}}^{n}_+.
    \end{aligned}
    \right.
    \end{equation}
    
    To prove the Theorem \ref{th1.2}, we will use the same arguments that were used in the proof of the Theorem \ref{th1.1}. The following lemma shows the symmetry of solutions of the above system.
    
    \begin{lemma}
    \label{l2.3}
    Let $(u_1,u_2)$ be a pair of solutions of \eqref{p1.2}. Then $v_1=\xi_1 (u_1\circ\mathcal{F})$ and $v_2=\xi_1 (u_2\circ\mathcal{F})$ are symmetric with respect to the axis $y^n$.
    \end{lemma}
    
    \begin{proof}
   As in Lemma \ref{l2.1}, for $t\in \mathbb{R}$ we define the same tools we need; let $Q_t$, $\partial\tilde{Q}_t$, $U_t$, $y_t:= I_t(y):=(2t-y^1,y')$, and the reflected functions by $v_1^t(y):=v_1(y_t)$, $v_2^t:=v_2(y_t)$. The proof will be carried out in three steps. In the first step we show that
    \begin{equation*}
    \Lambda:=\inf\{t>0; v_1 \geq v_1^\mu, v_2\geq v_2^\mu \text{ in } Q_\mu \cup \partial\tilde{Q}_\mu, \forall \mu\geq t\}
    \end{equation*}
    is well-defined, i.e. $\Lambda<+\infty$. The second step consists in proving that if $\Lambda>0$ then $v_1\equiv v_1^{\Lambda}$ or $v_2\equiv v_2^{\Lambda}$ in $Q_{\Lambda}$. The third step we conclude the symmetries of $v_1$ and $v_2$.
   \begin{Step}\label{step:2.1}
   $\Lambda<+\infty$.
   \end{Step}  
    Assume that there is $t>0$ such that $v_1^t(y)\geq v_1(y)$ for some $y\in Q_t\cup \partial\tilde{Q}_t$. Since $|y \vert < \vert y_t \vert$, we have, $\frac{v_1}{\xi_1}<\frac{v_1^t}{\xi_1^t}$.
 
  If $v_2>v_2^t$, then from (\ref{e1.3.1}) we have
    \begin{align}
    \notag
    -\Delta(v_1^t-v_1) & =  h_{11}\left(\frac{v_1^t}{\xi_1^t},\frac{v_2^t}{\xi_1^t}\right)(v_1^t)^{\frac{n+2}{n-2}}-h_{11}\left(\frac{v_1}{\xi_1},\frac{v_2}{\xi_1}\right)v_2^{\frac{n+2}{n-2}}\\
    \notag
    & \leq  h_{11}\left(\frac{v_1^t}{\xi_1^t},\frac{v_2^t}{\xi_1^t}\right)(v_1^t)^{\frac{n+2}{n-2}}-h_{11}\left(\frac{v_1}{\xi_1},\frac{v_2^t}{\xi_1}\right)v_1^{\frac{n+2}{n-2}}\\
    \notag
    & =  h_{11}\left(\frac{v_1^t}{\xi_1^t},\frac{v_2^t}{\xi_1^t}\right)(v_1^t)^{\frac{n+2}{n-2}}-h_{11}\left(\frac{v_1}{\xi_1},\frac{v_2^tv_1}{\xi_1 v_1}\right)v_1^{\frac{n+2}{n-2}}\\
    \notag
    & \leq  h_{11}\left(\frac{v_1^t}{\xi_1^t},\frac{v_2^t}{\xi_1^t}\right)(v_1^t)^{\frac{n+2}{n-2}}-h_{11}\left(\frac{v_1^t}{\xi_1^t},\frac{v_2^t v_1^t}{\xi_1^t v_1}\right)v_1^{\frac{n+2}{n-2}}\\
    \notag
    & \leq h_{11}\left(\frac{v_1^t}{\xi_1^t},\frac{v_2^t}{\xi_1^t}\right)(v_1^t)^{\frac{n+2}{n-2}}-h_{11}\left(\frac{v_1^t}{\xi_1^t},\frac{v_2^t}{\xi_1^t}\right)v_1^{\frac{n+2}{n-2}}\\
    \notag
    &\leq  C (v_1^t)^{\frac{4}{n-2}}\max\left\{h_{11}\left(\frac{v_1^t}{\xi_1^t},\frac{v_2^t}{\xi_1^t}\right),0\right\}(v_1^t-v_1)\\
    \label{e2.13}
    &\leq  C (v_1^t)^{\frac{4}{n-2}}(v_1^t-v_1),
    \end{align}   
    where $C$ is a non-negative constant.
    
    If $v_2^t>v_2$, then $\frac{v_1^t v_2^t}{\xi_1^t}>\frac{v_1 v_2}{\xi_1}$ and 
    \begin{align}
    \notag
    -\Delta(v_1^t-v_1) & =   h_{11}\left(\frac{v_1^t}{\xi_1^t},\frac{v_2^t}{\xi_1^t}\right)(v_1^t)^{\frac{n+2}{n-2}}-h_{11}\left(\frac{v_1v_2}{\xi_1 v_2},\frac{v_2v_1}{\xi_1 v_1}\right)v_1^{\frac{n+2}{n-2}}\\
     \notag
    & \leq  h_{11}\left(\frac{v_1^t}{\xi_1^t},\frac{v_2^t}{\xi_1^t}\right)(v_1^t)^{\frac{n+2}{n-2}}-h_{11}\left(\frac{v_1^t v_2^t}{\xi_1^t v_2},\frac{v_2^t v_1^t}{\xi_1^t v_1}\right)v_1^{\frac{n+2}{n-2}}\\
     \notag
    & \leq  h_{11}\left(\frac{v_1^t}{\xi_1^t},\frac{v_2^t}{\xi_1^t}\right)(v_1^t)^{\frac{n+2}{n-2}}-h_{11}\left(\frac{v_1^t v_2^t}{\xi_1^t v_2},\frac{v_2^t}{\xi_1^t}\right)v_1^{\frac{n+2}{n-2}}\\
     \notag
    & \leq  h_{11}\left(\frac{v_1^t}{\xi_1^t},\frac{v_2^t}{\xi_1^t}\right)(v_1^t)^{\frac{n+2}{n-2}}-h_{11}\left(\frac{v_1^t}{\xi_1^t},\frac{v_2^t}{\xi_1^t}\right)\left(\frac{v_2}{v_2^t}\right)^{\frac{n+2}{n-2}}v_1^{\frac{n+2}{n-2}}\\
     \notag
    & =  (v_2^t)^{-\frac{n+2}{n-2}} h_{11}\left(\frac{v_1^t}{\xi_1^t},\frac{v_2^t}{\xi_1^t}\right)((v_1^t v_2^t)^{\frac{n+2}{n-2}}-(v_1v_2)^{\frac{n+2}{n-2}})\\
     \notag
    &\leq  C (v_1^t)^{\frac{4}{n-2}}\max\left\{h_{11}\left(\frac{v_1^t}{\xi_1^t},\frac{v_2^t}{\xi_1^t}\right),0\right\}[(v_1^t-v_1)+v_1^t (v_2^t)^{-1}(v_2^t-v_2)]\\
     \notag
    &\leq  C (v_1^t)^{\frac{4}{n-2}}[(v_1^t-v_1)+(u_1\circ\mathcal{F})^t((u_2\circ\mathcal{F})^t)^{-1}(v_2^t-v_2)]\\
    \label{e2.14}
    &\leq  C (v_1^t)^{\frac{4}{n-2}}[(v_1^t-v_1)+(v_2^t-v_2)],
    \end{align}
    where the last inequality is consequence of $(u_1\circ\mathcal{F})^t((u_2\circ\mathcal{F})^t)^{-1}\in L^{\infty}(\mathbb{R}^n_+)$ and $C$ is a non-negative constant. Making similar computations, we have
    \begin{equation}
    \label{e2.14.1}
    \frac{\partial (v_1^t-v_1)}{\partial \nu} \leq C (v^t)^{\frac{2}{n-2}}[(v_1^t-v_1)+ (v_2^t-v_2)^+],
    \end{equation}
    where $C$ is a nonnegative constant. Since $v_1(y)\rightarrow 0$ as $|y|\rightarrow \infty$, for fixed $\varepsilon>0$, we can take $(v_1^t-v_1-\varepsilon)^+$ as test function with compact support in $Q_t \cup \partial \tilde{Q}_t$ for (\ref{e2.13})-(\ref{e2.14.1}). Then we obtain
    \begin{equation}
    \label{e2.14.2}
    \begin{aligned}
    \int_{Q_t}|\nabla (v_1^t-v_1-\varepsilon)^+|^2dy \leq & C \int_{Q_t}(v_1^t)^{\frac{4}{n-2}}[(v_1^t-v_1)+(v_2^t-v_2)^+](v_1^t-v_1-\varepsilon)^+dy\\
     +& C \int_{\partial Q_t}(v^t)^{\frac{2}{n-2}}[(v_1^t-v_1)+ (v_2^t-v_2)^+](v_1^t-v_1-\varepsilon)^+dy'.
    \end{aligned}
    \end{equation}
    By (\ref{e2.10}), we obtain that the right hand side of the above inequality is limited by the integral of a function independent of $\varepsilon$, namely, 
    \begin{equation*}
    \begin{aligned}
    (v_1^t)^{\frac{4}{n-2}}[(v_1^t-v_1)+(v_2^t-v_2)^+](v_1^t-v_1-\varepsilon)^+\leq  (v_1^t)^{\frac{n+2}{n-2}}(v_1^t+v_2^t)\in L^1(\mathbb{R}^n_+),\\
    (v_1^t)^{\frac{2}{n-2}}[(v_1^t-v_1)+(v_2^t-v_2)^+](v_1^t-v_1-\varepsilon)^+\leq  (v_1^t)^{\frac{2(n-1)}{n-2}}(v_1^t+v_2^t)\in L^1(\partial \mathbb{R}^n_+).
    \end{aligned}
    \end{equation*}
We denote 
\begin{equation*}
I_{t,i}=:\int_{Q_t}[(v_i^t-v_i)^+]^{\frac{2n}{n-2}} dy, ~~ II_{t,i}=:\int_{\partial\tilde{Q}_t}[(v_2^t-v_2)^+]^{\frac{2(n-1)}{n-2}} dy',~i=1,2.
\end{equation*}     
Using Fatou's lemma, Sobolev inequality, \eqref{e2.14.2}, dominate convergence theorem and H\"{o}lder inequality, we obtain
    \begin{align}
    \notag
     I_{t,1}& \leq \liminf_{\varepsilon \rightarrow 0} C \left(\int_{Q_t}|\nabla (v_1^t-v_1-\varepsilon)^+|^2dy \right)^{\frac{n}{n-2}}\\
	  \notag
     & ~\text{ }~= C \left( \int_{Q_t}(v_1^t)^{\frac{4}{n-2}}[(v_1^t-v_1)+(v_2^t-v_2)^+](v_1^t-v_1)^+dy \right.\\
     \notag
     & ~\text{ }~~\text{ }~~ \left. +\int_{\partial \tilde{Q}_t}(v_1^t)^{\frac{2}{n-2}}[(v_1^t-v_1)+(v_2^t-v_2)^+](v_1^t-v_1)^+dy' \right)^{\frac{n}{n-2}}\\
     \notag
    & ~\text{ }~\leq C \phi_{1,1}(t)\left[ I_{t,1} +\left(I_{t,1}\right)^{\frac{1}{2}} \left(I_{t,2} \right)^{\frac{1}{2}}\right] + C\phi_{1,2}(t)\left[ \left(II_{1,t} \right)^{\frac{n}{n-1}}\right.\\
    \label{e2.18.1}
    & ~\text{ }~ ~\text{ }\text{ }~ \left. +\left(II_{t,1} \right)^{\frac{n}{2(n-1)}} \left(II_{t,2}\right)^{\frac{n}{2(n-1)}}\right]
    \end{align}
    where 
    \begin{equation*}
    \phi_{1,1}(t)^{\frac{n-2}{2}}=\int_{Q_t}(v_1^t)^{\frac{2n}{n-2}}dy ~\text{ and } ~\phi_{1,2}(t)^{\frac{(n-1)(n-2)}{n}}=\int_{\partial\tilde{Q}_t}(v_1^t)^{\frac{2(n-1)}{n-2}}dy'.
    \end{equation*}
    Arguing as in the above computations and using Sobolev trace inequality, 
   \begin{align}
   \notag
     \left(II_{t,1} \right)^{\frac{n}{n-1}} \leq &  \liminf_{\varepsilon \rightarrow 0} C \left(\int_{Q_t}|\nabla(v_1^t-v_1-\varepsilon)^+|^2 dy\right)^{\frac{n}{n-2}}\\
     \notag
\leq & C \phi_{1,1}(t)\left[ I_{t,1} +\left(I_{t,1}\right)^{\frac{1}{2}} \left(I_{t,2} \right)^{\frac{1}{2}}\right] + C\phi_{1,2}(t)\left[ \left(II_{1,t} \right)^{\frac{n}{n-1}}\right.\\
\label{e2.18.2}
    &  \left. +\left(II_{t,1} \right)^{\frac{n}{2(n-1)}} \left(II_{t,2}\right)^{\frac{n}{2(n-1)}}\right].
   \end{align}
      Making a similar computation, we have
   \begin{align}
   \notag
   I_{t,2}\leq & C \phi_{2,1}(t)\left[ I_{t,2} +\left(I_{t,1}\right)^{\frac{1}{2}} \left(I_{t,2} \right)^{\frac{1}{2}}\right] + C\phi_{2,2}(t)\left[ \left(II_{2,t} \right)^{\frac{n}{n-1}}\right.\\
   \label{e2.18.3}
    & \left. +\left(II_{t,1} \right)^{\frac{n}{2(n-1)}} \left(II_{t,2}\right)^{\frac{n}{2(n-1)}}\right],
   \end{align}
   \begin{align}
   \notag
     \left(II_{t,2} \right)^{\frac{n}{n-1}} \leq &  C \phi_{2,1}(t)\left[ I_{t,2} +\left(I_{t,1}\right)^{\frac{1}{2}} \left(I_{t,2} \right)^{\frac{1}{2}}\right] + C\phi_{2,2}(t)\left[ \left(II_{2,t} \right)^{\frac{n}{n-1}}\right.\\
\label{e2.18.4}
    &  \left. +\left(II_{t,1} \right)^{\frac{n}{2(n-1)}} \left(II_{t,2}\right)^{\frac{n}{2(n-1)}}\right],
   \end{align}
   where 
   \begin{equation*}
   \phi_{2,1}(t)^{\frac{n-2}{2}}=\int_{Q_t}(v_2^t)^{\frac{2n}{n-2}}dy ~\text{ and } ~\phi_{2,2}(t)^{\frac{(n-1)(n-2)}{n}}=\int_{\partial\tilde{Q}_t}(v_2^t)^{\frac{2(n-1)}{n-2}}dy'.
   \end{equation*}
     By \eqref{e2.10} we deduce that $\lim_{t\rightarrow \infty}\phi_{i,j}(t)=0$ for all $i,j=1,2$. Then we can choose a $t_1\in \mathbb{R}$ such that $
    C\phi_{i,j}(t)<\frac{1}{8}$ for all  $t>t_1$ and $i,j=1,2$.
    From (\ref{e2.18.1}) - (\ref{e2.18.4}), we deduce
    \begin{equation*}
    \int_{Q_t}[(v_i^t-v_i)^+]^{\frac{2n}{n-2}} dy=\int_{\partial\tilde{Q}_t}[(v_2^t-v_2)^+]^{\frac{2(n-1)}{n-2}} dy' = 0~ \text{ for all } t>t_1 \text{ and } i=1,2.
    \end{equation*}
   It follows that $(v_1^t-v_1)^+\equiv 0\equiv(v_2^t-v_2)^+$ on $Q_t\cup \partial \tilde{Q}_t$ for all $t>t_1$. Thus $\Lambda$ is well defined, i.e. $\Lambda<+\infty$.\

	    \begin{Step}\label{step:2.2}
	    If $\Lambda>0$ then $v_1\equiv v_1^{\Lambda}$ or $v_2\equiv v_2^{\Lambda}$ in $Q_{\Lambda}$.
   \end{Step}   
	
	By definition of $\Lambda$ and continuity of solutions, we get $v_1\geq v_1^{\Lambda}$ and $v_2\geq v_2^{\Lambda}$ in $Q_{\Lambda}$. Then, for $i=1,2$,
	\begin{equation*}
	\left\{\begin{aligned}
	-\Delta (v_i-v_i^{\Lambda}) & =h_{i1}\left(\frac{v_1}{\xi_1},\frac{v_2}{\xi_1}  \right)v_i^{\frac{n-2}{n+2}}-h_{i1}\left(\frac{v_1^{\Lambda}}{\xi_1^{\Lambda}},\frac{v_2^{\Lambda}}{\xi_1^{\Lambda}} \right)(v_i^{\Lambda})^{\frac{n-2}{n+2}}\geq 0 &\text{in } Q_{\Lambda},\\
	 \frac{\partial(v_i-v_i^{\Lambda})}{\partial \nu} & =0 &\text{on } \partial\tilde{Q}_{\Lambda},\\
	v_i-v_i^{\Lambda} & \geq 0 & \text{in } Q_{\Lambda}.
	\end{aligned}\right.
	\end{equation*}
    From maximum principle, we obtain either $v_i\equiv v_i^{\Lambda}$ in $Q_{\Lambda}$ for some $i=1,2$, or $v_i> v_i^{\Lambda}$ in $Q_{\Lambda}$ for all $i=1,2$. Suppose $v_1>v_1^{\Lambda}$ and $v_2>v_2^{\Lambda}$ in $Q_{\Lambda}$. We can choose compacts $K \subset Q_{\Lambda}$, $W\subset \partial \tilde{Q}_{\lambda}$ and a number $\delta > 0$ such that for all  $t \in (\Lambda-\delta, \Lambda )$ we have $K \subset Q_t$, $W\subset \partial \tilde{Q}_t$ and
    \begin{equation}
    \label{e2.19}
    C\phi_{i,1}(t)^{\frac{n-2}{2}}=\int_{Q_{t}\backslash K}(v_i^t)^{\frac{2n}{n-2}}dy<\frac{1}{8},~C\phi_{2,i}(t)^{\frac{(n-1)(n-2)}{n}}=\int_{\partial\tilde{Q}_t\backslash W}(v_i^t)^{\frac{2(n-1)}{n-2}}dy' < \frac{1}{8}.
    \end{equation}
    On the other hand, there exists $0 < \delta_1 < \delta$, such that 
    \begin{equation}
    \label{e2.20}
    v_1 > v_1^t \text{ and } v_2>v_2^t \text{ in } K \cup W,~ \text{for all } t \in (\Lambda - \delta_1 , \Lambda ).
    \end{equation}
Using (\ref{e2.19}) and following as in Step \ref{step:2.1}, since the integrals are over $Q_t \backslash K$, we see that
$(v_i^t - v_i)^+ \equiv 0$ in $Q_t \backslash K$ for $i=1,2$. By (\ref{e2.20}) we get $v_1 > v_1^t$ and $v_2>v_2^t$ in $Q_t$ for all $t \in (\Lambda-\delta_1 , \Lambda )$, contradicting the definition of $\Lambda$. \

   \begin{Step}\label{step:2.3}
   Symmetry.
   \end{Step}  
Suppose $\Lambda>0$. From Step \ref{step:2.2} we can assume $v_1=v_1^{\Lambda}$. Then
	\begin{equation*}
	\begin{aligned}
	h_{11}\left(\frac{v_1}{\xi_1},\frac{v_2}{\xi_1}\right) & = -\frac{\Delta v_1}{v_1^{\frac{n+2}{n-2}}} = -\frac{\Delta v_1^{\Lambda}}{(v_1^{\Lambda})^{\frac{n+2}{n-2}}} = h_{11}\left(\frac{v_1^{\Lambda}}{\xi_1^{\Lambda}},\frac{v_2^{\Lambda}}{\xi_1^{\Lambda}}\right)\\
	& \leq h_{11}\left(\frac{v_1}{\xi_1^{\Lambda}},\frac{v_2}{\xi_1^{\Lambda}}\right) < h_{11}\left(\frac{v_1}{\xi_1},\frac{v_2}{\xi_1}\right).
	\end{aligned}
	\end{equation*}
This is a contradiction. Thus $\Lambda=0$. By continuity of $v_1$ and $v_2$, we have $v_1^{0}(y) \leq v_1(y)$ and $v_2^{0}(y)\leq v_2(y)$ for all $y\in Q_0$ . We can also perform the moving plane method to the left side and find a corresponding $\Lambda'$. By a similar argument as in Step \ref{step:2.1} and Step \ref{step:2.2}, we can show that $\Lambda' = 0$. Then we get $v_1^0(y) \geq v_1(y)$ and $v_2^0(y)\geq v_1(y)$ for $y\in Q_0$. This fact and the above inequality imply that $v_1$ and $v_2$ are symmetric with respect to $U_0$. Therefore, if $\Lambda=\Lambda'=0$ for all directions that are vertical to the $y^n$ direction, then $v_1$ and $v_2$ are symmetric with respect to the axis $y^n$.    
    \end{proof}
    
    \begin{lemma}
    \label{l2.4}
    Let $(u_1, u_2)$ be a solution of \eqref{p1.2}. Then for each $r\in [0,\pi/2]$, we have that $u_1=u_1(r)$ and $u_2=u_2(r)$ in $A_r$, where $A_r$ is defined by \eqref{e2.8}.
    \end{lemma}
    
    \begin{proof}
    The arguments for the proof are the same as the Lemma \ref{l2.2}.
    \end{proof}
    
    {\it Proof of Theorem \ref{th1.2}}.
    	
    	Since $h_{i1}(at,bt)$ and $h_{i1}(at,bt)t^{\frac{n+2}{n-2}}$ are decreasing and nondecreasing in $t$, respectively, it follows that 
    	\begin{align*}
    	h_{i1}(at,bt)t^{\frac{n+2}{n-2}}\geq h_{i1}(ar,br)r^{\frac{n+2}{n-2}}\geq h_{i1}(at,bt)r^{\frac{n+2}{n-2}} \text{ for all } 0<r<t \text{ and } a,b>0.
\end{align*}    	
Fixing $t$ and taking $r\rightarrow 0$, we have $h_{i1}(at,bt)\geq 0$, and therefore, $h_{i1}$ is nonnegative in $(0,+\infty)\times(0,+\infty)$ for $i=1,2$. Similarly $k_{i1}$ is nonnegative in $(0,+\infty)\times(0,+\infty)$ for $i=1,2$. Therefore, using  Lemma \ref{l2.4} and maximum principle, the prove of Theorem \ref{th1.2} is entirely similar to  the proof of Theorem \ref{th1.1}, so we omit the details here.
    \hfill $\square$

    \section{Proof of Theorems \ref{th1.5} and \ref{th1.6}}
    \label{sec3}
    
    \subsection{Problems involving the fractional Laplace operator}
\label{subsec3.1}    
    \
    
    We start with a brief review of the operator $(-\Delta)^s$ and the relationship between differential equations and integral equations.
    
    Let $\mathcal{S}(\mathbb{R}^n)$ be the Schwartz space of the $C^{\infty}$ rapidly decreasing functions in $\mathbb{R}^n$. Given $\psi \in \mathcal{S}(\mathbb{R}^n)$ and $0<2s<n$, the fractional Laplacian operator $(-\Delta)^s$ is defined by
    \begin{equation*}
    \widehat{(-\Delta)^s}\psi(x):= |x|^{2s}\widehat{\psi}(x),~x\in\mathbb{R}^n,
\end{equation*}
    where $~\widehat{ }~$ denotes the Fourier transform. In particular, for $0<s<1$, the above definition is equivalent up to constants to (see \cite{DNPV}) 
    \begin{equation*}
    (-\Delta)^s\psi(x)=P.V.\int_{\mathbb{R}^n}\frac{\psi(x)-\psi(y)}{|x-y|^{n+2s}}dy,
    \end{equation*}
    where $P.V.$ denotes $\lim_{\varepsilon \rightarrow 0^+}\int_{|x-y|>\varepsilon}$.
     Henceforth we will write for $s\notin \mathbb{N}$ as $s=m+\gamma$, with $m\in \mathbb{N}$ and $0<\gamma<1$. By definition of the fractional Laplacian, $(-\Delta)^s=(-\Delta)^m(-\Delta)^{\gamma}=(-\Delta)^{\gamma}(-\Delta)^m$. Moreover, it is very simple to check that
\begin{equation*}
(-\Delta)^s\psi(x)\leq C(1+|x|^{n+2\gamma})^{-1},~\psi\in \mathcal{S}(\mathbb{R}^n),
\end{equation*} 
where $C$ is a positive constant that depends on $n$, $s$ and $\psi$.
      One can extend this operator to a wider space of distributions as the following. Let
    \begin{equation*}
    \mathcal{L}_{s}=\left\{ v:\mathbb{R}^n\rightarrow \mathbb{R}; \int_{\mathbb{R}^n}\frac{|v(x)|}{1+|x|^{n+2\gamma}}<+\infty \right\}, ~s=m+\gamma, ~m\in \mathbb{N}, ~0<\gamma<1.
    \end{equation*}
     For $v\in \mathcal{L}_s$, we define $(-\Delta)^s v$ as a distribution:
\begin{equation*}
\langle (-\Delta)^s v, \phi \rangle :=\langle v,  (-\Delta)^s \phi \rangle ~ \text{ for all } \phi\in \mathcal{S}(\mathbb{R}^n).
\end{equation*}
    
    Consider the following equation
    \begin{equation}
    \label{3.e3.1}
     (-\Delta)^s v = f(x,v) \text{ in } \mathbb{R}^n,
    \end{equation}
    where $f:\mathbb{R}^n\times\mathbb{R}\rightarrow \mathbb{R}$ is a continuous function. We say that $v\in \mathcal{L}_s$ is a solution of \eqref{3.e3.1} in the distributional sense if
    \begin{equation*}
\langle v,  (-\Delta)^s \phi \rangle =\langle (f(\cdot,v), \phi \rangle  ~ \text{ for all } \phi\in \mathcal{S}(\mathbb{R}^n);
\end{equation*}
    and if $v$ belong to  the $H^s(\mathbb{R}^n)$ classical fractional Sobolev space in $\mathbb{R}^n$, then we say that $v$ is a weak solution of \eqref{3.e3.1}. We also consider a class of solutions where the fractional Laplacian is defined in a little bit more stronger sense, that is, $v\in \mathcal{L}_s$ is a strong (classical) solution of \eqref{3.e3.1} if $v\in C_{loc}^{2m+\lfloor 2\gamma\rfloor, 2\gamma - \lfloor 2\gamma \rfloor +\varepsilon}(\mathbb{R}^n)$, for some $\varepsilon>0$, and $(-\Delta)^s v = (-\Delta)^m(-\Delta)^{\gamma}v$ is pointwise well-defined and continuous in the whole $\mathbb{R}^n$ satisfying \eqref{3.e3.1}. Moreover, due to the nonlocal feature of $(-\Delta)^\gamma$, we have that $(-\Delta)^\gamma v\in C^{2m}(\mathbb{R}^n)$ (see \cite{Silv}). 
    
    The equation (\ref{3.e3.1}) is closely related to the following integral equation
    \begin{equation}
    \label{3.e3.2}
    v(x)=c_{n,s}\int_{\mathbb{R}^n}\frac{f(y,v(y))}{|x-y|^{n-2s}}dy \text{ in } x\in \mathbb{R}^n,
    \end{equation}
    where $ c_{n,s}=\frac{\Gamma(\frac{n}{2}-s)}{\pi^{\frac{n}{2}}2^{2s}\Gamma(s)}$.  In \cite{CLO1}, it showed that $v\in H^s(\mathbb{R}^n)$ is a solution of (\ref{3.e3.1}) if and only if $v$ is a solution of (\ref{3.e3.2}) provided that $f(\cdot,v)\in L^{\frac{2n}{n+2s}}(\mathbb{R}^n)$. We will show the same result for $v\in \mathcal{L}_s$, a much weaker restriction. The Riesz potential is defined by 
    \begin{equation}
	\label{3.e3.3}
	(\mathcal{I}_{\alpha}v)(x)=c_{n,\frac{\alpha}{2}}\int_{\mathbb{R}^n}\frac{v(y)}{|x-v|^{n-\alpha}}dy,    
    \end{equation}
    where $0<\alpha<n$. The following properties can be found in \cite[Chapter 5]{SKM}:
    \begin{enumerate}
    \item[(P1)] $\|\mathcal{I}_{\alpha}v\|_{L^q(\mathbb{R}^n)}\leq C\|v\|_{L^p(\mathbb{R}^n)}$ if  $1<p<\frac{n}{\alpha}$ and  $\frac{1}{q}=\frac{1}{p}-\frac{\alpha}{n}$;
    \item[(P2)] $\mathcal{I}_{\alpha+\beta}v=\mathcal{I}_{\alpha}\mathcal{I}_{\beta}v$ if  $v\in L^{p}(\mathbb{R}^n)$,  $1<p<\frac{n}{\alpha+\beta}$, $\alpha,\beta>0$ and $\alpha+\beta<n$;
    \item[(P3)] $(-\Delta)^\frac{\alpha}{2}\mathcal{I}_{\alpha} v = v$ if $v\in L^{p}(\mathbb{R}^n)$ and  $1<p<\frac{n}{\alpha}$.
    \end{enumerate}
    
    By the above properties we can see that if $v\in \mathcal{L}_s$ is a smooth solution of \eqref{3.e3.2} such that $f(\cdot,v(\cdot))\in L^p(\mathbb{R}^n)$ with $1<p<\frac{n}{2s}$, then $v$ is a strong  solution of \eqref{3.e3.1}. We need some additional conditions on $f$ or $v$ to get the converse.
  \begin{theorem}
  \label{3.th1}
  Let $f:\mathbb{R}^n\times [0,+\infty)\rightarrow [0,+\infty)$ be a continuous function and let  $v\in \mathcal{L}_s$ be  a nonnegative strong solution of \eqref{3.e3.1}. If either $f(\cdot,t)$ is nondecreasing in $t>0$ and $\inf_{x\in \mathbb{R}^n}f(x,t)>0$ for any $t>0$ or $\lim_{x\rightarrow\infty } v(x)=0$, then $v$ is a nonnegative solution of \eqref{3.e3.2} provided $f(\cdot,v(\cdot)) \in L^p(\mathbb{R}^n)$ for some $1<p<\frac{n}{2s}$.
  \end{theorem}
  \begin{proof}
 In our proof, we follows close the arguments of \cite{ZCCY}. Let $v$ be a nonnegative solution of \eqref{3.e3.1}. The case $0<s<1$ is similar to case $s\geq 1$. Assume $s=1+\gamma$, $0\leq \gamma<  1$. Denote $v_0=:v$, $v_1=:(-\Delta)^\gamma v_0$. Then
 \begin{equation}
 \label{3.e3.4}
 (-\Delta)^\gamma v_0=v_1, ~ -\Delta v_1 = f(x,v_0) \text{ in } \mathbb{R}^n.
\end{equation}  
From \cite{CDQ, GL1}, $v_0$ and $v_1$ are nonnegatives. Let
\begin{equation*}
w_{1,R}(x)=\int_{B_R}G^1_R(x,y)f(y,v_0(y))dy,
\end{equation*}
where $G^1_R(x,y)$ is Green's function on the ball $B_R$:
\begin{equation*}
-\Delta G^1_R(x,y)=\delta(x-y), ~ x,~y\in B_R.
\end{equation*}
It is easy to verify that
\begin{equation}
\label{3.e3.5}
\left\{\begin{aligned}
-\Delta w_{1,R} & = f(y, v_0) &\text{in } B_R,\\
w_{1,R}& = 0 &\text{on } \partial B_R.
	\end{aligned}\right.
\end{equation}
By \eqref{3.e3.4} and \eqref{3.e3.5}, we have
  \begin{equation*}
  \left\{\begin{aligned}
-\Delta(v_1- w_{1,R} )& = 0 &\text{in } B_R,\\
v_1-w_{1,R}& \geq 0 &\text{on } \partial B_R.
	\end{aligned}\right.
  \end{equation*}
    From the Maximum Principle we have that 
    \begin{equation*}
    v_1(x)\geq w_{1,R}(x) \text{ for all } x\in B_R \text{ and for all } R>1.
    \end{equation*}
 Since $f(\cdot,v_0)\in L^p(\mathbb{R}^n)$ for $1<p<\frac{n}{2s}<\frac{n}{2}$, it follows from (P1), dominate convergence theorem and (P3) that $\mathcal{I}_{2}(f(\cdot,v_0))(x)<+\infty$ for all $x\in \mathbb{R}^n$,
 \begin{equation*}
 v_1\geq \mathcal{I}_{2}(f(\cdot,v_0))=:w_1 \text{ and } -\Delta (v_1-w_1)=0 \text{ in }\mathbb{R}^n.
 \end{equation*}
This implies that $v_1=w_1+C$ on whole space $\mathbb{R}^n$, where $C$ is constant nonnegative. Assume $\gamma>0$ and let
 \begin{equation*}
 w_{0,R}(x):=\int_{B_R}G^\gamma_R(x,y)v_1(y)dy,
\end{equation*}  
    where $ G^\gamma_R(x, y)$ is Green's function on the ball $B_R$ (see e.g. \cite{ZCCY}):
    \begin{equation*}
    \begin{aligned}
    (-\Delta)^\gamma G^\gamma_R(x,y) & =\delta(x-y) &\text{ for } x,y\in B_R,\\ G^\gamma_R(x,y) &= 0 &\text{ for } x\text{ or } y \in  \mathbb{R}^n\backslash B_R. 
    \end{aligned}
    \end{equation*}
    Moreover,
     \begin{equation*}
  \left\{\begin{aligned}
(-\Delta)^\gamma(v_0- w_{0,R} )& = 0 &\text{in } B_R,\\
v_0-w_{0,R}& \geq 0 &\text{on }  \mathbb{R}^n\backslash B_R.
	\end{aligned}\right.
  \end{equation*}
    From the Maximum Principle \cite{Silv} we have that 
    \begin{equation*}
    v_0(x)\geq w_{0,R}(x) \text{ for all } x\in B_R \text{ and for all } R>1.
    \end{equation*}
    Since $G^{\gamma}_R(x,y)\geq 0$ and $\lim_{R\rightarrow+\infty}G^\gamma_R(x,y)=c_{n,\gamma}|x-y|^{2\gamma-n}$, we have from Fatou's lemma that
     \begin{equation*}
     v_0\geq \mathcal{I}_{2 \gamma}(v_1)=\mathcal{I}_{2\gamma}(w_1+C)=: w_0,
\end{equation*}      
    which implies $v_1\equiv w_1$ in $\mathbb{R}^n$. By (P1) and a simple computation shows that
    \begin{equation*}
    v_1\in L^{p_1}(\mathbb{R}^n), ~1<p_1=\frac{np}{n-2p}, ~p_1<\frac{n}{2\gamma}.
    \end{equation*}
   By (P3), it follows that
    \begin{equation*}
(-\Delta)^\gamma(v_0- w_0 )\equiv 0 ~\text{ and }~v_0-w_0 \geq 0~ \text{ in }  \mathbb{R}^n.
    \end{equation*}
    From \cite{ZCCY} we have $v_0=w_0+C$, where $C$ is a nonnegative constant (This Liouville Theorem type has been first proved by Bogdan, Kulczycki, and Nowak \cite{BKN}). Since either $\lim_{x\rightarrow \infty}v=0$ or $f(\cdot, v)\geq f(\cdot,C)>0$, we have that $C=0$, and by (P2), 
    \begin{equation*}
    v(x)=w_0=\mathcal{I}_{2\gamma}\mathcal{I}_{2}(f(\cdot,v))(x)=c_{n,s}\int_{\mathbb{R}^n}\frac{f(y,v(y))}{|x-y|^{n-2s}}dy.
    \end{equation*}
   For the case $s=m+\gamma$, with $m\in \mathbb{N}$, one can prove the theorem using iteratively the above same arguments.   
  \end{proof}

  The previous theorem holds if we substitute the condition $f(\cdot,v)\in L^p(\mathbb{R}^n)$ by $0\leq f\in C_{loc}^{0,\varepsilon}$, where $0<\varepsilon<1$. This last condition guarantees 
  \begin{equation*}
  (-\Delta)^sI_{2s}(f(\cdot,v))=f(\cdot,v) \in \mathbb{R}^n
  \end{equation*}
    in the proof of Theorem \ref{3.th1} without using (P3). Note that if $s=1$ and
  \begin{equation}
  \label{3.e.3.6}
 w_1(x):= \mathcal{I}_{2}(f(\cdot,v))(x)<\infty \text{ for all } x\in \mathbb{R}^n,
  \end{equation}
   we can and argue as in \cite[Lemmas 4.1 and 4.2]{Trudg} to obtain \begin{equation*}
    -\Delta w_1=f(\cdot,v) \text{ in } \mathbb{R}^n.
    \end{equation*}
    Moreover, making an inductive argument and considering
  (see \cite{EStein})
    \begin{equation*}
    \int_{\mathbb{R}^n}\frac{c_{n,\alpha}}{|x-z|^{n-\alpha}}\frac{c_{n,\beta}}{|y-z|^{n-\beta}}dy =\frac{c_{n,\alpha+\beta}}{|x-y|^{n-(\alpha+\beta)}} \text{ for } 0<\alpha,\beta<n,~ \alpha+\beta<n,
    \end{equation*}    
 one can show for $s=m\in \mathbb{N}$ that if $ w_m(x):=\mathcal{I}_{2m}(f(\cdot,v))(x)<+\infty$ for all $x\in \mathbb{R}^n$, then 
    \begin{equation*}
    (-\Delta)^mw_m=f(\cdot,v) \text{ in } \mathbb{R}^n.
    \end{equation*}

      In the case $0<s=\gamma<1$, it is not difficult to show that if $w_{\gamma}:=\mathcal{I}_{2\gamma}(f\cdot,v) <+\infty$, then $w_{\gamma}:=\mathcal{I}_{2\gamma}(f\cdot,v)\in C_{loc}^{\lfloor 2\gamma\rfloor, 2\gamma-\lfloor 2\gamma\rfloor+\varepsilon}(\mathbb{R}^n)$. In fact, let be a smooth cutoff function such that $\eta(x)\in [0,1]$ for each $x\in \mathbb{R}^n$, $\text{supp}(v)\subset B_1$, and $\eta=1$ in $B_{7/8}$.
We can write $w_\gamma$ as 
\begin{equation*}
\begin{aligned}
w_\gamma(x) & = c_{n,s}\int_{\mathbb{R}^n}\frac{\eta(y)f(y,v)}{|x-y|^{n-2\gamma}}dy + c_{n,s}\int_{\mathbb{R}^n}\frac{(1-\eta(y))f(y,v)}{|x-y|^{n-2\gamma}}dy \\
&:=w_{1,\gamma}(x)+w_{2,\gamma}(x) \text{ a.e. in } \mathbb{R}^n.
\end{aligned}
\end{equation*}
By the standard Riesz potential theory it follows that
\begin{equation*}
\|w_{1,\gamma}\|_{C^{\lfloor 2\gamma\rfloor, 2\gamma -\lfloor 2\gamma \rfloor +\varepsilon}(\overline{B_{1/2}})}\leq C(\|w_{1,\gamma}\|_{L^{\infty}(\mathbb{R}^n)}+\|\eta f\|_{C^{0,\varepsilon}(\mathbb{R}^n)})\leq C\| f\|_{C^{0,\varepsilon}(\overline{B_1})}.
\end{equation*}
On the other hand, we can argue as in  \cite{Trudg} to obtain $w_{2,\gamma}\in C^{\lfloor 2\gamma\rfloor}_{loc}(\mathbb{R}^n)$. Then, for $x_1,x_2\in \overline{B_{1/2}}$ and $2\gamma\leq 1$, we have
      \begin{equation*}
      \begin{aligned}
      w_{2,\gamma}(x_1)-w_{2,\gamma}(x_1) = &c_{n,s} \int_{\mathbb{R}^n\backslash B_1}(1-\eta(y))f(y,v)[K(x_1-y)-K(x_2-y)]dy\\
	\leq & c_{n,s}\int_{\mathbb{R}^n\backslash B_1}f(y,v)|x_1-x_2||\nabla K(x_3-y)|dy\\
	\leq & C|x_1-x_2|\int_{\mathbb{R}^n\backslash B_1}f(y,v)|x_3-y|^{2\gamma - n-1}dy\\
	\leq & C|x_1-x_2|\int_{\mathbb{R}^n\backslash B_1}f(y,v)|x_1-y|^{2\gamma - n-1}dy\\
	\leq & C|x_1-x_2|\int_{\mathbb{R}^n\backslash B_1}f(y,v)|x_1-y|^{2\gamma - n}dy	\leq  C|x_1-x_2|,
	\end{aligned}
      \end{equation*}
      where $x_3$ is on the line from $x_1$ to $x_2$; and for $2\gamma>1$,
       \begin{equation*}
      \begin{aligned}
      |\nabla w_{2,\gamma}(x_1)-\nabla w_{2,\gamma}(x_1)| \leq  &C\int_{\mathbb{R}^n\backslash B_1}f(y,v)|\nabla K(x_1-y)-\nabla K(x_2-y)|dy\\
		\leq &  C|x_1-x_2|.
	\end{aligned}
      \end{equation*}
      So, $ \|w_{\gamma}\|_{C^{\lfloor 2\gamma\rfloor, 2\gamma -\lfloor 2\gamma \rfloor +\varepsilon}(\overline{B_{1/2}})}<+\infty$.
      
       Finally, if  $w_{\gamma}\in \mathcal{L}_{\gamma}$ then by the Fubini theorem and (P3), 
      \begin{equation*}
      \begin{aligned}
      \int_{\mathbb{R}^n}(-\Delta)^{\gamma}w_{\gamma}(x)\phi(x)dx = & \int_{\mathbb{R}^n}f(x,v(x))(-\Delta)^{\gamma}\mathcal{I}_{2\gamma}\phi(x)dx\\
       = & \int_{\mathbb{R}^n}f(x,v(x))\phi(x)dx \text{ for all } \phi\in C_c^{\infty}(\mathbb{R}^n),
      \end{aligned}
      \end{equation*}
       and therefore, 
      \begin{equation*}
      (-\Delta)^\gamma w_{\gamma}=f(\cdot,v) \text{ in } \mathbb{R}^n.
      \end{equation*}
      
      Therefore, from the above discussed and following the proof of Theorem \ref{3.th1} we have the following result:
      
      \begin{theorem}
     \label{3.th2}
        Let  $f:\mathbb{R}^n\times [0,+\infty)\rightarrow [0,+\infty)$ be a $C_{loc}^{0 ,\varepsilon}$ function and $0<s=m+\gamma<n/2$.
        \begin{enumerate}
        \item[(i)] Any $\mathcal{L}_s \cap C_{loc}^{2m+\lfloor 2\gamma\rfloor ,2\gamma - \lfloor 2\gamma\rfloor+\varepsilon}$ solution of \eqref{3.e3.2} is a strong solution of \eqref{3.e3.2}.
        \item[(ii)] Let  $v\in \mathcal{L}_s$ be  a nonnegative strong solution of \eqref{3.e3.1}. If either $f(\cdot,t)$ is nondecreasing in $t\geq 0$ and $\inf_{x\in \mathbb{R}^n}f(x,t)>0$ for any $t>0$ or $\lim_{x\rightarrow\infty } v(x)=0$, then $v$ is a nonnegative solution of \eqref{3.e3.2}.
        \end{enumerate} 
      \end{theorem}

    The above theorem can be adapted for systems
    \begin{equation}
    \label{3e.1}
    \left\{\begin{aligned}
    (-\Delta)^sv_1=f_1(x,v_1,v_2) \text{ in } \mathbb{R}^n,\\
    (-\Delta)^sv_2=f_2(x,v_1,v_2) \text{ in } \mathbb{R}^n,
    \end{aligned}\right.
\end{equation}     
    such that its pair $(v_1,v_2)$ of solutions is also a pair of solutions of
    \begin{equation}
    \label{3e.2}
    \left\{\begin{aligned}
   v_1(x)=c_{n,s}\int_{\mathbb{R}^n}\frac{f_1(y,v_1,v_2)}{|x-y|^{n-2s}}dy,~x\in \mathbb{R}^n,\\
     v_2(x)=c_{n,s}\int_{\mathbb{R}^n}\frac{f_2(y,v_1,v_2)}{|x-y|^{n-2s}}dy,~x\in \mathbb{R}^n.
    \end{aligned}\right.
\end{equation} 
   \begin{theorem}
     \label{3.th3}
        Let  $f_1,f_2:\mathbb{R}^n\times[0,\infty)\times [0,+\infty)\rightarrow [0,+\infty)$ be $C_{loc}^{0 ,\varepsilon}$ functions and $0<s=m+\gamma<n/2$.
        \begin{enumerate}
        \item[$(i)$] Any pair of $\mathcal{L}_s \cap C_{loc}^{2m+\lfloor 2\gamma\rfloor ,2\gamma - \lfloor 2\gamma\rfloor+\varepsilon}$ solutions of \eqref{3e.2} is a pair of strong solutions of \eqref{3e.1}.
        \item[$(ii)$] Let  $(v_1,v_2)\in \mathcal{L}_s\times \mathcal{L}_s$ be  a pair of nonnegative strong solutions of \eqref{3e.1}. If either $f_j(\cdot,r,t)$ is nondecreasing in $r\geq 0$ and $t\geq 0$ for $j=1,2$, and $\inf_{x\in \mathbb{R}^n}(f_1+f_2)(x,r,t)>0$ for any $r+t>0$ or $\lim_{x\rightarrow\infty } v_j(x)=0$ for $j=1,2$, then $(v_1,v_2)$ is a pair of nonnegative solutions of \eqref{3e.2}.
        \end{enumerate} 
      \end{theorem}

    \subsection{Problems involving the conformally operator $P^g_s$}
	\label{subsec3.2}        
    \

    The operator $P_s:=P_s^{g_{\mathbb{S}^n}}$ can be seen more concretely on $\mathbb{R}^n$ using stereographic projection (see for more details \cite{Bra, CM}).  Let $\zeta$ be an arbitrary point on ${\S}^{n}$, which we will rename the north pole $N$ and let $\mathcal{F}^{-1}:{\S}^n\backslash\{N\} \rightarrow \mathbb{R}^n$ be the stereographic projection. Then
    \begin{equation}
    \label{e4.1}
    P_s(u)(\mathcal{F} (y)) = \xi_s(y)^{\frac{n+2s}{n-2s}}(-\Delta)^s(\xi_s(y)(u\circ\mathcal{F})(y)),~y\in\mathbb{R}^n, \text{ for }~u\in C^{\infty}({\mathbb{S}^n}),
    \end{equation}
    where $s\in (0,n/2)$, $\xi_s$ is defined in (\ref{e2.1}) and $(-\Delta)^s$ is the fractional Laplacian operator. Moreover, since the set $C^{\infty}(\mathbb{S}^n)$ is dense in $H^s(\mathbb{S}^n,g_{\mathbb{S}^n})$, it follows that the above identity holds in $H^s(\mathbb{S}^n,g_{\mathbb{S}^n})$.

Consider the problem
    \begin{equation}
   \label{4.e3.1}
    -L_g^s u=P_s^g u-R_s^g=f(u) \text{ in } \mathbb{S}^n,
    \end{equation}
    where $0<2s<n$, $n>2$, $g\in [g_{\mathbb{S}^n}]$, $R_s^g$ is the $s$-curvature and $f:\mathbb{R}\rightarrow \mathbb{R}$ is a continuous function.

From here on we will write the metric $g\in [g_{\mathbb{S}^n}]$ as $g=\varphi^{\frac{4}{n-2s}}g_{\mathbb{S}^n}$ where  $\varphi\in C^{\infty}(\mathbb{S}^n)$ is a positive function. By \eqref{r1.1}, the above problem is equivalent to 
\begin{equation}
\label{4.e3.2}
P_s(\varphi u)=\varphi^{\frac{n+2s}{n-2s}}(f(u)+R^g_s u) \text{ in } \mathbb{S}^n.
\end{equation}

 Now, let $u$ be a smooth solution of (\ref{4.e3.1}). We define 
\begin{equation*}
v(y)= \xi_s(y)\varphi(\mathcal{F}(y)) u(\mathcal{F}(y)), ~ y\in \mathbb{R}^n.
\end{equation*}
 Then
	\begin{equation}
	\label{e4.2}
	v\in L^{\frac{2n}{n-2s}}(\mathbb{R}^n)\cap L^{\infty}(\mathbb{R}^n)\cap \mathcal{L}_s ~\text{ and }~ \lim_{x\rightarrow \infty}v(x)=0.
\end{equation}	    
 By (\ref{e4.1}), (\ref{4.e3.1}) and (\ref{4.e3.2}), $v$ resolves
   \begin{equation}
    \label{4.e3.3}
    (-\Delta)^s v  =  [\xi_s (\varphi\circ \mathcal{F})]^{\frac{n+2s}{n-2s}}\left( f\left(\frac{v}{\xi_s (\varphi\circ \mathcal{F})}\right) +( R^g_s\circ\mathcal{F})\frac{v}{\xi_s (\varphi\circ \mathcal{F})} \right)   \text{ in } {\mathbb{R}}^{n}.
    \end{equation}
 Moreover, if $f(t)+R^g_s t$ is a nonnegative function in $t\geq 0$, then  by Theorem \ref{3.th1}, $v$ is a solution of
 \begin{equation}
 \label{4.e3.4}
 v(y)=c_{n,s}\int_{\mathbb{R}^n}\frac{ [\xi_s (\varphi\circ \mathcal{F})]^{\frac{n+2s}{n-2s}}\left( f(\frac{v}{\xi_s (\varphi\circ \mathcal{F})}) +( R^g_s\circ\mathcal{F})\frac{v}{\xi_s (\varphi\circ \mathcal{F})} \right)}{|y-z|^{n-2s}}dz,  ~y\in \mathbb{R}^n,
 \end{equation}
where$c_{n,s}=\frac{\Gamma(\frac{n}{2}-s)}{\pi^{\frac{n}{2}}2^{2s}\Gamma(s)}$. Thus, by the stereographic projection, $u$ solves
\begin{equation}
 \label{4.e3.5}
 u(\zeta) =c_{n,s} \int_{\mathbb{S}^n}\frac{f(u)(\omega)+R^g_s(\omega) u(\omega)}{\mathcal{K}(\zeta, \omega)}d\upsilon_g^{(\omega)},~\zeta\in \mathbb{S}^n,
\end{equation}   
 where $\mathcal{K}(\zeta, \omega)=\varphi(\zeta,\omega)|\zeta-\omega|^{n-2s}$ and $|\cdot|$ is the Euclidean distance in $\mathbb{R}^{n+1}$. On the other hand, by the stereographic projection, \eqref{4.e3.4}, the  property (P3) and (\ref{e4.1}), any continuous positive solution of \eqref{4.e3.5} is also a continuous positive solution of \eqref{4.e3.1}.
 In the case $s> 1$, we will use the expression \eqref{4.e3.4} to determine the symmetry of the solutions of (\ref{4.e3.3}).

Therefore, we can write the above mentioned as:
\begin{theorem}
\label{th3.2.1}
Let $g$ be a conformal metric in $[g_{\mathbb{S}^n}]$ and $0<2s<n$. 
\begin{enumerate}
\item[(i)] Any continuous nonnegative solution of \eqref{4.e3.5} is also solution of \eqref{4.e3.1}.
\item[(ii)] If $f(t)+R^g_s t$ is continuous nonnegative in $\mathbb{S}^n$ and $t\geq 0$, then any nonnegative smooth solution of \eqref{4.e3.1} is also solution of \eqref{4.e3.5}.
\end{enumerate}
\end{theorem}

 \subsection{Proof of Theorem \ref{th1.5} for the case $0<s<1$}
 \label{subsec3.3}    
 \
 
 Let $g=g_{\mathbb{S}^n}$ be the standard metric on $\mathbb{S}^n$ and $0<s<1$ such that $2s<n$.  Let $u$ be a positive solution of \eqref{4.e3.1}. Taking an arbitrary point  $\zeta_0\in \mathbb{S}^n$, \eqref{4.e3.3} becomes
        \begin{equation}
    \label{p4.1}
    (-\Delta)^s v  =  h_s\left(\frac{v}{\xi_s}\right)v^{\frac{n+2s}{n-2s}},~ v>0   \text{ in } {\mathbb{R}}^{n},
    \end{equation}
    where $v=\xi_s(u\circ\mathcal{F})$, $\mathcal{F}^{-1}$ is the stereographic projection with $\zeta_0$ being the north pole and 
    \begin{equation*}
        h_s(t)=t^{-\frac{n+2s}{n-2s}}\left(f(t)+d_{n,s}t\right), ~t>0 ~ \text { and }~ d_{n,s}=\frac{\Gamma(\frac{n}{2}+s)}{\Gamma(\frac{n}{2}-s)}.
    \end{equation*}  
    
    First we will prove Theorem \ref{th1.5} for case $s\in(0,1)$. Given $t\in \mathbb{R}$ we set 
\begin{equation}
\label{e3.Q}
Q_t=\{ y\in {\mathbb{R}}^n;~ y^1<t\}\mbox{ and } U_t=\{y\in {\mathbb{R}}^n; ~y^1=t\},
\end{equation}
    where $y_t:=(2t-y^1,y')$ is the image of point $y=(y^1,y')$ under of reflection through the hyperplane $U_t$. We define the reflected function by $v^t(y):=v(y_t)$. We define the following functions for $t\geq 0$ and $\varepsilon>0$:
    \begin{equation} \label{e4.3}
    w_{\varepsilon}^t(y)=\begin{cases}
    					 (v^t(y)-v(y)-\varepsilon)^+,~y\in Q_t,\\
    					 (v^t(y)-v(y)+\varepsilon)^-,~y\in Q_t^c,
    					\end{cases} \text{and }
    	w^t(y)=\begin{cases}
    					 (v^t(y)-v(y))^+,~y\in Q_t,\\
    					 (v^t(y)-v(y))^-,~y\in Q_t^c,
    					\end{cases}
    \end{equation}
	where $(v^t-v)^+=\max\{v^t-v,0\}$ and $(v^t-v)^-=\min\{v^t-v,0\}$. The following result is a Sobolev inequality type for the function $w_\varepsilon^t$, which will be used to demonstrate the symmetry of $v$.
	\begin{lemma}
	\label{l4.1}
	Under the assumptions of Theorem \ref{th1.5} for $s\in (0,1)$, there exists a positive constant $C=C(n,s)>0$ such that
	\begin{equation}\label{e4.5}
	\left(\int_{Q_t} |w_{\varepsilon}^t|^{\frac{2n}{n-2s}}dy\right)^{\frac{n-2s}{n}}\leq C \int_{Q_t}(-\Delta)^s(v^t-v)(v^t-v-\varepsilon)^+dy,
	\end{equation}
	for each $t>0$.
	\end{lemma}
\begin{proof}
See \cite[Lemma 4.2]{FW}.
\end{proof}

	{\it Proof of Theorem \ref{th1.5}}. Let $u$ be the solution of problem \eqref{p1.5}. We take an arbitrary point $\zeta_0\in \mathbb{S}^n$, and let $\mathcal{F}^{-1}:{\mathbb{S}^n}\backslash \{ \zeta\}\rightarrow \mathbb{R}^n$ be the stereographic projection. We define $v=\xi_s(u\circ\mathcal{F})$ in $\mathbb{R}^n$. The proof is carried out in three steps. In the first step we show that
    \begin{equation*}
    \Lambda:=\inf\{t>0;~ v \geq v^\mu \text{ in } Q_\mu, \forall \mu\geq t\}.
    \end{equation*}
    is well-defined, i.e. $\Lambda<+\infty$. The second step consists in proving that if $\Lambda=0$. In the third step we conclude the proof. \
    
   \begin{Step}\label{step:5.1}
    $\Lambda<+\infty$.
   \end{Step}
    
    For $\varepsilon>0$ and $t>0$, we consider the functions $w_{\varepsilon}^t$ and $w^t$ defined by (\ref{e4.3}). Using the same arguments as in (\ref{e2.4}), we have
	\begin{equation}
	\label{e4.9}
	(-\Delta)^s(v^t-v)(y)\leq Cv^t(y)^{\frac{4s}{n-2s}}(v^t-v)(y)~\text{ for } v^t(y)\geq v(y), ~y\in Q_t.
	\end{equation}
     Following the computations in (\ref{e2.5}), we use Fatou's lemma, Lemma \ref{l4.1}, (\ref{e4.9}), H\"{o}lder's and Sobolev's inequalities, and dominate convergence theorem to obtain
     \begin{align}
     \label{e4.15}
     \left(\int_{Q_t} |w^t|^{\frac{2n}{n-2s}}dy\right)^{\frac{n-2s}{n}} 
 \leq \phi(t) \left(\int_{Q_t} |w^t|^{\frac{2n}{n-2s}}dy\right)^{\frac{n-2s}{n}},
     \end{align} 
where $\phi(t)=C(\int_{Q_t}(v^t)^{\frac{2n}{n-2s}}dy)^{\frac{2s}{n}}$. Since $v^{\frac{2n}{n-2s}}\in L^1(\mathbb{R}^n)$, then $\lim_{t\rightarrow +\infty}\phi(t)=0$. Thus, choosing $t_1>0$ large sufficiently such that $\varphi(t_1)<1$, we have from (\ref{e4.15})
    \begin{equation*}
    \int_{Q_t}|w^t|^{\frac{2n}{n-2s}} dy=0,~~~\text{ for all } t>t_1.
    \end{equation*}
     This implies $(v^t-v)^+\equiv 0$ in $Q_t$ for $t>t_1$. Therefore $\Lambda$ is well defined, i.e. $\Lambda < +\infty$.\
  
   \begin{Step}\label{step:5.2}
  $\Lambda=0$.
   \end{Step}
     
    Assume $\Lambda>0$. By definition of $\Lambda$ and continuity of the solution, we get $v\geq v^{\Lambda}$ and $\xi_s>\xi_s^{\Lambda}$ in $Q_{\Lambda}$. 
     
     Suppose there is a point $y_0\in Q_{\Lambda}$ such that $v(y_0)=v^{\Lambda}(y_0)$. Using the fact of $h$ is decreasing, we have
     \begin{equation}
     \label{e4.15.1}
     \begin{aligned}
     (-\Delta)^s v(y_0)-(-\Delta)^s v^{\Lambda}(y_0)&= h_s\left( \frac{v(y_0)}{\xi_s(y_0)} \right)v(y_0)^{\frac{n+2s}{n-2s}}-h_s\left( \frac{v^{\Lambda}(y_0)}{\xi_s^{\Lambda}(y_0)} \right)v^{\Lambda}(y_0)^{\frac{n+2s}{n-2s}}\\
     & = \left[ h_s\left( \frac{v(y_0)}{\xi_s(y_0)} \right)- h_s\left( \frac{v(y_0)}{\xi_s^{\Lambda}(y_0)} \right) \right]v(y_0)^{\frac{n+2s}{n-2s}}>0.
     \end{aligned}
     \end{equation}
On the other hand,
\begin{equation*}
\begin{aligned}
&(-\Delta)^s v(y_0)-(-\Delta)^s v^{\Lambda}(y_0)\\
& ~ = -\int_{Q_{\Lambda}}\frac{v(z)-v(z_{\Lambda})}{|y_0-z|^{n+2s}}dz-\int_{Q_{\Lambda}^c}\frac{v(z)-v(z_{\Lambda})}{|y_0-z|^{n+2s}}dz\\
& ~ = - \int_{Q_{\Lambda}}(v(z)-v(z_{\Lambda}))\left(\frac{1}{|y_0-z|^{n+2s}}-\frac{1}{|y_0-z_{\Lambda}|^{n+2s}}\right)dz\leq 0,
\end{aligned}
\end{equation*}    
    which contradicts (\ref{e4.15.1}). As a sequence, $v > v^\Lambda$ in $Q_{\Lambda}$.
    
  We can choose a compact $K \subset Q_{\Lambda}$ and a number $\delta > 0$ such that $\forall  t \in (\Lambda-\delta, \Lambda )$ we have $K \subset Q_t$ and
    \begin{equation}
    \label{e4.16}
    \phi(t)=C\left(\int_{Q_t\backslash K}(v^t)^{\frac{2n}{n-2s}}dy\right)^{\frac{2s}{n}}<\frac{1}{2}.
    \end{equation}
    On the other hand, there exists $0 < \delta_1 < \delta$, such that 
    \begin{equation}
    \label{e4.17}
    v > v^t ,\text{ in } K~ \forall t \in (\Lambda - \delta_1 , \Lambda ).
    \end{equation}
Using (\ref{e4.16}), (\ref{e4.17}) and following as in Step \ref{step:5.1}, in (\ref{e4.15}), since the integrals are over $Q_t \backslash K$, we see that
$(v^t - v)^+ \equiv 0$ in $Q_t \backslash K$. By (\ref{e4.17}) we get $v > v^t$ in $Q_t$ for all $t \in (\Lambda-\delta_1 , \Lambda )$. This is a contradiction with the definition of $\Lambda$.\

    \begin{Step}\label{step:5.3}
    Symmetry.
   \end{Step} 
	
	By Step \ref{step:5.2} we have $\Lambda=0$ for all directions. This implies that $v$ is radially symmetrical in $\mathbb{R}^n$. By definition of $v$, we obtain that $u$ is constant on every $(n-1)$-sphere whose elements $\omega \in {\mathbb{S}^n}$ satisfy $|\omega-N|= \text{constant}$. Since $\zeta\in {\mathbb{S}^n}$ is arbitrary on ${\mathbb{S}^n}$, $u$ is constant.	
	\hfill $\square$
\

Before continuing with the proof of Theorem \ref{th1.5} for the case $s>1$, we want to comment on the above proof. In \cite{XY3},  the function $f$ was required to be a nondecreasing continuous to get the symmetry of the positive solutions on $\mathbb{R}^n $ of \eqref{p4.1}. Thanks to Lemma \ref{l4.1} and the proof of Step \ref{step:5.2}, we don’t need this assumption here.

 \subsection{Proof of Theorem \ref{th1.5} for the case $1<s<n/2$}
 \label{subsec3.4}    
 \

 Let $g=g_{\mathbb{S}^n}$ be the standard metric on $\mathbb{S}^n$.  Let $u$ be a positive solution of \eqref{4.e3.1}. Taking an arbitrary point  $\zeta_0\in \mathbb{S}^n$, $v=\xi_s(u\circ\mathcal{F})$ resolves (\ref{p4.1}),  where $\mathcal{F}^{-1}$ is the stereographic projection with $\zeta_0$ being the north pole. Assume that $h_s(t)$ and $h_s(t)t^{\frac{n+2s}{n-2s}}$ are nonincreasing and nondecreasing in $t>0$, respectively. Then, for $0<r<t$,
 \begin{equation*}
 h_s(r)\geq h_s(t) =\frac{h_s(t)t^{\frac{n+2s}{n-2s}}}{t^{\frac{n+2s}{n-2s}}}\geq \frac{h_s(r)r^{\frac{n+2s}{n-2s}}}{t^{\frac{n+2s}{n-2s}}}.
\end{equation*}  
  Taking $t\rightarrow +\infty$ in the above inequality, we have $h_s(r)\geq 0$ for all $r>0$. It follows from (\ref{p4.1}) and Theorem \ref{3.th1} that $v$ resolves (up to constants)
  \begin{equation}
  \label{p4.1i}
  v(y)=\int_{\mathbb{R}^n}\frac{h_s(\frac{v}{\xi_s}) v^{\frac{n+2s}{n-2s}}}{|y-z|^{n-2s}}dz,~y\in \mathbb{R}^n.
  \end{equation}

     To prove Theorem \ref{th1.5} in the case $s\in(1,n/2)$ we need the following result, which is essential to show the symmetry of the solutions of (\ref{p4.1i}).
     \begin{lemma}
     \label{lm3.2}
     \begin{equation*}
     \begin{aligned}
     v^t(y)-v(y) = \int_{Q_t} & \left( \frac{1}{|y-z|^{n-2s}} -  \frac{1}{|y_t-z|^{n-2s}}\right) \\
     & \times \left( h_s\left( \frac{v^t(z)}{\xi_s^t(z)}\right)(v^t(z))^{\frac{n+2s}{n-2s}} - h_s\left( \frac{v(z)}{\xi_s(z)}\right)v(z)^{\frac{n+2s}{n-2s}} \right)dz.
     \end{aligned}
     \end{equation*}
     \end{lemma}
     \begin{proof}
     See the proof of Lemma 1 in \cite{XY3}.
     \end{proof}
     
     {\it Proof of Theorem \ref{th1.5}}.  Let $u$ be the solution of problem \eqref{p1.5}. We take an arbitrary point $\zeta_0\in \mathbb{S}^n$, and let $\mathcal{F}^{-1}:{\mathbb{S}^n}\backslash \{ \zeta_0\}\rightarrow \mathbb{R}^n$ be the stereographic projection. We define $v=\xi_s(u\circ\mathcal{F})$ in $\mathbb{R}^n$. It follows that $v$ solves \eqref{p4.1i}. Similarly as in the proof of this theorem for case $s\in (0,1)$, this one consists of three steps. Let     \begin{equation*}
    \Lambda:=\inf\{t>0;~ v \geq v^\mu \text{ in } Q_\mu, \forall \mu\geq t\},
    \end{equation*}
    where $Q_t$ is defined by \eqref{e3.Q}.
 \begin{Step}
\label{step:5.1.1} 
  $\Lambda<+\infty.$
 \end{Step}
    As in Step \ref{step:1.1}, we assume that there is a $t>0$ such that: if $v^t(y)\geq v(y)$ for some $y\in Q_t$, then
    \begin{equation}
    \label{e6.1}
   h_s\left(\frac{v^t}{\xi_s^t}\right)(v^t)^{\frac{n+2s}{n-2s}}-h_s\left(\frac{v}{\xi_s}\right)v^{\frac{n+2s}{n-2s}}
\leq  C(v^t)^{\frac{4s}{n-2s}}(v^t-v),  
	\end{equation}
or if $v^t(y)\leq v(y)$, then 
\begin{equation}
\label{e6.2}
\begin{aligned}
h_s\left(\frac{v^t}{\xi_s^t}\right)(v^t)^{\frac{n+2s}{n-2s}}-h_s\left(\frac{v}{\xi_s}\right)v^{\frac{n+2s}{n-2s}}
& < h_s\left(\frac{v^t}{\xi_s}\right)(v^t)^{\frac{n+2s}{n-2s}}-h_s\left(\frac{v}{\xi_s}\right)v^{\frac{n+2s}{n-2s}}\\
& \leq h_s\left(\frac{v}{\xi_s}\right)v^{\frac{n+2s}{n-2s}}-h_s\left(\frac{v}{\xi_s}\right)v^{\frac{n+2s}{n-2s}}\leq 0,
\end{aligned}
\end{equation}
     where the last inequality in (\ref{e6.1}) is consequence of $h_s(\frac{v}{\xi})\in L^{\infty}(\mathbb{R}^n)$. 
From Lemma \ref{lm3.2} and the above inequalities, we have
\begin{equation}
\label{e5.61}
 [v^t(y)-v(y)]^+ \leq C \int_{\mathbb{R}^n}  \frac{1}{|y-z|^{n-2s}}
(v^t(z))^{\frac{4s}{n-2s}}[v^t(z)-v(z)]^+ \chi_{Q_t}(z)dz,
\end{equation}
where $|y-z|<|y_t-z|$ for all $z\in Q_t$ and $\chi_{Q_t}$ is the indicator function.

By the Hardy-Littlewood-Sobolev \cite{Lieb} and H\"{o}lder inequalities, it follows that
\begin{equation*}
\begin{aligned}
\int_{Q_t} & [(v^t(y)-v(y))^+]^q dy \\
& \leq C\int_{\mathbb{R}^n}\left( \int_{\mathbb{R}^n}  \frac{1}{|y-z|^{n-2s}}
(v^t(z))^{\frac{4s}{n-2s}}(v^t(z)-v(z))^+ \chi_{Q_t}(z)dz \right)^q dy \\
& \leq C \left( \int_{\mathbb{R}^n}[ (v^t(z))^{\frac{4s}{n-2s}}(v^t(z)-v(z))^+ \chi_{Q_t}(z)]^r dz \right)^{\frac{q}{r}}\\
& \leq C \left(\int_{Q_t} (v^t)^{\frac{4s qr}{(n-2s)(q-r)}}dy \right)^{\frac{q-r}{r}}  \int_{Q_t} [(v^t(y)-v(y))^+]^q dy,
\end{aligned}
\end{equation*}
where $1<r<q$ and $(q-r)/(rq) = 2s/n$. Since $v \in L^{\frac{2n}{n-2s}}(\mathbb{R}^n)$, we can choose a $t_1\in \mathbb{R}$ such that
\begin{equation*}
C \left(\int_{Q_t} (v^t)^{\frac{2n}{n-2s}}dy \right)^{\frac{q-r}{r}}<\frac{1}{2}
\end{equation*}
for all $t>t_1$. Then we conclude $(v^t-v)^+\equiv 0$ in $Q_t$ for all $t>t_1$. Therefore $\Lambda$ is well defined.
\begin{Step}
\label{step:5.1.2} 
$\Lambda=0$.
\end{Step}
Assume $\Lambda>0$. By definition of $\Lambda$, we have $v\geq v^{\Lambda}$ in $Q_t$. We can infer from (\ref{e6.2}) and Lemma \ref{lm3.2} that $v>v^{\Lambda}$ in $Q_{\Lambda}$. The conclusion follows from the arguments used in the last part of the proof of Step \ref{step:5.2}.
\begin{Step}
\label{step:5.1.3} 
Symmetry.
\end{Step}
This step is similar to Step \ref{step:5.3}.
\hfill $\square$    
\ 

\subsection{Uniqueness results for systems}
\label{subsec3.5}    
\

Based on the arguments presented in the above subsections we can extend Theorem \ref{th1.5} for systems.  Let $(u_1,u_2)$ be a pair of positive solutions of problem (\ref{p1.6}). We take an arbitrary point $\zeta_0\in \mathbb{S}^n$, and let $\mathcal{F}^{-1}:{\mathbb{S}^n}\backslash \{ \zeta_0\}\rightarrow \mathbb{R}^n$ be the stereographic projection. We define
	\begin{equation*}
	v_1(y)=\xi_s(y)u_1(\mathcal{F}(y)),~v_2(y)=\xi_s(y)u_2(\mathcal{F}(y)), ~y\in \mathbb{R}^n.
\end{equation*}	    
    Then we have that
    \begin{equation}
	\begin{aligned}
	\label{e4.19}
	v_1,v_2 \in L^{\frac{2n}{n-2s}}(\mathbb{R}^n)\cap L^{\infty}(\mathbb{R}^n)\cap\mathcal{L}_s ~\text{ and } ~\lim_{x\rightarrow \infty}v_1(x)=\lim_{x\rightarrow \infty}v_2(x)=0.
	\end{aligned}
\end{equation} 
By (\ref{p1.6}) and (\ref{e4.1}), $(v_1,v_2)$ solves
    \begin{equation}
    \label{p4.2}
    \left\{\begin{aligned}
    (-\Delta)^s v_1  =  h_{1s}\left(\frac{v_1}{\xi_s},\frac{v_2}{\xi_s}\right)v_1^{\frac{n+2s}{n-2s}}  \text{ in } {\mathbb{R}}^n,\\
    (-\Delta)^s v_2  =  h_{2s}\left(\frac{v_1}{\xi_s},\frac{v_2}{\xi_s}\right)v_2^{\frac{n+2s}{n-2s}}  \text{ in } {\mathbb{R}}^n,
    \end{aligned}\right.
    \end{equation}
    where 
    \begin{equation*}
    h_{js}(t_1,t_2)=t_j^{-\frac{n+2s}{n-2s}}(f_j(t_1,t_2)+d_{n,s} t_j), ~t_j>0,~j=1,2, \text{ and } d_{n,s}=\frac{\Gamma(\frac{n}{2}+s)}{\Gamma(\frac{n}{2}-s)}.
    \end{equation*}
	Following the proof of Theorem \ref{th1.2}, we have that $h_{1s}$, $h_{2s}$  are nonnegatives functions. Then by Theorem \ref{3.th3}, we have that $(v_1,v_2)$ is a par of positive solutions of  (up to constants)  
    \begin{equation}
    \label{p4.2i}
    \left\{\begin{aligned}
     v_1(y)  =\int_{\mathbb{R}^n}\frac{  h_{1s}\left(\frac{v_1(z)}{\xi_s(z)},\frac{v_2(z)}{\xi_s(z)}\right)v_1(z)^{\frac{n+2s}{n-2s}}}{|y-z|^{n-2s}} dz, ~y\in \mathbb{R}^n,\\
     v_2(y)  = \int_{\mathbb{R}^n} \frac{ h_{2s}\left(\frac{v_1(z)}{\xi_s(z)},\frac{v_2(z)}{\xi_s(z)}\right)v_2(z)^{\frac{n+2s}{n-2s}} }{|y-z|^{n-2s}}dz, ~y  \in \mathbb{R}^n.
    \end{aligned}\right.
    \end{equation}
    
    We define the reflected functions $v_1^t(y):=v_1(y_t)$ and $v_2^t(y):=v_2(y_t)$ with respect to the hyperplane $U_t$. The following result is a version of Lemma \ref{lm3.2} for systems.
  \begin{lemma}
  \label{lm3.3}
  \begin{align*}
  v_j^t(y)-v_j(y)=\int_{Q_t} & \left( \frac{1}{|y-z|^{n-2s}} -  \frac{1}{|y_t-z|^{n-2s}}\right) \left[ h_{js}\left( \frac{v_1^t(z)}{\xi_s^t(z)},\frac{v_2^t(z)}{\xi_s^t(z)}\right)(v_j^t(z))^{\frac{n+2s}{n-2s}} \right. \\
&  \left. - h_{js}\left( \frac{v_1(z)}{\xi_s(z)}, \frac{v_2(z)}{\xi_s(z)} \right)v_j(z)^{\frac{n+2s}{n-2s}} \right]dz, ~j=1,2. 
  \end{align*}
  \end{lemma}

   {\it Proof of Theorem \ref{th1.6}.}  Let $(u_1,u_2)$ be a pair of solutions of problem \eqref{p1.6}. We take an arbitrary point $\zeta_0\in \mathbb{S}^n$, and let $\mathcal{F}^{-1}:{\mathbb{S}^n}\backslash \{ \zeta_0\}\rightarrow \mathbb{R}^n$ be the stereographic projection. We define $v_i=\xi_s(u_j\circ\mathcal{F})$ in $\mathbb{R}^n$, $j=1,2$. It follows that $(v_1,v_2)$ solves \eqref{p4.2i}. Similarly as in the proofs of Theorem \ref{th1.5} and Lemma \ref{l2.3}, this one consists of three steps. Let  \begin{equation*}
    \Lambda:=\inf\{t>0;~ v_j \geq v_j^\mu \text{ in } Q_\mu, \forall \mu\geq t,~j=1,2\},
    \end{equation*}
    where $Q_t$ is defined by \eqref{e3.Q}.
    \begin{Step}
    $\Lambda<+\infty$.
    \end{Step}
    Following the computations in \eqref{e2.13} and \eqref{e2.14}, we have for $j=1,2$,
    \begin{equation*}
     h_{js}\left( \frac{v_1^t}{\xi_s^t},\frac{v_2^t}{\xi_s^t}\right)(v_j^t)^{\frac{n+2s}{n-2s}}  - h_{js}\left( \frac{v_1}{\xi_s}, \frac{v_2}{\xi_s} \right)v_j^{\frac{n+2s}{n-2s}}\leq C (v_{j}^t)^{\frac{4s}{n-2s}}[(v_1^t-v_1)^+ + (v_2^t-v_2)^+].
    \end{equation*}
   By Lemma \ref{lm3.3} and arguing as Step \ref{step:5.1.1}, we obtain
   \begin{equation*}
   \int_{Q_t}  [(v_j^t-v_j)^+]^q dy\leq C \left( \int_{Q_t}(v_j^t)^{\frac{2n}{n-2s}}dy\right)^{\frac{2sq}{n}}\int_{Q_t}  [(v_1^t-v_1)^+ + (v_2^t-v_2)^+]^q dy,
   \end{equation*}
  where $q>1$. Since $v_j \in L^{\frac{2n}{n-2s}}(\mathbb{R}^n)$, we can choose a $t_1\in \mathbb{R}$ such that
\begin{equation*}
C \left(\int_{Q_t} (v_j^t)^{\frac{2n}{n-2s}}dy \right)^{\frac{2sq}{n}}<\frac{1}{4}
\end{equation*}
for all $t>t_1$. Then we conclude $(v_j^t-v_j)^+\equiv 0$ in $Q_t$ for all $t>t_1$ and $j=1,2$. Therefore $\Lambda$ is well defined.

The other two steps ($\Lambda=0$ and symmetry) are similar to Steps \ref{step:5.1.2} and \ref{step:5.1.3}.

\hfill $\square$

     \section{Proof of Theorems \ref{th1.3}, \ref{th1.7} and \ref{th1.9}}
     \label{sec4}
     
   We start with the following lemma that shows some estimates for the solutions of (\ref{p1.8}) and (\ref{p1.9}). Recall that (\ref{g1.40}) holds for all $0<s<n/2$ and $g=\varphi^{\frac{4}{n-2s}}g_{M}$ denotes the conformal metric on $M=\mathbb{S}^n_+$ or $M=\mathbb{S}^n$, where $\varphi\in C^{\infty}(M)$ is a positive function on $M$. In the case when $M=\mathbb{S}^n_+$, $\lambda_{1,1,g}=\lambda_{1,g}$ is the first nonzero Neumann eigenvalue  of $-L^1_g=-\Delta_g$ on $(\mathbb{S}^n_+,g)$.

\begin{lemma}
 \label{gl2.1}
Let $(M,g)= (\mathbb{S}^n_+,g)$ or$(\mathbb{S}^n,g)$, $g\in [g_M]$, $n>2$ and $0<2s<n$.
\begin{enumerate}
\item[(i)] Suppose that $\tilde{f}$ satisfies (H1)-(H3). Then there exists a positive constant $C=C(n,s,\tilde{f})$ such that 
\begin{equation*}
(\lambda_{1,s,g}+\lambda)\|u-\overline{u}\|^2_{L^2(M)} \leq C( \|u\|^{\mu}_{L^{\infty}(M)}+\|u\|^{p-1}_{L^{\infty}(M)})\|u-\overline{u}\|^2_{L^2(M)}
\end{equation*}
for any nonnegative solution $u$ of \eqref{p1.8}, where $\overline{u}=\avint u~d\upsilon_g$.
\item[(ii)] Suppose that $\tilde{f}_{j}$ satisfies (F1)-(F4) for each $j=1,2$.  Then there exists a positive constant $C=C(n,s,\tilde{f}_1,\tilde{f}_2)$ such that
\begin{equation*}
\begin{aligned}
& (\lambda_{1,s,g} - \|{A}\|)  \left(\|u_1-\overline{u}_1\|^2_{L^2(M)} + \|u_2-\overline{u}_2\|^2_{(M)} \right) \\
& \leq C\left(\sum_{l=1}^{2}\|u_l\|^{\mu}_{L^{\infty}(M)}+\|u_l\|^{p-1}_{L^{\infty}(M)}\right)\left( \|u_1-\overline{u}_1\|^2_{L^2(M)}+ \|u_2-\overline{u}_2\|^2_{L^2(M)}\right)
\end{aligned}
\end{equation*}
for any pair $(u_1,u_2)$ of positive solutions of \eqref{p1.9}.
\end{enumerate}
\end{lemma}
\begin{proof}
{\it Case} (i). Let $u$ be a nonnegative solution of \eqref{p1.8}. By the assumptions on $\tilde{f}$, we have that
\begin{equation}
\label{g2.3.1}
[\tilde{f}(u)-\tilde{f}(\overline{u})](u-\overline{u})\leq C( \|u\|^{\mu}_{L^{\infty}(M)}+\|u\|^{p-1}_{L^{\infty}(M)}) (u-\overline{u})^2.
\end{equation}
 Multiplying \eqref{p1.8} by $(u-\overline{u})$ and integrating, we have
\begin{equation}
\label{g2.4}
\int_{M}(u-\overline{u})L^g_s u~ d\upsilon_g +\lambda\int_{M}u(u-\overline{u})~d\upsilon_g = \int_{M}\tilde{f}(u)(u-\overline{u})~d\upsilon_g.
\end{equation}
Clearly we have
\begin{equation}
\label{g2.5}
\int_{M}(u-\overline{u})L^g_s\overline{u}~d\upsilon_g=\int_{M}\overline{u}(u-\overline{u})~d\upsilon_g = \int_{M}\tilde{f}(\overline{u})(u-\overline{u})~d\upsilon_g=0.
\end{equation}
By (\ref{g2.4}) and (\ref{g2.5}) we have
\begin{equation*}
\int_{M}(u-\overline{u})L^g_s(u-\overline{u})~d\upsilon_g + \lambda \int_{M}(u-\overline{u})^2d\upsilon_g = \int_{M}[\tilde{f}(u)-\tilde{f}(\overline{u})](u-\overline{u})~d\upsilon_g .
\end{equation*}
The first part of the theorem follows from the above equality, (\ref{g2.3.1}), (\ref{g1.40}) for $M=\mathbb{S}^n$ or $\lambda_{1,g}$ for $M=\mathbb{S}^n_+$.\\
{\it Case} (ii). Let $(u_1,u_2)$ a pair of nonnegative solutions of \eqref{p1.9}. By the assumptions on $\tilde{f}_1,\tilde{f}_2$, we have
\begin{equation}
\label{l4.1.e6}
\begin{aligned}
&[\tilde{f}_j(u_1,u_2)-\tilde{f}_j(\overline{u}_1,\overline{u}_2)](u_j-\overline{u}_j)\\
&\leq C[(\|u_1\|_{L^{\infty}(M)}+\|u_2\|_{L^{\infty}(M)})^{\mu}+(\|u_1\|_{L^{\infty}(M)}+\|u_2\|_{L^{\infty}(M)})^{p-1}]\\
& \text{  }~ \times [(u_1-\overline{u}_1)^2+(u_2-\overline{u}_2)^2]
\end{aligned}
\end{equation}
for $j=1,2$.  Following as in the proof of (i), we have
\begin{align*}
& \int_{M}(u_1-\overline{u}_1)L^g_s(u_1-\overline{u}_1)~d\upsilon_g + \lambda_1 \int_{M}(u_1-\overline{u}_1)^2d\upsilon_g \\
&\text{ }~+\lambda_2 \int_{M}(u_1-\overline{u}_1)(u_2-\overline{u}_2)d\upsilon_g = \int_{M} [\tilde{f}_1(u_1,u_2)-\tilde{f}_1(\overline{u}_1,\overline{u}_2)](u_1-\overline{u}_1)~d\upsilon_g
\end{align*}
and
\begin{align*}
& \int_{M}(u_2-\overline{u}_2)L^g_s(u_2-\overline{u}_2)~d\upsilon_g+\lambda_3 \int_{M}(u_1-\overline{u}_1)(u_2-\overline{u}_2)d\upsilon_g \\
&\text{ }~+ \lambda_4 \int_{M}(u_2-\overline{u}_2)^2d\upsilon_g = \int_{M} [\tilde{f}_2(u_1,u_2)-\tilde{f}_2(\overline{u}_1,\overline{u}_2)](u_2-\overline{u}_2)~d\upsilon_g.
\end{align*}
It follows from the above equalities that
\begin{align*}
&\int_{M}(u_1-\overline{u}_1)L^g_s(u_1-\overline{u}_1)~d\upsilon_g + \int_{M}(u_2-\overline{u}_2)L^g_s(u_2-\overline{u}_2)~d\upsilon_g \\
&\text{  }~ + \int_{M}(u_1-\overline{u}_1,u_2-\overline{u}_2)A(u_1-\overline{u}_1,u_2-\overline{u}_2)~d\upsilon_g\\
&= \int_{M} \left \{[\tilde{f}_1(u_1,u_2)-\tilde{f}_1(\overline{u}_1,\overline{u}_2)](u_1-\overline{u}_1)+[\tilde{f}_2(u_1,u_2)-\tilde{f}_2(\overline{u}_1,\overline{u}_2)](u_2-\overline{u}_2) \right\} d\upsilon_g.
\end{align*}
The second of the theorem follows from the above equality, \eqref{l4.1.e6},  (\ref{g1.40}) if $M=\mathbb{S}^n$ or $\lambda_{1,g}$ if $M=\mathbb{S}^n_+$. 
\end{proof}

The following result, whose proof is easily demonstrated, will be used in the proof of Theorems \ref{th1.3}, \ref{th1.7} and \ref{th1.9} for the system case.
 \begin{lemma}
     \label{gl4.5}
     Assume that $f_1$, $f_2$ satisfy (F1)-(F4). Then $\psi_1,\psi_2:[0,+\infty)\times[0,+\infty)\rightarrow [0,+\infty)$ satisfy the following:
     \begin{enumerate}
     \item[(i)] $\psi_1(r,t)$, $\psi_2(r,t)$ are nondecreasing in $r \geq 0$ and $t\geq 0$;
     \item[(ii)] ${\psi_1(at,bt)}{t^{-\frac{n+2s}{n-2s}}}$, ${\psi_2(at,bt)}{t^{-\frac{n+2s}{n-2s}}}$ are nonincreasing in $t>0$ for any $a\geq 0$, $b\geq 0$;
     \item[(iii)] there exist $C>0$ such that
     \begin{equation*}
     \frac{|\psi_j(r_1,t_1)-\psi_j(r_2,t_2)|}{|(r_1,t_1)-(r_2,t_2)|}\leq C(r_1+t_1)^{p} \text{ for all } r_2>r_1\geq 0,t_1>t_1\geq 0, ~j=1,2;
     \end{equation*}
     \item[(iv)] $(\psi_1+\psi_2)(r,t)>0$ if $r+t>0$;
     \item[(v)] if $\lim_{k\rightarrow +\infty}a_k=a$ and $\lim_{k\rightarrow +\infty}b_k=b$,
     \begin{equation*}
     \lim_{k\rightarrow +\infty}\frac{\tilde{f}_j(a_kt_k,b_kt_k)}{t_k^p}=\psi_j(a,b) \text{ when } \lim_{k\rightarrow +\infty}t_k=+\infty, ~j=1,2.
     \end{equation*}
\end{enumerate}      
     \end{lemma}
     
     \begin{remark}
     \label{r4.0}
     If $(v_1,v_2)$ is a pair of nonnegative strong solutions of
     \begin{equation*}
     \left\{\begin{aligned}
     (-\Delta)^s v_1=\psi_1(v_1,v_2) \text{ in } \mathbb{R}^n,\\
     (-\Delta)^s v_2=\psi_2(v_1,v_2) \text{ in } \mathbb{R}^n,
     \end{aligned}\right.
\end{equation*}      
where $0<2s<n$ and $\psi_1,\psi_2$ are given by Lemma \ref{gl4.5}, then by  \cite{CDQ, GL1} and maximum principle we have that $v_j=0$ or $v_j>0$ in $\mathbb{R}^n$ for $j=1,2$. Moreover, if $v_1+v_2\neq 0$ then by Theorems \ref{3.th2} and \ref{3.th3}, either $v_j$ is positive solution of 
\begin{equation*}
v_j(x)=c_{n,s}\int_{\mathbb{R}^n}\frac{\psi_j(v_j(y),0)}{|x-y|^{n-2s}}dy,~x\in \mathbb{R}^n, \text{ for some } j=1,2,
\end{equation*}
or $(v_1,v_2)$ is a pair of positive solutions of
\begin{equation}
\left\{\begin{aligned}
v_1(x)=c_{n,s}\int_{\mathbb{R}^n}\frac{\psi_1(v_1(y),v_2(y))}{|x-y|^{n-2s}}dy,\\
v_2(x)=c_{n,s}\int_{\mathbb{R}^n}\frac{\psi_2(v_1(y),v_2(y))}{|x-y|^{n-2s}}dy.
\end{aligned}
\right.
\end{equation}
     \end{remark}

\subsection{Constants solutions for some elliptic problems on $\mathbb{S}^n_+$}\

In this subsection we will prove Theorem \ref{th1.3}.
The following result shows a type of a priori estimate for the solutions of (\ref{p1.8}). 
\begin{lemma}
\label{gl2.2.1}
Let $(M,g)=(\mathbb{S}^n_+,g)$ and $g\in[g_{\mathbb{S}^n_+}]$.  Suppose that $\tilde{f}$ satisfies (H1)-(H3) for $s=1$. Then there exists a positive constant $ C = C (n, g,\mathbb{S}^n_+,\tilde{f})$ such that for any $\lambda\geq 0$, any nonnegative continuous solution $u$ of \eqref{p1.8} satisfies
\begin{equation}
\label{l1.c1}
	\|u\|_{L^{\infty}(\mathbb{S}^n_+)}\leq C\lambda^{\frac{n-2}{4}}.
\end{equation}
\end{lemma}
\begin{proof}
Let us suppose that (\ref{l1.c1}) does not hold. Then there exist four
sequences $C_k$, $\lambda_k$, $\zeta_k$ and $u_k$ such that $\lambda_k\geq 0$ and $u_k\neq 0$ is a nonnegative solution of 
\begin{equation}
\label{l1.e1}
 \left\{ \begin{aligned}
    -L^1_g u_k & =  \tilde{f}(u_k)-\lambda_k u_k  && \text{ in } \mathbb{S}^n_+, \\
    \frac{\partial u_k}{\partial \nu}  &= 0 && \text{ on } \partial \mathbb{S}^n_+,
    \end{aligned}
    \right.
\end{equation}
 with the following properties:
\begin{align}
\label{l1.e2}
\lim_{k\rightarrow \infty}C_k & =+\infty, ~\{\lambda_k\} \text{  is uniformly bounded },\\
\label{l1.e3}
u_k(\zeta_k) & =\|u_k\|_{L^{\infty}(\mathbb{S}^n_+)}=C_k\lambda_k^{\frac{n-2}{4}},\\
\notag
\lim_{k\rightarrow \infty}\zeta_k & = \zeta_0.
\end{align}
 There are three possibilities:
 
\begin{case}
\label{c2.1.1} $\|u_k\|_{L^{\infty}(\mathbb{S}^n_+)}$ tends to zero when $k$ tends to infinity.\end{case} Then there exists some $k_0$ such that
\begin{equation*}
C( \|u\|^{\mu}_{L^{\infty}(\mathbb{S}^n_+)}+\|u\|^{p-1}_{L^{\infty}(\mathbb{S}^n_+)})<\lambda_{1,g}< \lambda_k+\lambda_{1,g}~ \text{ for each } k\geq k_0.
\end{equation*} 
 By Lemma \ref{gl2.1}, $u_k$ is
constant for $k$ large. From (\ref{l1.e1}) and (\ref{l1.e3}) it follows that
\begin{equation*}
1=C^{\frac{n-2}{4}}_k\tilde{f}(u_k)u_k^{-\frac{n+2}{n-2}},
\end{equation*} 
and, since $u_k<1$ for $k$ large and $\tilde{f}(t)t^{-\frac{n+2}{n-2}}$ is nonincreasing positive in $t>0$, we have that $1\geq C^{\frac{n-2}{4}}_k\tilde{f}(1)$, which contradicts (\ref{l1.e2}).

\begin{case}
\label{c2.1.2}  $\|u_k\|_{L^{\infty}(\mathbb{S}^n_+)}$ tends to some nonzero limit $c$ when $k$ tends to infinity. \end{case}
From (\ref{l1.e2}) - (\ref{l1.e3}), $\lambda_k$ tends to 0. Setting $\tilde{u}_k=u_k/\|u_k\|_{L^{\infty}(\mathbb{S}^n_+)}$  then
\begin{equation*}
\label{l1.e5}
\left\{ \begin{aligned}
-L^1_g \tilde{u}_k & = \|u_k\|^{-1}\tilde{f}(\|u_k\| \tilde{u}_k) - \lambda_k \tilde{u}_k & \mbox{ in } \mathbb{S}^n_+,\\
\frac{\partial \tilde{u}_k}{\partial \nu} & = 0 & \mbox{ on } \partial \mathbb{S}^n_+. 
\end{aligned}
    \right.
\end{equation*}
Setting the inverse of stereographic projection $\mathcal{F}$ for some $\zeta \in \partial\mathbb{S}^n_+$ being the north pole and denoting $H=\xi_1(\varphi \circ\mathcal{F})$, we can see that $v_k=H(\tilde{u}_k\circ \mathcal{F})$ resolves
\begin{equation*}
\label{l1.e5.1}
\left\{\begin{aligned}
-\Delta v_k & = H^{\frac{n+2}{n-2}} \left[ \|u_k\|^{-1}\tilde{f}(\|u_k\|H^{-1}v_k) + (R_g\circ \mathcal{F}- \lambda_k) H^{-1}v_k \right] & \mbox{ in } \mathbb{R}^n_+,\\
 \frac{\partial v_k}{\partial \nu} &=  (\varphi \circ \mathcal{F})^{-1}\frac{ \partial (\varphi \circ \mathcal{F})}{\partial \nu} v_k & \mbox{ on } \partial \mathbb{R}^n_+.
\end{aligned}\right.
\end{equation*}
 One can see that the right side of the above problem is in $L^1(\mathbb{R}^n_+)$ and is uniformly bounded in $k$.  Since $\tilde{f}\in C^{0,\gamma}_{loc}(\mathbb{R})$, it follows from \cite[Theorems 8.1, 8.2 and 8.3]{MS} and after passing to a subsequence that
\begin{equation*}
v_k\rightarrow v  \text{ in } C^{2, \beta}_{loc}(\overline{\mathbb{R}^n_+}) ~\mbox{ and }~ \tilde{u}_k \rightarrow \tilde{u}  \text{ in } C^2_{loc}(\overline{\mathbb{S}^n_+}\backslash \{ \zeta_0\}),
\end{equation*}
where $v$ and $\tilde{u}$ are non-negative smooth functions defined, respectively, in $\overline{\mathbb{R}^n_+}$ and $\overline{\mathbb{S}^n_+}$, with $0<\beta<1$. If we consider $-\zeta \in \partial \mathbb{S}^n_+$ being the north pole, and following the above arguments, we have after passing to a subsequence that $\tilde{u}$ solves
\begin{equation}
\left\{\begin{aligned}
\label{l1.6.1}
-L^1_g \tilde{u}  &=c^{-1}\tilde{f}(c\tilde{u}) &\text{in } \mathbb{S}^n_+,\\
 \frac{\partial \tilde{u}}{\partial \nu} & = 0 &\text{on } \partial \mathbb{S}^n_+,
\end{aligned}\right.
\end{equation}
and $\tilde{u}(\zeta_0)=1$, which is impossible because the integral on $\mathbb{S}^n_+$ on the right side of (\ref{l1.6.1}) is zero while the integral of the left side is not zero.
\begin{case}
\label{c2.1.3} $\|u_k\|_{L^{\infty}(\mathbb{S}^n_+)}$ tends to infinity when $k$ tends to infinity.
\end{case} There exists a sequence $\{\zeta_k'\}_{k\in \mathbb{N}}$ on $\overline{\mathbb{S}^n_+}$ such that
\begin{equation*}
\begin{aligned}
(\varphi u_k)(\zeta_k')=\|\varphi u_k\|_{L^{\infty}(\mathbb{S}^n_+)} \text{ and } \lim_{k \rightarrow +\infty}\zeta_k'=\zeta_0' \text{ for some } \zeta_0'\in \overline{\mathbb{S}^n_+}.
\end{aligned}
\end{equation*}
 By (\ref{l1.e1}), $v_k=H(u_k\circ \mathcal{F})$ resolves
\begin{equation*}
\label{l1.e7}
\left\{\begin{aligned}
-\Delta v_k &=H^{\frac{n+2}{n-2}}\left[\tilde{f}(H^{-1}v_k) + (R_g\circ \mathcal{F}-\lambda_k)H^{-1}v_k \right] & \text{ in } \mathbb{R}^n_+,\\
\frac{\partial v_k}{\partial \nu} & = (\varphi\circ\mathcal{F})^{-1}\frac{\partial \varphi}{\partial \nu} v_k &\mbox{ on } \partial \mathbb{R}^n_+,
\end{aligned}\right.
\end{equation*}
where $H=\xi_1 (\varphi \circ \mathcal{F})$ and $\mathcal{F}$ is the inverse of stereographic projection for some $\zeta \in \partial \mathbb{S}^n_+$, $\zeta\neq \zeta_0'$, being the north pole. Then there exist $R>0$ and $C>0$ such that $x_k:=\mathcal{F}^{-1}(\zeta_k')\in D_R:=B_R\cap \overline{\mathbb{R}^n_+}$ for all $k$,
\begin{equation}
\label{l4.2.e11}
\begin{aligned}
 C^{-1}\|\varphi u_k\|_{L^{\infty}(\mathbb{S}^n_+)} \leq  \xi_1(x_k)\varphi(\zeta_k')u(\zeta_k') &\leq  \|v_k\|_{L^{\infty}(D_R)} \\
 & \leq  \|v_k\|_{L^{\infty}(\mathbb{R}^n_+)}\leq C\|\varphi u_k\|_{L^{\infty}(\mathbb{S}^n_+)} .
 \end{aligned}
\end{equation}
 Let us define the following scaling
\begin{equation*}
\tilde{v}_k(x)={\alpha}_k^{\frac{2}{p-1}}v_k(\alpha_k x),
\end{equation*}
where $\alpha_k$ is defined by
\begin{equation*}
\alpha_k^{\frac{2}{p-1}}\|v_k\|_{L^{\infty}(\mathbb{R}^n_+)}=1.
\end{equation*}
By (\ref{l4.2.e11}), $\tilde{v}_k$ satisfies $\|\tilde{v}_k\|_{L^{\infty}(D_R)}>\tilde{C}>0$ and 
\begin{equation}
\label{l1.e8}
\left\{\begin{aligned}
-\Delta \tilde{v}_k(x) = & \alpha_k^{\frac{2p}{p-1}} H(\alpha_k x)^{\frac{n+2}{n-2}}\left[\tilde{f}(\alpha_k^{-\frac{2}{p-1}}H(\alpha_k x)^{-1} \tilde{v}_k) \right] + \tilde{\lambda}_k(x)\tilde{v}_k(x), & x\in \mathbb{R}^n_+,\\
\frac{\partial \tilde{v}_k}{\partial \nu}(x')  =& \alpha_k \varphi(\mathcal{F}(\alpha_k x'))^{-1}\frac{\partial \varphi}{\partial \nu}(\alpha_k x') \tilde{v}_k(\alpha_k x'), &x'\in \partial \mathbb{R}^n_+,
\end{aligned}\right.
\end{equation}
where 
\begin{equation*}
\tilde{\lambda}_k(x)= +(R_g\circ \mathcal{F}(\alpha_k x)-\lambda_k)H^{\frac{4}{n-2}}(\alpha_k x)\alpha_k^{2}.
\end{equation*}
By the conditions on $\tilde{f}$, the right side of (\ref{l1.e8}) is in $L^1(\mathbb{R}^n_+)$ and is uniformly bounded in $k$. Since $\tilde{f}\in C^{0,\gamma}_{loc}([0,+\infty))$, it follows from the regularity theory in \cite{MS} that
\begin{equation*}
\tilde{v}_k \rightarrow \tilde{v} \mbox{ in } C^{1, \beta}_{loc}(\overline{\mathbb{R}^n_+})
\end{equation*}
for some nonnegative function $\tilde{v}$ and $\beta \in (0,1)$. Moreover, since 
\begin{equation*}
\lim_{k\rightarrow +\infty}H(\alpha_k x)= H(0) \text{ and }
\lim_{k\rightarrow +\infty} (R_g\circ \mathcal{F})(\alpha_k x) = (R_g\circ \mathcal{F})(0),
\end{equation*}
 it gets that
\begin{equation*}
 \lim_{k\rightarrow +\infty}[\alpha_k^{\frac{2}{p-1}} H(\alpha_k x)]^{-1}\tilde{v}_k( x) = +\infty \mbox{ if } \tilde{v}(x) \neq 0.
\end{equation*}
Therefore, since $\lim_{t\rightarrow +\infty}\tilde{f}(t)t^{-p}=c_1$, it follows that
\begin{equation*}
\lim_{k\rightarrow +\infty} [\alpha_k^{\frac{2}{p-1}}H(\alpha_k x)]^{p} \tilde{f}((\alpha_k^{\frac{2}{p-1}} H(\alpha_k x))^{-1}\tilde{v}_k( x))= c_1\tilde{v}(x)^p \text{ for all } x\in \mathbb{R}^n_+, 
\end{equation*}
and  $\tilde{v}$ solves
\begin{equation*}
\left\{\begin{aligned}
-\Delta \tilde{v} &= \tilde{c}_1 \tilde{v}^{p} & \text{in } \mathbb{R}^n_+,\\
\frac{\partial \tilde{v}}{\partial \nu} & = 0 & \text{on } \partial\mathbb{R}^n_+,
\end{aligned}\right.
\end{equation*}
and $ \|\tilde{v}\|_{L^{\infty}(D_R)}>0$, which is impossible from \cite{XY1} since $p< (n + 2) / (n - 2)$.
\end{proof}
We also have a type of a priori estimate for the solutions of \eqref{p1.9}. 
\begin{lemma}
\label{gl2.3}
Let $(M,g)=(\mathbb{S}^n_+,g)$ and $g\in[g_{\mathbb{S}^n_+}]$.  Suppose that $\tilde{f}_j$ satisfies (F1)-(F4) for $s=1$ and $j=1,2$. Then there exists a positive constant $C=C(n,g, \mathbb{S}^n_+,\tilde{f}_1,\tilde{f}_2)$ such that for any $\|A\|$, and any nonnegative continuous solution pair $(u_1,u_2)$ of \eqref{p1.9} with $M=\mathbb{S}^n_+$ satisfies
\begin{equation}
\label{le2.c2}
\|u_1\|_{L^{\infty}(\mathbb{S}^n_+)} + \|u_2\|_{L^{\infty}(\mathbb{S}^n_+)}\leq C \|A\|^{\frac{4}{n-2}}.
\end{equation}
\end{lemma}
     \begin{proof}
      We will follow the arguments used in the proof of Lemma \ref{gl2.2.1}. Let us suppose that (\ref{le2.c2}) does not hold. Then there exist sequences $C_k$, $\zeta_k$, $\omega_k$, $u_{1,k}$, $u_{2,k}$ and 
     \begin{equation*}
     A_k= \left(\begin{matrix}
     \lambda_{1,k} & \lambda_{2,k}\\
     \lambda_{3,k} & \lambda_{4,k}
     \end{matrix}\right)
     \end{equation*}
such that $U_k:=(u_{1,k},u_{2,k})\neq (0,0)$ is a pair of nonnegative solutions of
     \begin{equation}
	\label{l2.e1}
	\left\{ \begin{aligned}
    -L^1_g { U_k} &=( \tilde{f}_1({U_k}), \tilde{f}_2({U_k})) -A_k{U_k} && \text{ in } \mathbb{S}^n_+, \\
   \frac{\partial {U_k}}{\partial \nu} &=  0  && \text{ on } \partial \mathbb{S}^n_+,
    \end{aligned}
    \right.
    \end{equation}
with the following properties:
	\begin{align}
	\label{l2.e3}
	\lim_{k\rightarrow \infty}C_k  =+\infty,~ & \{\|A_k\|\}_k \text{ is uniformly bounded}, \\
	\label{l2.e4}
	u_{1,k}(\zeta_k)+u_{2,k}(\omega_k)&=  \|u_{1,k}\|_{L^{\infty}(\mathbb{S}^n_+)}+\|u_{2,k}\|_{L^{\infty}(\mathbb{S}^n_+)}=C_k \|A_k\|^{\frac{n-2}{4}},\\
	\label{l2.e5}
	\lim_{k\rightarrow \infty}\zeta_k=\zeta_0,&  ~\lim_{k\rightarrow \infty}\omega_k=\omega_0,
	\end{align}
	for some $\zeta_0$, $\omega_0\in \overline{\mathbb{S}^n_+}$.
     There are three possibilities.
     \begin{case}
     \label{c4.5.1}
      $\|u_{1,k}\|_{L^{\infty}(\mathbb{S}^n_+)}+\|u_{2,k}\|_{L^{\infty}(\mathbb{S}^n_+)}$ tends to zero when $k$ tends to infinity.
      \end{case} From (\ref{l2.e3})-(\ref{l2.e5}), we have $\lim_{k\rightarrow +\infty}\|A_k\|=0$. Then there exists $k_0$ such that for $k\geq k_0$, it has that
     \begin{align*}
     4C\left(\sum_{l=1}^{2}\|u_{l,k}\|^{\mu_j}_{L^{\infty}(\mathbb{S}^n_+)}+\|u_{l,k}\|^{p_j-1}_{L^{\infty}(\mathbb{S}^n_+)}\right)&< \lambda_{1,g},~j=1,2,\\
         2 (\lambda_{1,g}-\|A_k\|)&>\lambda_{1,g}.
     \end{align*}
    By Lemma \ref{gl2.1}, $u_{1,k}$ and $u_{2,k}$ are constants. From \eqref{l2.e1},
    \begin{align*}
  \tilde{f}_1(u_{1,k},u_{2,k})+\tilde{f}_2(u_{1,k},u_{2,k}) \leq & |(\tilde{f}_1(U_k),\tilde{f}_2({U_k}))|\\
  \leq & \|A_k\||U_k|\leq 2\|A_k\|(u_{1,k}+u_{2,k}).
    \end{align*}
     Since $\tilde{f}_j(r,t)(r+t)^{-\frac{n+2}{n-2}}$ is nonincreasing in $(r,t)$ for $j=1,2$, and $u_{1,k}<1$, $u_{2,k}<1$ for k large, it follows from \eqref{l2.e4} that
     \begin{align*}
     2> &  \frac{2 C_k^{\frac{4}{n-2}}\|A_k\|}{(u_{1,k}+u_{2,k})^{\frac{4}{n-2}}}\\
     \geq &  \frac{\tilde{f}_1(u_{1,k},u_{2,k})+\tilde{f}_2(u_{1,k},u_{2,k})}{(u_{1,k}+u_{2,k})^{\frac{n+2}{n-2}}}\\
	\geq&  C_k^{\frac{4}{n-2}} (\tilde{f}_1(1,1)+\tilde{f}_2(1,1)),
     \end{align*}
     which is a contradiction since $\tilde{f}_1(1,1)+\tilde{f}_2(1,1)>0$.
     \begin{case}
     \label{c4.5.2}
     $\|u_{1,k}\|_{L^{\infty}(\mathbb{S}^n_+)}$ and $\|u_{2,k}\|_{L^{\infty}(\mathbb{S}^n_+)}$ tend to some limit $c_1\neq 0$ and $c_2$, respectively, when $k$ tends to infinity.
     \end{case}
     From (\ref{l2.e3})-(\ref{l2.e5}), $\|A_k\|$ tend to $0$. Following the proof of Lemma \ref{gl2.2.1} (Case \ref{c2.1.2}), we set $\tilde{u}_{1,k}=u_{1,k}/\|u_{1,k}\|_{L^{\infty}(\mathbb{S}^n_+)}$ and then $v_{1,k}=H (\tilde{u}_{1,k}\circ \mathcal{F})$, $v_{2,k}=H (u_{2,k}\circ \mathcal{F})$ resolve 
     \begin{equation*}
   \left\{  \begin{aligned}
-\Delta v_{1,k}  =& H^{\frac{n+2}{n-2}} \left[\|u_{1,k}\|_{L^{\infty}(\mathbb{S}^n_+)}^{-1}\tilde{f}_1(\|u_{1,k}\|_{L^{\infty}(\mathbb{S}^n_+)}H^{-1}v_{1,k},H^{-1}v_{2,k})\right]\\
 &+ \left. \tilde{a}_{k}(x) v_{1,k}+\tilde{b}(x)v_{2,k}\|u_{1,k}\|_{L^{\infty}(\mathbb{S}^n_+)}^{-1} \right.& \text{ in } \mathbb{R}^n_+,\\
-\Delta v_{2,k}  =& H^{\frac{n+2}{n-2}} \left[\tilde{f}_2(\|u_{1,k}\|_{L^{\infty}(\mathbb{S}^n_+)}H^{-1}v_{1,k},H^{-1}v_{2,k})\right]\\
 &+ \left. \tilde{c}_{k}(x) \|u_{1,k}\|_{L^{\infty}(\mathbb{S}^n_+)}v_{1,k}+\tilde{d}(x)v_{2,k} \right.& \text{ in } \mathbb{R}^n_+,\\
 \frac{\partial v_{1,k}}{\partial \nu}  =& (\varphi\circ\mathcal{F})^{-1}\frac{\partial \varphi}{\partial \nu} v_{1,k}, ~~\frac{\partial v_{2,k}}{\partial \nu}  = (\varphi\circ\mathcal{F})^{-1}\frac{\partial \varphi}{\partial \nu} v_{2,k}& \text{ on } \partial\mathbb{R}^n_+,
\end{aligned} \right.
\end{equation*}
where $\mathcal{F}$ is the inverse of stereographic projection for some $\zeta \in \partial\mathbb{S}^n_+$ being the north pole, $H=\xi_1(\varphi \circ\mathcal{F})$ in $\mathbb{R}^n$,
\begin{equation}
\label{l4.e31}
\begin{aligned}
\tilde{a}_k(x)=H^{\frac{4}{n-2}} (R_g(\mathcal{F}(x)) +\lambda_{1,k}),& ~\tilde{b}_k(x)=H^{\frac{4}{n-2}} \lambda_{2,k},\\
\tilde{d}_k(x)=H^{\frac{4}{n-2}} (R_g(\mathcal{F}(x)) +\lambda_{4,k}),& ~\tilde{c}_k(x)=H^{\frac{4}{n-2}} \lambda_{3,k}.
\end{aligned}
\end{equation}
One can see that the right hand side of the above system is in $L^1(\mathbb{R}^n)$ and is uniformly bounded in $k$. Then, since $\tilde{f}_1, \tilde{f}_2\in C^{0,\gamma}_{loc}([0,+\infty)\times [0,+\infty))$, it follows from \cite[Theorems 8.1, 8.2 and 8.3]{MS} and after passing to a subsequence that
\begin{equation*}
\begin{aligned}
v_{1,k}\rightarrow v_1  \text{ in } C^{2, \beta}_{loc}(\overline{\mathbb{R}^n_+}),&~ \tilde{u}_{1,k} \rightarrow \tilde{u}_1  \text{ in } C^2_{loc}(\overline{\mathbb{S}^n_+}\backslash \{ \zeta_0\}),\\
v_{2,k}\rightarrow v_2  \text{ in } C^{2, \beta}_{loc}(\overline{\mathbb{R}^n_+}),&~ \tilde{u}_{2,k} \rightarrow \tilde{u}_2  \text{ in } C^2_{loc}(\overline{\mathbb{S}^n_+}\backslash \{ \zeta_0\}),
\end{aligned}
\end{equation*}
where $v_1$, $\tilde{u}_1$, $v_2$ and $\tilde{u}_2$ are non-negative functions and $0<\beta<1$. If we consider $-\zeta \in \partial \mathbb{S}^n_+$ being the north pole and following the same previous arguments, we have that $(\tilde{u}_1,u_2)$ solves
\begin{equation*}
\left\{\begin{aligned}
\label{l4.6.1}
-L^1_g \tilde{u}_1  &=c_1^{-1}\tilde{f}_1(c_1\tilde{u}_1,u_2) &\text{in } \mathbb{S}^n_+,\\
-c_1^{-1}L^1_g {u}_2  &=c_1^{-1} \tilde{f}_2(c_1\tilde{u}_1,u_2) &\text{in } \mathbb{S}^n_+,\\
 \frac{\partial \tilde{u}_1}{\partial \nu} & =   \frac{\partial {u}_2}{\partial \nu}= 0 &\text{on } \partial \mathbb{S}^n_+,
\end{aligned}\right.
\end{equation*}
and $\tilde{u}_1(\zeta_0)=1$, which is impossible because either $\tilde{u}_1\equiv 0$ or the integral of $L^1_g \tilde{u}_1+c_1^{-1}L^1_g {u}_2$ on $\mathbb{S}^n_+$ is zero while the integral of $\tilde{f}_1(c_1\tilde{u}_1,u_2)+\tilde{f}_2(c_1\tilde{u}_1,u_2)$ is not zero.

\begin{case}
\label{c4.5.3}
 $\|u_{1,k}\|_{L^{\infty}(\mathbb{S}^n)}$ or $ \|u_{2,k}\|_{L^{\infty}(\mathbb{S}^n)}$ tend to infinity when $k$ tends to infinity.
\end{case}
Assume  after passing to a subsequence that
 \begin{equation}
\label{l2.e9}
\|v_{1,k}\|_{L^{\infty}(\mathbb{R}^n_+)}\geq \|v_{2,k}\|_{L^{\infty}(\mathbb{R}^n_+)} \text{ for all } k\in \mathbb{N}.
\end{equation}
There exist sequences $\{\zeta_k'\}_{k\in \mathbb{N}}, \{\omega_k'\}_{k\in \mathbb{N}}$ on $\overline{\mathbb{S}^n_+}$ such that
\begin{equation*}
\begin{aligned}
(\varphi u_{1,k})(\zeta_k')=\|\varphi u_{1,k}\|_{L^{\infty}(\mathbb{S}^n_+)},~ (\varphi u_{2,k})(\omega_k')=\|\varphi u_{2,k}\|_{L^{\infty}(\mathbb{S}^n_+)} \\
 \text{ and }\lim_{k \rightarrow +\infty}\omega_k'=\omega_0',~\lim_{k \rightarrow +\infty}\zeta_k'=\zeta_0' \text{ for some } \zeta_0',\omega_0'\in \overline{\mathbb{S}^n_+}.
\end{aligned}
\end{equation*}
By \eqref{l2.e1}, $v_{1,k}=H(u_{1,k}\circ \mathcal{F}]$ and $v_{2,k}=H(u_{2,k}\circ \mathcal{F})$ resolve
 \begin{equation*}
    \left\{ \begin{aligned}
-\Delta v_{1,k}  =& H^{\frac{n+2}{n-2}} \left[\tilde{f}_1(H^{-1}v_{1,k},H^{-1}v_{2,k})\right]+  \tilde{a}_{k}(x) v_{1,k}+\tilde{b}(x)v_{2,k} & \text{ in } \mathbb{R}^n_+,\\
-\Delta v_{2,k}  =& H^{\frac{n+2}{n-2}} \left[\tilde{f}_2(H^{-1}v_{1,k},H^{-1}v_{2,k})\right]+ \tilde{c}_{k}(x) v_{1,k}+\tilde{d}(x)v_{2,k} & \text{ in } \mathbb{R}^n_+,\\
 \frac{\partial v_{1,k}}{\partial \nu}  =& (\varphi\circ\mathcal{F})^{-1}\frac{\partial \varphi}{\partial \nu} v_{1,k}, ~~\frac{\partial v_{2,k}}{\partial \nu}  = (\varphi\circ\mathcal{F})^{-1}\frac{\partial \varphi}{\partial \nu} v_{2,k}& \text{ on } \partial\mathbb{R}^n_+,
\end{aligned} \right.
\end{equation*}
where $\mathcal{F}^{-1}$ is the stereographic projection for some $\zeta\in \partial \mathbb{S}^n_+$ being the north pole such that $\zeta\neq\zeta_0'$, $\zeta\neq \omega_0'$, and $\tilde{a}_k$, $\tilde{b}_k$, $\tilde{c}_k$, $\tilde{d}_k$ are given by (\ref{l4.e31}).
Following the proof of Lemma \ref{gl2.2.1}, there exist $R>0$ and $C>0$ such that $x_{k}:=\mathcal{F}^{-1}(\zeta_k'),~y_{k}:=\mathcal{F}^{-1}(\omega_k') \in D_R:=B_R\cap \overline{\mathbb{R}^n_+}$ for all $k$,
\begin{equation}
\label{gl2.2.e11}
\begin{aligned}
 C^{-1}\|\varphi u_{1,k}\|_{L^{\infty}(\mathbb{S}^n_+)} \leq  \xi_1(x_k)\varphi(\zeta_k')u_{1,k}(\zeta_k') &\leq  \|v_{1,k}\|_{L^{\infty}(D_R)} \\
 & \leq  \|v_{1,k}\|_{L^{\infty}(\mathbb{R}^n_+)}\leq C\|\varphi u_{1,k}\|_{L^{\infty}(\mathbb{S}^n_+)} .
 \end{aligned}
\end{equation}
 Let us define the following scaling
\begin{equation*}
\tilde{v}_{1,k}(x)={\alpha}_k^{\frac{2}{p-1}}v_{1,k}(\alpha_k x),~\tilde{v}_{2,k}(x)=\alpha_k^{\frac{2}{p-1}}v_{2,k}(\alpha_k x),~x\in \mathbb{R}^n,
\end{equation*}
where $\alpha_k$ is defined by $\alpha_k^{\frac{2}{p_1-1}}\|v_{1,k}\|_{L^{\infty}(\mathbb{R}^n)}=1$.
 By (\ref{l2.e9}), 
\begin{equation}
\label{l2.e10}
\alpha_k^{\frac{n-2}{2}}\|v_{2,k}\|_{L^{\infty}(\mathbb{R}^n)}\leq 1 \text{ and } \lim_{k\rightarrow \infty}\alpha_k=0.
\end{equation}
 From (\ref{gl2.2.e11}), $(\tilde{v}_{1,k},\tilde{v}_{2,k})$ satisfies $\|\tilde{v}_{1,k}\|_{L^{\infty}(D_R)}\geq \tilde{C}>0$ for all $k$ and 
\begin{align}
\label{gl2.2.e38}
-\Delta \tilde{v}_{1,k} = &\alpha_k^{\frac{2p}{p-1}} H(\alpha_kx)^{\frac{n+2}{n-2}} \left[\tilde{f}_1(\alpha_k^{-\frac{2}{p-1}} H(\alpha_kx)^{-1}\tilde{v}_{1,k},\alpha_k^{-\frac{2}{p-1}} H(\alpha_kx)^{-1}\tilde{v}_{2,k})\right]\\
\notag
& +  \alpha_k^2 \tilde{a}_{k}(\alpha_kx)\tilde{v}_{1,k}+ \alpha_k^2\tilde{b}(\alpha_kx) \tilde{v}_{2,k},\\
\label{gl2.2.e39}
-\Delta \tilde{v}_{2,k} = &\alpha_k^{\frac{2p}{p-1}} H(\alpha_kx)^{\frac{n+2}{n-2}} \left[\tilde{f}_2(\alpha_k^{-\frac{2}{p-1}} H(\alpha_kx)^{-1}\tilde{v}_{1,k},\alpha_k^{-\frac{2}{p-1}} H(\alpha_kx)^{-1}\tilde{v}_{2,k})\right]\\
\notag
& +  \alpha_k^2 \tilde{c}_{k}(\alpha_kx)\tilde{v}_{1,k}+ \alpha_k^2\tilde{d}(\alpha_kx) \tilde{v}_{2,k},\\
\notag
 \frac{\partial \tilde{v}_{1,k}}{\partial \nu}(x')  =& \alpha_k(\varphi\circ\mathcal{F})^{-1}(\alpha_k x')\frac{\partial \varphi}{\partial \nu}(\alpha_k x') \tilde{v}_{1,k}(x'),\\
 \notag
\frac{\partial \tilde{v}_{2,k}}{\partial \nu}(x')  =&  \alpha_k(\varphi\circ\mathcal{F})^{-1}(\alpha_k x')\frac{\partial \varphi}{\partial \nu}(\alpha_k x') \tilde{v}_{1,k}(x').
\end{align}
By the conditions on $\tilde{f}_1$ and $\tilde{f_2}$, the right side of (\ref{gl2.2.e38}) and (\ref{gl2.2.e39}) are in $L^1(\mathbb{R}^n_+)$ and are uniformly bounded in $k$. Since $\tilde{f}_1, \tilde{f}_2\in C^{0,\gamma}_{loc}([0,+\infty)\times [0,\infty))$, it follows from the regularity theory in \cite{MS} that
 \begin{equation*}
\tilde{v}_{1,k} \rightarrow \tilde{v}_1 \text{ and } \tilde{v}_{2,k} \rightarrow \tilde{v}_2 \mbox{ in } C^{1, \beta}_{loc}(\overline{\mathbb{R}^n_+})
\end{equation*}
for some nonnegative functions $\tilde{v}_1$, $\tilde{v}_2$ and $\beta \in (0,1)$. Moreover, 
\begin{equation*}
 \lim_{k\rightarrow +\infty}[\alpha_k^{\frac{2}{p-1}} H(\alpha_k x)]^{-1}\tilde{v}_{1,k}( x) = \lim_{k\rightarrow +\infty}[\alpha_k^{\frac{2}{p-1}} H(\alpha_k x)]^{-1}\tilde{v}_{2,k}( x)= +\infty 
\end{equation*}
for $\tilde{v}_1(x)\tilde{v}_{2,k}(x) \neq 0$. By Lemma \ref{gl4.5}, it follows that
\begin{equation*}
\lim_{k\rightarrow +\infty} [\alpha_k^{\frac{2}{p-1}}H(\alpha_k x)]^{p} \tilde{f}_j(\alpha_k^{-\frac{2}{p-1}} H(\alpha_k x)^{-1}\tilde{v}_{1,k},\alpha_k^{-\frac{2}{p-1}} H(\alpha_k x)^{-1}\tilde{v}_{2,k})= \psi_j(\tilde{v}_1,\tilde{v}_2)
\end{equation*}
for $j=1,2$, and  $(\tilde{v}_1,\tilde{v}_2)$ solves
\begin{equation*}
\left\{\begin{aligned}
-\Delta \tilde{v}_1 &= \tilde{c}_{1}\psi_1( \tilde{v}_1,\tilde{v}_2) & \text{in } \mathbb{R}^n_+,\\
-\Delta \tilde{v}_2 &= \tilde{c}_{2}\psi_2( \tilde{v}_1,\tilde{v}_2) & \text{in } \mathbb{R}^n_+,\\
\frac{\partial \tilde{v}_1}{\partial \nu} & = \frac{\partial \tilde{v}_2}{\partial \nu} = 0 & \text{on } \partial\mathbb{R}^n_+,
\end{aligned}\right.
\end{equation*}
for some $\tilde{c}_{1},\tilde{c}_{2}$ positive constants and $ \|\tilde{v}_1\|_{L^{\infty}(D_R)}>0$, which is impossible from maximum principle and \cite{XY1} since $\psi_j$ satisfies Lemma \ref{gl4.5} (iii) with $p<\frac{n+2}{n-2}$.
     \end{proof}

{\it Proof of Theorem \ref{th1.7}.} It follows from Lemma \ref{gl2.2.1}, Lemma \ref{gl2.3} and Lemma \ref{gl2.1} for $\lambda$ and $\|A\|$ small values.     
\hfill $\square$

\

\subsection{Constant solutions for problems involving the operator $P^g_s$ on $\mathbb{S}^n$}\

In this subsection we will prove Theorems \ref{th1.7} and \ref{th1.9}. We begin by showing the following lemma that is used to determine Schauder type estimates of the solutions of \eqref{p1.8}, \eqref{p1.9}, \eqref{p1.10} and \eqref{p1.11}. 
\begin{lemma}
\label{l4.3.0}
 Let $H:\mathbb{R}^n\rightarrow (0,+\infty)$ be a $ C^{\infty}(\mathbb{R}^n)$ function such that  $\|H\|_{C^{0,\gamma}(\mathbb{R}^n)}<C_1$, $|\nabla H(x)|\leq C_1 H(x)$ for all $x\in \mathbb{R}^n$ and
\begin{equation*}
\left| \frac{H(x)}{H(y)} \right|\leq C_1 \text{ for all } x,y \in\mathbb{R}^n \text{ with } |x-y|<1,
\end{equation*}
 where $C_1$ depends on $n$ and $s$. Assume that $\tilde{f}$ and $\tilde{f}_1,\tilde{f}_2$ satisfy (H1)-(H3) and (F1)-(F4), respectively. Then:
\begin{enumerate}
\item[(i)] for any $ v\in C^{0,\gamma}(\mathbb{R}^n)$ nonnegative function, 
\begin{equation*}
\| H^{\frac{n+2s}{n-2s}} \tilde{f}(H^{-1}v)\|_{C^{0,\gamma}(\mathbb{R}^n)}\leq C,
\end{equation*}
where $C$ depends on $C_1$, $\tilde{f}$ and $\|v\|_{C^{0,\gamma}(\mathbb{R}^n)}$;
\item[(ii)]for any $v_1,v_2\in C^{0,\gamma}(\mathbb{R}^n)$ nonnegative functions, 
\begin{equation*}
\| H^{\frac{n+2s}{n-2s}} \tilde{f}_j(H^{-1}v_1,H^{-1}v_2)\|_{C^{0,\gamma}(\mathbb{R}^n)}\leq C,~j=1,2,
\end{equation*}
where $C$ depends on $C_1$, $\tilde{f}_1$, $\tilde{f}_2$, $\|v_1\|_{C^{0,\gamma}(\mathbb{R}^n)}$ and $\|v_2\|_{C^{0,\gamma}(\mathbb{R}^n)}$.
\end{enumerate} 
  
\end{lemma}
\begin{proof}
We will prove only (i), since the proof of (ii) is similar to (i). By the assumptions on $\tilde{f}$, we have
\begin{equation*}
\tilde{f}(H^{-1}v)\leq C\tilde{f}(0)+ C(H^{-(\mu+1)}v^{\mu+1} +H^{-p}v^{p} ) \text{ in } \mathbb{R}^n.
\end{equation*}
Let $x,y\in \mathbb{R}^n$. Assume first that $|x-y|\geq 1$. Then
\begin{equation}
\label{l4.3.e1}
\begin{aligned}
\frac{1}{|x-y|^{\gamma}} | H(x)^{\frac{n+2s}{n-2s}} \tilde{f}(H(x)^{-1}v(x)) -H(y)^{\frac{n+2s}{n-2s}} \tilde{f}(H(y)^{-1}v(y))|\\
 \leq C(\tilde{f}(0)+\|v\|^{\mu+1}_{C^{0,\gamma}(\mathbb{R}^n)}+\|v\|^{p}_{C^{0,\gamma}(\mathbb{R}^n)}).
\end{aligned}
\end{equation}
Now assume that $|x-y|<1$. Then
\begin{align}
\notag
\frac{1}{|x-y|^{\gamma}} &| H(x)^{\frac{n+2s}{n-2s}} \tilde{f}(H(x)^{-1}v(x)) -H(y)^{\frac{n+2s}{n-2s}} \tilde{f}(H(y)^{-1}v(y))|\\
\notag
 \leq& \frac{|\tilde{f}(H(x)^{-1}v(x))-\tilde{f}(H(y)^{-1}v(y))|}{|x-y|^{\gamma}}H(y)^{\frac{n+2s}{n-2s}}\\
 \notag
&+ \frac{|H(x)^{\frac{n+2s}{n-2s}}-H(y)^{\frac{n+2s}{n-2s}}|}{|x-y|^{\gamma}} \tilde{f}(H(y)^{-1}v(y)) \\
\notag
\leq& \frac{|\tilde{f}(H(x)^{-1}v(x))-\tilde{f}(H(y)^{-1}v(y))|}{|H(x)^{-1}v(x)-H(y)^{-1}v(y)|}\frac{|H(x)^{-1}v(x)-H(y)^{-1}v(y)|}{|x-y|^{\gamma}} H(y)^{\frac{n+2s}{n-2s}}\\
\notag
& + C\max\{H(x),H(y)\}^{\frac{4s}{n-2s}}|\nabla H(\theta x + (1-\theta)y)|\tilde{f}(H(y)^{-1}v(y))\\
\notag
\leq& C \left(\max\{H(x)^{-1}v(x),H(y)^{-1}v(y)\}^{\mu}+\max\{H(x)^{-1}v(x),H(y)^{-1}v(y)\}^{p-1} \right) \\
\notag
& \times \left( \frac{|\nabla H(\theta x + (1-\theta)y)|}{H(x)H(y)} v(y) + \frac{\|v\|_{C^{0,\gamma}(\mathbb{R}^n)}}{H(y)}\right)H(y)^{\frac{n+2s}{n-2s}}\\
\notag
& +  C\max\{H(x),H(y)\}^{\frac{4s}{n-2s}}H(\theta x + (1-\theta)y)\\
\notag
& \times \left( \tilde{f}(0)+H(y)^{-(\mu+1)}v(y)^{\mu+1}+H(y)^{-p}v(y)^{p} \right)\\
\label{l4.3.e2}
\leq& C \left( \tilde{f}(0)+\|v\|^{\mu+1}_{C^{0,\gamma}(\mathbb{R}^n)} +\|v\|^{p}_{C^{0,\gamma}(\mathbb{R}^n)} + \right),
\end{align}
where $0\leq\theta \leq 1$. Therefore, the lemma follows from \eqref{l4.3.e1} and \eqref{l4.3.e2}.
\end{proof}

The following lemma shows a type of a priori estimation of the solutions of \eqref{p1.8}. Recall that the condition $R^g_s$ is used to transform the problems \eqref{p1.8}, \eqref{p1.9} into integral problems of the type \eqref{p1.10}, \eqref{p1.11}, and thus, to be able to determine some Schauder type estimates for the solutions of the problems addressed.
\begin{lemma}
\label{gl2.2.2}
Let $(M,g)=(\mathbb{S}^n,g)$ and $g\in[g_{\mathbb{S}^n}]$.  Suppose that $\tilde{f}$ satisfies (H1)-(H3) for $0<2s<n$ and $R^g_s$ is positive for $s>1$. Then there exists a positive constant $ C = C (n,s, g,\mathbb{S}^n,\tilde{f})$ such that for any $\lambda\in [ 0,+\infty)$ if $0<s\leq 1$ or $ \lambda \in [0, \min R^g_s]$ if $s>1$, any nonnegative smooth solution $u$ of \eqref{p1.8} satisfies
\begin{equation}
\label{l1.c1.1}
	\|u\|_{L^{\infty}(\mathbb{S}^n)}\leq C\lambda^{\frac{n-2s}{4s}}.
\end{equation}
\end{lemma}     
     \begin{proof}
     Let us suppose that (\ref{l1.c1.1}) does not hold. Then there exist four
sequences $C_k$, $\lambda_k$, $\zeta_k$ and $u_k$ such that $\lambda_k\geq 0$ and $u_k$ is a nonnegative solution of 
\begin{equation}
\label{l1.e1.1}
    -L^s_g u_k  =  \tilde{f}(u_k)-\lambda_k u_k   \text{ in } \mathbb{S}^n, 
\end{equation}
 with the following properties:
\begin{align}
\label{l1.e2.1}
\lim_{k\rightarrow \infty}C_k & =+\infty,\\
\label{l1.e2.1.1}
\max_{k\in \mathbb{N}}\{\lambda_k\} & \leq \begin{cases}
			M & \text{ if } 0<s\leq 1,\\
			\min R^g_s & \text{ if } s>1,
\end{cases}\\
\label{l1.e3.1}
u_k(\zeta_k) & =\|u_k\|_{L^{\infty}(M)}=C_k\lambda_k^{\frac{n-2s}{4s}},\\
\notag
\lim_{k\rightarrow \infty}\zeta_k & = \zeta_0,
\end{align}
for some $M\in \mathbb{R}^n$ and $\zeta_0\in \mathbb{S}^n$.
 Following the proof of Lemma \ref{gl2.2.1}, there are three possibilities:
\begin{case}
\label{c2.2.1} $\|u_k\|_{L^{\infty}(M)}$ tends to zero when $k$ tends to infinity.
\end{case} This case is similar to Case \ref{c2.1.1}.
\begin{case} 
\label{c2.2.2} $\|u_k\|_{L^{\infty}(M)}$ tends to some nonzero limit $c$ when $k$ tends to infinity.
\end{case}
From (\ref{l1.e2.1}) - (\ref{l1.e3.1}), $\lambda_k$ tends to 0. Setting $\tilde{u}_k=u_k/\|u_k\|_{L^{\infty}(\mathbb{S}^n)}$ then
\begin{equation*}
-L^s_g \tilde{u}_k  = \|u_k\|^{-1}\tilde{f}(\|u_k\| \tilde{u}_k) - \lambda_k \tilde{u}_k  \mbox{ in } \mathbb{S}^n.
\end{equation*}
Setting the inverse of stereographic projection $\mathcal{F}$ for some $\zeta \in \mathbb{S}^n$ being the north pole and denoting $H_s=\xi_s(\varphi \circ\mathcal{F})$, it can be deduced from (\ref{r1.1}) and (\ref{l1.e1.1}) that $v_k=H_s(\tilde{u}_k\circ \mathcal{F})$ resolves
\begin{equation}
\label{l1.e6.2}
(-\Delta)^s v_k  = H_s^{\frac{n+2s}{n-2s}} \left[ \|u_k\|^{-1}\tilde{f}(\|u_k\|H_s^{-1}v_k) + (R^g_s\circ \mathcal{F}- \lambda_k) H_s^{-1}v_k \right] \mbox{ in } \mathbb{R}^n.
\end{equation}
 One can see that the right side of the above problem is in $L^1(\mathbb{R}^n)$ and is uniformly bounded in $k$.  By (\ref{l1.e6.2}), standard regularity theory if $s=1$, \cite[Theorems 2.8 and 2.9]{Silv} if $0<s<1$, Theorem \ref{3.th1} and Theorem \ref{ath2.2} if $s>1$, we have that $\{\|v_k\|_{C^{1,\gamma_0}(\mathbb{R}^n)}\}_{k\in \mathbb{N}}$ is uniformly bounded for some $\gamma_0\in (0,1)$. It follows from (\ref{l1.e6.2}), Lemma \ref{l4.3.0}, regularity theory if $0<s\leq 1$, Theorems \ref{3.th1} and  \ref{ath2.1} if $s>1$ and after passing to a subsequence that
\begin{equation*}
\begin{aligned}
& v_k\rightarrow v  \text{ in } C^{2\lfloor s\rfloor, \beta}_{loc}(\mathbb{R}^n) \text{ if } s \in \mathbb{N},\\
& v_k\rightarrow v  \text{ in } C^{2\lfloor s\rfloor + 1, \beta}_{loc}(\mathbb{R}^n) \text{ if } s\notin \mathbb{N},\\
& \tilde{u}_k \rightarrow \tilde{u}  \text{ in } C^{2\lfloor s\rfloor}_{loc}(\mathbb{S}^n\backslash \{\zeta\}),\\
&\|(-\Delta)^{\lfloor s\rfloor} v_k\|_{ C^{1,\beta}(\mathbb{R}^n)} \leq C \text{ if } s \notin \mathbb{N},\\
& (-\Delta)^{\lfloor s\rfloor} v_k \rightarrow (-\Delta)^{\lfloor s\rfloor} v  \text{ in } C^{1, \beta}_{loc}(\mathbb{R}^n) \text{ if } s \notin {\mathbb{N}},
\end{aligned}
\end{equation*}
where $v$ and $\tilde{u}$ are non-negative functions defined, respectively, in $\mathbb{R}^n$ and $\mathbb{S}\backslash \{-\zeta_0\}$, $0<\beta<1$ with $2(s-\lfloor s\rfloor) < 1+ \beta$, and $C$ does not depend on $k$. Using for $s\notin \mathbb{N}$ the following estimate
\begin{equation*}
\begin{aligned}
& (-\Delta)^{\lfloor s\rfloor} v_k(\cdot+y)+(-\Delta)^{\lfloor s\rfloor} v_k(\cdot-y)-2(-\Delta)^{\lfloor s\rfloor} v_k(\cdot) \\
& \leq C \|(-\Delta)^{\lfloor s\rfloor} v_k\|_{ C^{1,\beta}(\mathbb{R}^n)}\left[ |y|^{1+\beta}\chi_{\{|y|<1\}}+ \chi_{\{|y|>1\}}\right],
\end{aligned}
\end{equation*} 
where $\chi$ is the indicator function on the variable $y$, and by the dominated convergence theorem, it follows (up to constants) that
\begin{equation*}
\begin{aligned}
(-\Delta)^s v_k(x)& =(-\Delta)^{s-\lfloor s\rfloor} \circ(-\Delta)^{\lfloor s\rfloor}  v_k (x)\\
 & =\frac{1}{2}\int_{\mathbb{R}^n}\frac{2(-\Delta)^{\lfloor s\rfloor} v_k(x) - (-\Delta)^{\lfloor s\rfloor} v_k(x+y)-(-\Delta)^{\lfloor s\rfloor} v_k(x-y)}{|y|^{n+2(s-\lfloor s \rfloor)}}dy\\
& \rightarrow (-\Delta)^s  v(x) \text{ as } k\rightarrow \infty.
\end{aligned}
\end{equation*}
 Tending $k$ to infinity, we have that $v=H_s (\tilde{u}\circ \mathcal{F})$ solves
 \begin{equation*}
(-\Delta)^s v(x) =H^{\frac{n+2s}{n-2s}} ( c^{-1}\tilde{f}(c\tilde{u})+ R^g_s \tilde{u})(\mathcal{F}(x)), ~x\in \mathbb{R}^n.
 \end{equation*}
  If we consider $-\zeta \in \partial \mathbb{S}^n_+$ being the north pole, and following the above argument, we have, after passing to a subsequence, that $\tilde{u}$ solves
\begin{equation}
\label{l1.e6}
-L^g_s \tilde{u} = c^{-1}\tilde{f}(c\tilde{u}) ~ \text{ in } \mathbb{S}^n,
\end{equation}
and $\tilde{u}(\zeta_0)=1$. Since $P^g_s$ is self adjoint in $L^2(\mathbb{S}^n,g)$, one can see that the integral on $\mathbb{S}^n$ of the left side of (\ref{l1.e6}) is zero while the integral of the left side is not zero, which is impossible.

\begin{case} 
\label{c2.2.3} $\|u_k\|_{L^{\infty}(\mathbb{S}^n)}$ tends to infinity when $k$ tends to infinity.
\end{case}
There exists a sequence $\{\zeta_k'\}_{k\in \mathbb{N}}$ in $\mathbb{S}^n$ such that
\begin{equation*}
\begin{aligned}
(\varphi u_k)(\zeta_k)'=\|\varphi u_k\|_{L^{\infty}(N)} \text{ and } \lim_{k \rightarrow +\infty}\zeta_k'=\zeta_0'
\end{aligned}
\end{equation*}
for some $\zeta_0'\in \mathbb{S}^n$. By (\ref{l1.e1}) and (\ref{e4.1}), $v_k=H_s(u\circ \mathcal{F})$ resolves
\begin{equation*}
(-\Delta)^s v_k =H_s^{\frac{n+2s}{n-2s}}\left[\tilde{f}(H_s^{-1}v_k) + (R^s_g\circ \mathcal{F}-\lambda_k)H_s^{-1}v_k \right]  \text{ in } \mathbb{S}^n,
\end{equation*}
where $H_s=\xi_s (\varphi \circ \mathcal{F})$, $\mathcal{F}$ is the inverse of stereographic projection with $\zeta_0'$ being the south pole.   Then there exist $R>0$ such that for all $k$,
\begin{equation}
\begin{aligned}
 C^{-1}\|\varphi u_k\|_{L^{\infty}(\mathbb{S}^n)}  \leq  \|v_k\|_{L^{\infty}(B_R)} \leq  \|v_k\|_{L^{\infty}(\mathbb{R}^n)}\leq C\|\varphi u_k\|_{L^{\infty}(\mathbb{S}^n)} .
 \end{aligned}
\end{equation}
 Let us define the following scaling
\begin{equation*}
\tilde{v}_k(x)={\alpha}_k^{\frac{2s}{p-1}}v_k(\alpha_k x),
\end{equation*}
where $\alpha_k$ is defined by
\begin{equation*}
\alpha_k^{\frac{2s}{p-1}}\|v_k\|_{L^{\infty}(\mathbb{R}^n)}=1.
\end{equation*}
For each $k$, $\tilde{v}_k$ satisfies $\|\tilde{v}_k\|_{L^{\infty}(B_R)}>\tilde{C}_1$ and 
\begin{equation}
\label{l1.e8.1}
\begin{aligned}
(-\Delta)^s \tilde{v}_k(x) =  \alpha_k^{\frac{2sp}{p-1}} H_s(\alpha_k x)^{\frac{n+2s}{n-2s}}\left[\tilde{f}(\alpha_k^{-\frac{2s}{p-1}}H_s(\alpha_k x)^{-1} \tilde{v}_k) \right] +\tilde{\lambda}_k(x) \tilde{v}_k , 
\end{aligned}
\end{equation}
where
\begin{equation*}
\tilde{\lambda}_k(x)= (R^s_g\circ \mathcal{F}(\alpha_k x)-\lambda_k)H_s^{\frac{4s}{n-2s}}(\alpha_k x)\alpha_k^{2s}.
\end{equation*}
One can see that the right side of (\ref{l1.e8.1}) is in $L^1(\mathbb{R}^n)$ and is uniformly bounded in $k$. By Lemma \ref{l4.3.0},standard  regularity theory if $0<s\leq 1$, Theorems \ref{ath2.1} and \ref{ath2.2} if $s>1$, and following the arguments of the above case, we have that
\begin{equation*}
\tilde{v}_k \rightarrow \tilde{v} \mbox{ in } C^{1, \alpha}_{loc}(\overline{\mathbb{R}^n_+}) \text{ and } (-\Delta)^s\tilde{v}_k(x) \rightarrow (-\Delta)^s\tilde{v}(x) \text{ for each } x\in \mathbb{R}^n,
\end{equation*}
Moreover,
\begin{equation*}
\begin{aligned}
& \lim_{k\rightarrow +\infty}H_s(\alpha_k x)= H_s(0) \text{ and }
\lim_{k\rightarrow +\infty} (R^s_g\circ \mathcal{F})(\alpha_k x) = (R^s_g\circ \mathcal{F})(0),\\
& \lim_{k\rightarrow +\infty}[\alpha_k^{\frac{2s}{p-1}} H_s(\alpha_k x)]^{-1}\tilde{v}_k(\alpha_k x) = +\infty \mbox{ if } \tilde{v}(x) \neq 0.
\end{aligned}
\end{equation*}
Therefore, since $\lim_{k\rightarrow +\infty}\alpha_k=0$ and $\lim_{t\rightarrow +\infty}\tilde{f}(t)t^{-p}=c_1$, it follows that
\begin{equation*}
 \lim_{k\rightarrow +\infty} [\alpha_k^{\frac{2s}{p-1}}H_s(\alpha_k x)]^{p} \tilde{f}((\alpha_k^{\frac{2s}{p-1}} H_s(\alpha_k x))^{-1}\tilde{v}_k(\alpha_k x))= c_1 \tilde{v}(x)^p, 
\end{equation*}
and  $\tilde{v}\in \mathcal{L}_s$ is a nonnegative strong solution of
\begin{equation*}
(-\Delta)^s \tilde{v} =\tilde{c}_1v^p \text{ in } \mathbb{R}^n \text{ and } \tilde{v}(0)=1,
\end{equation*}
which is impossible from Theorem \ref{3.th2} and  \cite[Corollary 1]{XY3} since $p<\frac{n+2s}{n-2s}$.
     \end{proof}

Following the arguments used in the proofs of Lemma \ref{gl2.2.2} and Lemma \ref{gl2.3} along with Lemma \ref{l4.3.0}, Lemma \ref{gl4.5} and Remark \ref{r4.0}, we can obtain the following result that shows a type of a priori estimate for the solutions of \eqref{p1.9} on $M=\mathbb{S}^n$.

\begin{lemma}
\label{l4.3.2}
Let $(M,g)=(\mathbb{S}^n,g)$ and $g\in[g_{\mathbb{S}^n}]$.  Suppose that $\tilde{f}_1, \tilde{f}_2$ satisfy (F1)-(F4) for $0<2s<n$ and $R^g_s$ is positive for $s>1$.  Then there exists a positive constant $ C = C (n,s, g,\mathbb{S}^n,\tilde{f}_1,\tilde{f}_2)$ such that
\begin{equation*}
\left\{ \begin{aligned}
&\text{ for any  } \|A\| \text{ if } 0<s<1 \text{ or }\\
&\text{ for any  } \|A\|\leq \min R^g_s, ~\lambda_2\geq 0, ~\lambda_3\geq 0 \text{ if } s>1,
\end{aligned} \right.
\end{equation*}
  any nonnegative smooth solution pair $(u_1,u_2)$ of \eqref{p1.9}  satisfies
\begin{equation*}
\|u_1\|_{L^{\infty}(\mathbb{S}^n)} + \|u_2\|_{L^{\infty}(\mathbb{S}^n)}\leq C \|A\|^{\frac{4}{n-2s}}.
\end{equation*}
\end{lemma}

{\it Proof of Theorem \ref{th1.7}.} It follows from Lemma \ref{gl2.2.2},  Lemma \ref{l4.3.2} and Lemma \ref{gl2.1} for $\lambda$ and $\|A\|$ smalls.     
\hfill $\square$     

\begin{remark}
\label{r4.1}
By the projection stereographic, Theorems \ref{ath2.1} and \ref{ath2.2},  Theorem \ref{th3.2.1} (i), Theorem \ref{3.th2} and the proofs of Lemma \ref{gl2.2.2} and Lemma \ref{l4.3.2}, we have that the conclusions of Lemmas \ref{gl2.2.2} and \ref{l4.3.2} hold for the solutions of \eqref{p1.10} and \eqref{p1.11} without the hypothesis on the sign of $R^g_s$, $\lambda_2$ or $\lambda_3$.
\end{remark}

{\it Proof of Theorem \ref{th1.9}.} It follows from Remark \ref{r4.1}, Theorem \ref{th3.2.1} and Lemma \ref{gl2.1} for $\lambda$ and $\|A\|$ smalls.     
\hfill $\square$

     \section{Appendix}
     
     \subsection{Conformally invariant operators on spheres}\
     
     In this subsection, we recall some results for $P_s$ with $s\in (0,1)$.  Pavlov and Samko \cite{PS} showed that
	\begin{equation}
	\label{g6.15}
	P_s(u)(\zeta)=c_{n,-s}P.V.\int_{{\mathbb{S}^n}}\frac{u(\zeta)-u(\omega)}{|\zeta-\omega|^{n+2s}}d\upsilon_{g_{\mathbb{S}^n}}^{(\omega)} + P_s(1)u(\zeta),~u\in C^{\infty}({\mathbb{S}^n}),~\zeta\in {\mathbb{S}^n},
	\end{equation}
	where $c_{n,-s}=\frac{2^{2s}s\Gamma(\frac{n+2s}{2})}{\pi^{\frac{n}{2}}\Gamma(1-s)}$, $|\cdot|$ is the Euclidean distance in $\mathbb{R}^{n+1}$ and $P.V.\int_{\mathbb{S}^n}$ is understood as $\lim_{\varepsilon \rightarrow 0}\int_{|\zeta-\omega|>\varepsilon}$ and $P_s(1)=\frac{\Gamma(\frac{n}{2}+s)}{\Gamma(\frac{n}{2}-s)}$. Then, using arguments based on \cite{DNPV}, we may write $P_s$ as a weighted second order differential quotient.
\begin{proposition}
\label{gap1.1}
If $s\in(0,1)$ , then 
\begin{equation*}
P_su(\zeta) - R_s u(\zeta)= \frac{c_{n,-s}}{2}\int_{\mathbb{S}^n}\frac{2u(\zeta) - u(\omega)-u(\theta(\zeta,\omega))}{|\zeta-\omega|^{n+2s}}d\upsilon_{g_{\mathbb{S}^n}}^{(\omega)}, ~\zeta\in \mathbb{S}^n, u\in C^{\infty}(\mathbb{S}^n),
\end{equation*}
where $\theta(\zeta,\omega)$ denotes the symmetric value of $\omega$ on $\mathbb{S}^n$ with respect to $\zeta$.
\end{proposition}

   The above proposition allows us to give an analogous expression for $\|\cdot\|_{s,g_{\mathbb{S}^n}}$ in terms of a singular integral operator:
 \begin{proposition}
 \label{gp6.1}
 If $s\in (0,1)$, then
 \begin{equation*}
 \|u\|_{s,g_{\mathbb{S}^n}} = \frac{c_{n,-s}}{2}\int_{\mathbb{S}^n}\int_{\mathbb{S}^n}\frac{[u(\zeta)-u(\omega)]^2}{|\zeta-\omega|^{n+2s}}d\upsilon_{g_{\mathbb{S}^n}}^{(\zeta)}d\upsilon_{g_{\mathbb{S}^n}}^{(\omega)} + P_s(1)\int_{\mathbb{S}^n}u^2 d\upsilon_{g_{\mathbb{S}^n}}
 \end{equation*}
 for all $u\in H^s(\mathbb{S},g_{\mathbb{S}^n})$.
 \end{proposition}
\begin{proof}
Denote by $\theta(\zeta,\omega)$ the symmetric value of $\omega$ on $\mathbb{S}^n$ with respect to $\zeta\in \mathbb{S}^n$. Given $u\in C^2(\mathbb{S}^n)$,  from Proposition \ref{gap1.1} and the Fubini theorem, we have
\begin{equation*}
\begin{aligned}
& \int_{\mathbb{S}^n}P.V.\int_{\mathbb{S}^n}\frac{u(\zeta)^2 - u(\zeta)u(\omega)}{|\zeta-\omega|^{n+2s}}d\upsilon_{g_{\mathbb{S}^n}}^{(\omega)}d\upsilon_{g_{\mathbb{S}^n}}^{(\zeta)}\\
& = \frac{1}{2}\int_{\mathbb{S}^n}\int_{\mathbb{S}^n}\frac{2u(\zeta)^2 - u(\zeta)u(\omega)-u(\zeta)u(\theta(\zeta,\omega))}{|\zeta-\omega|^{n+2s}}d\upsilon_{g_{\mathbb{S}^n}}^{(\omega)}d\upsilon_{g_{\mathbb{S}^n}}^{(\zeta)}\\
& = \frac{1}{2}\int_{\mathbb{S}^n}\int_{\mathbb{S}^n}\frac{[u(\zeta)-u(\omega)]^2}{|\zeta-\omega|^{n+2s}}d\upsilon_{g_{\mathbb{S}^n}}^{(\zeta)}d\upsilon_{g_{\mathbb{S}^n}}^{(\omega)} + \frac{1}{2}\frac{[u(\zeta)-u(\theta(\zeta,\omega))]^2}{|\zeta-\theta(\zeta,\omega)|^{n+2s}}d\upsilon_{g_{\mathbb{S}^n}}^{(\zeta)}d\upsilon_{g_{\mathbb{S}^n}}^{(\omega)}\\
& - \frac{1}{2}\int_{\mathbb{S}^n}\int_{\mathbb{S}^n}\frac{u(\omega)^2 -u(\omega)u(\zeta) + u(\theta(\zeta, \omega))^2 - u(\theta(\zeta,\omega))u(\zeta)}{|\zeta-\omega|^{n+2s}}d\upsilon_{g_{\mathbb{S}^n}}^{(\zeta)}d\upsilon_{g_{\mathbb{S}^n}}^{(\omega)}\\
& = \int_{\mathbb{S}^n}\int_{\mathbb{S}^n}\frac{[u(\zeta)-u(\omega)]^2}{|\zeta-\omega|^{n+2s}}d\upsilon_{g_{\mathbb{S}^n}}^{(\zeta)}d\upsilon_{g_{\mathbb{S}^n}}^{(\omega)} - \frac{1}{2} \int_{\mathbb{S}^n}P.V.\int_{\mathbb{S}^n}\frac{u(\omega)^2 - u(\omega)u(\zeta)}{|\zeta-\omega|^{n+2s}}d\upsilon_{g_{\mathbb{S}^n}}^{(\zeta)}d\upsilon_{g_{\mathbb{S}^n}}^{(\omega)}\\
& -  \frac{1}{2} \int_{\mathbb{S}^n}P.V.\int_{\mathbb{S}^n}\frac{u(\theta(\zeta,\omega))^2 - u(\theta(\zeta,\omega))u(\zeta)}{|\zeta-\theta(\zeta,\omega)|^{n+2s}}d\upsilon_{g_{\mathbb{S}^n}}^{(\zeta)}d\upsilon_{g_{\mathbb{S}^n}}^{(\omega)}\\
& =   \int_{\mathbb{S}^n}\int_{\mathbb{S}^n}\frac{[u(\zeta)-u(\omega)]^2}{|\zeta-\omega|^{n+2s}}d\upsilon_{g_{\mathbb{S}^n}}^{(\zeta)}d\upsilon_{g_{\mathbb{S}^n}}^{(\omega)} - \int_{\mathbb{S}^n}P.V.\int_{\mathbb{S}^n}\frac{u(\omega)^2 - u(\omega)u(\zeta)}{|\zeta-\omega|^{n+2s}}d\upsilon_{g_{\mathbb{S}^n}}^{(\zeta)}d\upsilon_{g_{\mathbb{S}^n}}^{(\omega)}
\end{aligned}
\end{equation*}
Therefore, the proposition follows immediately from the above equality and (\ref{g6.15}).
\end{proof}	
	
	The Proposition \ref{gp6.1} is needed to show the existence of eigenvalues for the operator $P^g_s - R^g_s$, which is described as follows:
\begin{lemma}
\label{gl6.1}
If $s\in(0,1)$, then the operator $P^g_s-R^g_s$ has a positive eigenvalue $\lambda_{1,s,g}$ such that \eqref{g1.40} holds.
\end{lemma}  
\begin{proof}
Assume that $g=\varphi^{4/(n-2s)}g_{\mathbb{S}^n}$ with $0<\varphi\in C^{\infty}(\mathbb{S}^n)$. Following the proof of Proposition \ref{gp6.1}, we define
\begin{equation*}
\begin{aligned}
I(u) & :=\int_{\mathbb{S}^n}(u-\overline{u})(P^g_s-R^g_s)(u-\overline{u})d\upsilon_g\\
	&=\frac{c_{n,-s}}{2}\int_{\mathbb{S}^n}\int_{\mathbb{S}^n}\frac{\varphi(\zeta)\varphi(\omega)(u(\zeta)-u(\omega))^2}{|\zeta-\omega|^{n+2s}}d\upsilon^{(\zeta)}_{g_{\mathbb{S}^n}}d\upsilon^{(\omega)}_{g_{\mathbb{S}^n}}, ~u\in H^s(\mathbb{S}^n,g),
\end{aligned}
\end{equation*}
where $\overline{u}= \avint u$. From the proof of \cite[Proposition A.1]{SV} and the Sobolev embeddings (see \cite[p.1587]{JLX2}), there exists a $u^*\in H^s(\mathbb{S}^n,g)$ such that
\begin{equation*}
0<I(u^*)=\min_{u\in H^s(\mathbb{S}^n,g)}I(u),
\end{equation*}
and 
\begin{equation*}
\langle u^*- \overline{u^*},w \rangle_{s,g} = 2I(u^*) \int_{\mathbb{S}^n}(u^*-\overline{u^*})w~ d\upsilon_g \text{ for all } w\in H^s(\mathbb{S}^n,g).
\end{equation*}
Therefore, $u^*-\overline{u^*}$ is an eigenvector corresponding to the eigenvalue $\lambda_{1,s,g}=2I(u^*)$. 
\end{proof}

One can follow the arguments from \cite{SV} and obtain a sequence $\{\lambda_{i,s,g}\}_{i\in \mathbb{N}}$ in $\mathbb{R}$ such that each $\lambda_{i,s,g}$ is an eigenvalue of $P^g_s - R^g_s$,
\begin{equation*}
0<\lambda_{1,s,g}<\lambda_{2,s,g}\leq \lambda_{3,s,g}\leq...\leq \lambda_{i,s,g}\leq \lambda_{i+1,s,g} ...,
\end{equation*}
and
\begin{equation*}
\lim_{i\rightarrow \infty}\lambda_{i,s,g} = +\infty. 
\end{equation*}

\subsection{Schauder type estimates}\

In this subsection we shall prove some local Schauder estimates for solutions of nonlocal equations on $\mathbb{R}^n$. Other Schauder type estimates can be found in \cite{JLX1, JLX, Silv}.

\begin{theorem}
\label{ath2.1}
Let $v=\mathcal{I}_{2s}(w)$ be the Riesz potential of $w \in L^{\infty}(\mathbb{R}^n)$ and $s\in (0,n/2)$. Assume $w\in L^1(\mathbb{R}^n)\cup C^{0,\alpha}(\mathbb{R}^n)$, $v\in L^{\infty}(\mathbb{R}^n)$, $\alpha \in (0,1)$ and $s=m+\gamma$ with $m\in \mathbb{N}$ and $\gamma \in [0,1)$.
\begin{enumerate}
\item[(i)] If $2\gamma + \alpha \notin \mathbb{N}$ and  $2\gamma +\alpha<1$, then $v\in C^{2m, 2\gamma + \alpha}(\mathbb{R}^n)$ and 
\begin{equation*}
\label{ape2.4}
				\|v\|_{C^{2m, 2\gamma + \alpha}(\mathbb{R}^n)}\leq C(n,s,\alpha) ( \|v\|_{L^{\infty}(\mathbb{R}^n)}+\|w\|_{C^{0,\alpha}(\mathbb{R}^n)}).
\end{equation*}
\item[(ii)] If $2\gamma + \alpha \notin \mathbb{N}$ and $2\gamma + \alpha >1$, then $v\in C^{2m+1, 2\gamma + \alpha -1}(\mathbb{R}^n)$ and
\begin{equation*}
\label{ape2.5}
\|v\|_{C^{2m+1, 2\gamma + \alpha-1}(\mathbb{R}^n)}\leq C(n,s,\alpha) ( \|v\|_{L^{\infty}(\mathbb{R}^n)}+\|w\|_{C^{0,\alpha}(\mathbb{R}^n)}).
\end{equation*}
\end{enumerate}
\end{theorem}
\begin{proof}
We will show that $v$ has the corresponding regularity in a neighborhood of the origin. The same argument works for a neighborhood of every point; so we get, respectively, that $v\in C^{2m, 2\gamma + \alpha}(\mathbb{R}^n)$ or $v\in C^{2m+1, 2\gamma + \alpha -1}(\mathbb{R}^n)$. Let be a smooth cutoff function such that $\eta(x)\in [0,1]$ for each $x\in \mathbb{R}^n$, $\text{supp}(v)\subset B_1$, and $\eta=1$ in $B_{7/8}$.
We can write $v$ as
\begin{equation*}
\begin{aligned}
v(x) & =c_{n,s} \int_{\mathbb{R}^n}\frac{w(y)}{|x-y|^{n-2s}}dy = c_{n,s}\int_{\mathbb{R}^n}\frac{(1-\eta(y))w(y)}{|x-y|^{n-2s}}dy + c_{n,s}\int_{\mathbb{R}^n}\frac{\eta(y)w(y)}{|x-y|^{n-2s}}dy \\
&:=v_1(x)+v_2(x) \text{ a.e. in } \mathbb{R}^n.
\end{aligned}
\end{equation*}
It is clear that $v_1\in C^{\infty}(\overline{B_{1/2}})$. By interpolation inequalities we have
\begin{equation*}
\|v_1\|_{C^{k, \beta}(\overline{B_{1/2}})}\leq C(n,s,k,\beta)(\|D^{2k}v_1\|_{L^{\infty}(\overline{B_{1/2}})}+\|v_1\|_{L^{\infty}(\overline{B_{1/2}})}),
\end{equation*}
where $k=2m $ if $0<\beta=2\gamma +\alpha<1$,  or $k=2m+1$ if $0<\beta=2\gamma+\alpha-1<1$. Then we see that
  \begin{align}
  \label{ape2.8}
  \|v_1\|_{C^{k, \beta}(\overline{B_{1/2}})} \leq C (\|w\|_{L^{\infty}(\mathbb{R}^n)} + \|v-v_2\|_{L^{\infty}(\overline{B_{1/2}})}).
  \end{align}
On the other hand, it follows from the conditions on $2\gamma +\alpha$ and the standard Riesz potential theory that
\begin{equation}
\label{ape2.9}
\|v_2\|_{C^{k, \beta}(\overline{B_{1/2}})}\leq C(\|v_2\|_{L^{\infty}(\mathbb{R}^n)}+\|\eta w\|_{C^{\alpha}(\mathbb{R}^n)})\leq C\| w\|_{C^{\alpha}(\overline{B_1})}.
\end{equation}
Therefore, the theorem follows from \eqref{ape2.8} and \eqref{ape2.9}. 
\end{proof}

\begin{theorem}
\label{ath2.2}
Let $ v:=\mathcal{I}_{2s}(w)$  be the Riesz potential of $w\in L^{\infty}(\mathbb{R}^n)$ and $s\in (0,n/2)$. Assume $w\in L^1(\mathbb{R}^n)$, $v\in L^{\infty}(\mathbb{R}^n)$ and $s=m+\gamma$ with $m\in \mathbb{N}$ and $\gamma \in [0,1)$.
\begin{enumerate}
\item[(i)] If $\gamma=0$, then $v\in C^{2m-1,\alpha}(\mathbb{R}^n)$ for any $\alpha\in (0,1)$, and 
\begin{equation*}
				\|v\|_{C^{2m-1,\alpha}(\mathbb{R}^n)}\leq C(n,s,\alpha) ( \|v\|_{L^{\infty}(\mathbb{R}^n)}+\|w\|_{L^{\infty}(\mathbb{R}^n)}).
\end{equation*}
\item[(ii)] If $0<2\gamma\leq 1$, then $v\in C^{2m,\alpha}(\mathbb{R}^n)$ for any $\alpha\in (0,2\gamma)$, and 
\begin{equation*}
				\|v\|_{C^{2m,\alpha}(\mathbb{R}^n)}\leq C(n,s,\alpha) ( \|v\|_{L^{\infty}(\mathbb{R}^n)}+\|w\|_{L^{\infty}(\mathbb{R}^n)}).
\end{equation*}
\item[(iii)] If $2\gamma>1$, then $v\in C^{2m+1,\alpha}(\mathbb{R}^n)$ for any $\alpha\in (0,2\gamma-1)$, and 
\begin{equation*}
				\|v\|_{C^{2m+1,\alpha}(\mathbb{R}^n)}\leq C(n,s,\alpha) ( \|v\|_{L^{\infty}(\mathbb{R}^n)}+\|w\|_{L^{\infty}(\mathbb{R}^n)}).
\end{equation*}
\end{enumerate}
\end{theorem}
\begin{proof}
Following the proof of Theorem \ref{ath2.1}, it is sufficient to estimate $v_2$. Denote $K(x)=|x|^{2s-n}$ for all $x\in \mathbb{R}^n\backslash\{0\}$. By the intermediate value theorem, on the line from $x$ to $y$ there exists
some $x_0$ with 
\begin{equation*}
\left|D^{2k}K(x-z)-D^{2k}K(y-z)\right| \leq C\frac{|x-y|}{|x_0-z|^{n-2s+2k+1}}, ~x,y\in B_{1/2}.
\end{equation*}
We put $\delta := 2|x-y|$. Then we have
\begin{equation*}
\begin{aligned}
|D^{2k}v_2(x)-D^{2k}v_2(y)| \leq & C\int_{B_3} \left|D^{2k}K(x-z)-D^{2k}K(y-z)\right|\eta(z)|w(z)|dz \\
\leq & C  \int_{B_{\delta}(x_0)}\left|D^{2k}K(x-z)-D^{2k}K(y-z)\right|\eta(z)|w(z)|dz \\
& +  \int_{B_3\backslash B_{\delta}(x_0)}\left|D^{2k}K(x-z)-D^{2k}K(y-z)\right|\eta(z)|w(z)|dz\\
\leq & C \|w\|_{L^{\infty}(\mathbb{R}^n)}\int_{B_{\delta}(x_0)}\frac{1}{|x-z|^{n-2s+2k}}dz \\
& + C  \|w\|_{L^{\infty}(\mathbb{R}^n)}\int_{B_3\backslash B_{\delta}(x_0)}\frac{\delta}{|x_0-z|^{n-2s+2k+1}}dz
\end{aligned}
\end{equation*}
Taking
\begin{equation*}
2k=\begin{cases}
	2m-1, &\text{ if } \gamma=0,\\
	2m, &\text{ if } 0<2\gamma\leq 1,\\
	2m + 1, & \text{ if } 2\gamma> 1, 
	\end{cases}
\end{equation*}
we have
\begin{equation*}
|D^{2k}v_2(x)-D^{2k}v_2(y)|\leq C \delta^{\alpha}.
\end{equation*}
This proves the theorem, because obviously we also have
\begin{equation*}
\|v_2\|_{C^j(B_{1/2})}\leq C \|w\|_{L^{\infty}(\mathbb{R}^n)}, \text{ for } j=0,1,..., 2k.
\end{equation*}
\end{proof}

\bibliographystyle{abbrv}
\bibliography{bibliography}

\end{document}